\documentclass{amsart}

\usepackage{amsmath,url,graphicx,amssymb,enumerate,stmaryrd, amsthm}
\usepackage[pagebackref,colorlinks,citecolor=blue,linkcolor=blue,urlcolor=blue,filecolor=blue]{hyperref}
\usepackage{microtype}
\usepackage[all]{xy}
\usepackage{dsfont}
\usepackage{aliascnt}
\usepackage{pinlabel}
\usepackage{etoolbox}
\usepackage{tikz} \usetikzlibrary{matrix,arrows}
\usepackage[hmargin=3cm,vmargin=3cm]{geometry}

\newcommand{\m}{\to}

\providecommand{\fS}{\mathfrak S}

\providecommand{\Bun}{\ensuremath\mathrm{Bun}}

\providecommand{\PBr}{\ensuremath\mathrm{PBr}}
\providecommand{\Br}{\ensuremath\mathrm{Br}}
\providecommand{\Bur}{\ensuremath\mathrm{Bur}}
\providecommand{\Ab}{\ensuremath\mathsf{Ab}}
\providecommand{\FI}{\ensuremath\mathsf{FI}}
\providecommand{\VIC}{\ensuremath\mathsf{VIC}}

\providecommand{\SI}{\ensuremath\mathsf{SI}}

\providecommand{\C}{\ensuremath\mathcal{C}}
\providecommand{\G}{\ensuremath\mathcal{G}}
\providecommand{\UG}{\ensuremath U\G}

\providecommand{\hooklongrightarrow}{\lhook\joinrel\longrightarrow}

\providecommand{\inject}{\hooklongrightarrow}

\providecommand{\N}{\ensuremath\mathbb N_0}
\providecommand{\Z}{\ensuremath\mathbb Z}
\providecommand{\uZ}{ {\underline{ \mathbb Z}}}
\providecommand{\Q}{\ensuremath\mathbb Q}
\providecommand{\R}{\ensuremath\mathbb R}
\providecommand{\F}{\ensuremath\mathbb F}
\providecommand{\bR}{\ensuremath\mathbf R}
\providecommand{\barR}{\overline{\ensuremath\mathbf R}}
\providecommand{\bk}{\mathbb{K}}

\providecommand{\cM}{\ensuremath\mathcal M}
\providecommand{\cQ}{\ensuremath\mathcal Q}
\providecommand{\cN}{\ensuremath\mathcal N}
\providecommand{\cG}{\ensuremath\mathcal G}
\providecommand{\cL}{\ensuremath\mathcal L}

\DeclareMathOperator{\coker}{coker}

\DeclareMathOperator{\cone}{cone}

\DeclareMathOperator{\Map}{Map}

\providecommand{\id}{\ensuremath\mathrm{id}}
\providecommand{\SSt}{\St^{E_1}}

\DeclareMathOperator{\Aut}{Aut}
\DeclareMathOperator{\hAut}{hAut}

\DeclareMathOperator{\Diff}{Diff}
\DeclareMathOperator{\Symp}{Symp}

\DeclareMathOperator{\GL}{GL}
\DeclareMathOperator{\colim}{colim}
\DeclareMathOperator{\Tor}{Tor}
\DeclareMathOperator{\Conf}{Conf}

\DeclareMathOperator{\Sp}{Sp}

\DeclareMathOperator{\Mod}{Mod}

\DeclareMathOperator{\Ind}{Ind}
\DeclareMathOperator{\Res}{Res}

\DeclareMathOperator{\St}{St}

\providecommand{\op}{\ensuremath\mathrm{op}}

\definecolor{grey}{gray}{.5}

\providecommand{\con}[1]{\textbf{\textup{#1}}}

\numberwithin{thmcounter}{section}
\newaliascnt{thmauto}{thmcounter}

\newaliascnt{Defauto}{thmcounter}

\newaliascnt{exauto}{thmcounter}

\newaliascnt{lemauto}{thmcounter}

\newaliascnt{propauto}{thmcounter}

\newaliascnt{corauto}{thmcounter}

\newaliascnt{remauto}{thmcounter}

\newaliascnt{convauto}{thmcounter}

\newtheorem{atheorem}{Theorem}

\newtheorem*{ThmA'}{Theorem A'}
\newtheorem*{ThmB'}{Theorem B'}
\newtheorem*{ThmC'}{Theorem C'}

\newtheorem{theorem}[thmauto]{Theorem}
\newtheorem{lemma}[lemauto]{Lemma}
\newtheorem{proposition}[propauto]{Proposition}
\newtheorem{corollary}[corauto]{Corollary}
\theoremstyle{definition}
\newtheorem{definition}[Defauto]{Definition}
\newtheorem{example}[exauto]{Example}
\newtheorem{remark}[remauto]{Remark}

\title{Representation stability, secondary stability, and polynomial functors}
\author{Jeremy Miller}\thanks{Jeremy Miller was supported in part by NSF grant DMS-1709726 and a Simons Foundation Collaboration Grant}
\address{Department of Mathematics, Purdue University, USA}
\email{jeremykmiller@purdue.edu}
\author{Peter Patzt}
\address{University of Copenhagen, Centre for Geometry and Topology, Universitatsparken 5, 2100 Copenhagen, Denmark} \address{University of Oklahoma, Department of Mathematics, 601 Elm Avenue, Norman, OK, 73019, USA} \email{ppatzt@ou.edu}
\thanks{Peter Patzt was supported by the Danish National Research Foundation through the Copenhagen Centre for Geometry and Topology (DNRF151) and the European Research Council under the European Union’s Seventh Framework Programme ERC Grant agreement ERC StG 716424 - CASe, PI Karim Adiprasito.}

\author{Dan Petersen}
\address{Department of Mathematics, Stockholm University, SE}
\email{dan.petersen@math.su.se}
\thanks{Dan Petersen was supported in part by ERC-2017-STG 759082 and by a Wallenberg Academy Fellowship}

 \date{\today}

\begin{document}

\begin{abstract}
We prove a general representation stability result for polynomial coefficient systems which lets us prove representation stability and secondary homological stability for many families of groups with polynomial coefficients. This gives two generalizations of classical homological stability theorems with twisted coefficients.  We apply our results to prove homological stability for hyperelliptic mapping class groups with twisted coefficients, prove new representation stability results for congruence subgroups, establish secondary homological stability for groups of diffeomorphisms of surfaces viewed as discrete groups, and improve the known stable range for homological stability for general linear groups of the sphere spectrum.

\end{abstract}
\maketitle

\tableofcontents

%%%%%%%%%%%%% SECTION %%%%%%%%%%%%%%%%%%%%

\section{Introduction}

\subsection{Homological stability with polynomial coefficients}
Consider a family of groups (or spaces) and maps between them: \[ G_0 \m G_1 \m G_2\m\cdots\] For example, $G_n$ could be the $n$th symmetric group or general linear group. Such a sequence of groups are said to exhibit \emph{homological stability} if the map $H_i(G_n) \m H_i(G_{n+1})$ is an isomorphism for $n$ large in comparison to $i$ ($n \gg i$). Homological stability is a ubiquitous phenomena and has been instrumental in the study of group cohomology, moduli spaces, and $K$-theory. Since the early days of homological stability, it was noticed that it is not only desirable to know that the homology stabilizes with trivial coefficients but also important to know that the homology stabilizes with certain families of twisted coefficient systems. For example, Dywer \cite{DwyerTwisted} used twisted homological stability for general linear groups to prove finiteness results for $A$-theory. 

Even if a sequence of groups $\{G_n\}_n$ have homological stability, it is not reasonable to expect that the homology will stabilize with arbitrary twisted coefficients $\{A_n\}_n$. In particular, there must be some compatibility between the coefficients for different $n$.  The usual condition on coefficients to ensure stability is called \emph{polynomiality}. See \autoref{defpoly} for the definition we use which agrees with that of Randal-Williams--Wahl \cite{RWW}, generalizing Dwyer's work for general linear groups \cite{DwyerTwisted}. In addition to implying that $H_i(G_n;A_n)$ is independent of $n$ for $n \gg i$ \cite{RWW}, this polynomiality condition implies that the groups $A_n$ have ranks that grow at most polynomially (or are infinite).

There are several recent generalizations of homological stability such as representation stability and secondary stability. The goal of this paper is to develop a set of tools which will allow us to prove representation stability and secondary homological stability theorems for the homology of families of groups with polynomial coefficients.

\subsection{Representation stability with polynomial coefficients}

There are many natural families of groups that do not exhibit homological stability; for example, the first homology of the pure braid group is given by $H_1(\PBr_n) \cong \Z^{\binom n2}$. In this case, the homology groups carry natural symmetric group actions which control the growth. In this and other examples, the homology satisfies \emph{representation stability} in the sense of Church--Farb \cite{CF} and Church--Ellenberg--Farb \cite{CEF}, which is an equivariant generalization of homological stability. There are many related definitions of representation stability with the most basic being \emph{finite generation degree}. Consider a sequence of groups \[Q_0 \m Q_1 \m \dots, \] and a sequence of $\Z Q_n$-modules $A_n$ with $Q_n$-equivariant maps $A_n \m A_{n+1}$. A sequence $\{A_n\}_n$ is said to have generation degree $\leq d$ if \[\Ind_{Q_n}^{Q_{n+1}}A_{n} \longrightarrow A_{n+1}\] is surjective for all $n \geq d$. In the previous example, $A_n$ would be $H_1(\PBr_n)$ and $Q_n$ would be the symmetric group on $n$ letters. Then $A$ has generation degree $\le 2$. Finite generation degree is an equivariant analogue of the statement that the stabilization map is surjective in a stable range. There are also equivariant analogues of the range where the stabilization map is an isomorphism such as \emph{presentation degree}. This is described in \autoref{SecC} and \autoref{defPresentation}. Before we can state our main representation stability theorem (see \autoref{main}), we need to review several notions used to state the theorem.

Usually, groups $\{N_n\}_n$ whose homology exhibits representation stability appear in short exact sequences
\[ 1 \longrightarrow N_n \longrightarrow G_n \longrightarrow Q_n \longrightarrow 1\]
with families $\{G_n\}_n$ and $\{Q_n\}_n$ satisfying homological stability. Such a short exact sequence gives a natural action of $Q_n$ on $H_i(N_n; A_n)$ for any $\Z G_n$-module $A_n$. Examples of such short exact sequences include the pure braid groups together with the braid groups and the symmetric groups as well as congruence subgroups together with general linear groups over the integers and over finite fields. In the construction used in this paper, we in fact need the three sequences $\mathcal N = \{N_n\}_n$, $\G = \{G_n\}_n$, and $\cQ = \{Q_n\}_n$ to form stability groupoids (which are monoidal groupoids with some extra conditions, see \autoref{Def:stability groupoid}) and we need the maps of groups to come from monoidal functors. We assemble all of this information by saying
\[  1 \longrightarrow \mathcal N  \longrightarrow \G \longrightarrow  \cQ \longrightarrow 1 \]
is a stability short exact sequence (see \autoref{Def:stability SES}).

Furthermore, we need a topological condition called \con{H3}  (see \autoref{Def:H3}) for braided stability groupoids $\G$ that was first introduced by \cite{RWW} to axiomatize homological stability arguments. It is also important in representation stability; see Putman--Sam \cite{PS}, \cite{MillerWilson2}, and \cite{Pa2}. There is a semisimplicial set $W_n^\G$ for every $n\in \N$ such that the set of $p$-simplices of $W_n^\G$ is given by $G_n/G_{n-p-1}$. \con{H3} is the condition that $W_n^\G$ is highly homologically connected in a range increasing linearly with $n$. 

%This semisimplicial set can by constructed using the hom-sets in the stability category $U\G$ of a braided stability groupoid $\G$ (see \autoref{Def:stability category}). 

Associated to each stability groupoid $\G$ is a category called $U\G$ (see \autoref{Def:stability category}) which can be used to formulate notions of representation stability for sequences of $G_n$-representations. For $G_n$ the symmetric group, $U\G$ is equivalent to the category of finite sets and injections $\FI$ studied by Church--Ellenberg--Farb \cite{CEF}. A $U\G$-\emph{module} is a functor from $U\G$ to the category of abelian groups and encodes a sequence of $G_n$-representations $A_n$ and equivariant maps $A_n \m A_{n+1}$. To state our representation stability result, we need one more condition called \emph{degree-wise coherence}. We say that a stability category $U\G$ is degree-wise coherent if for all $U\G$-submodules $A \subseteq B$, if $B$ is presented in finite degree and $A$ is generated in finite degree, then $A$ is presented in finite degree. This algebraic condition is a weakening of regularity in the sense of Church--Ellenberg \cite{CE}. We can now state our main theorem on representation stability.

\begin{atheorem} \label{main}
Let
\[  1 \longrightarrow \mathcal N  \longrightarrow \G \longrightarrow  \cQ \longrightarrow 1 \]
be a stability short exact sequence. Assume that $\G$ and $\cQ$ are braided and satisfy \con{H3} and that $U\cQ$ is degree-wise coherent. Let $A$ be a $U\G$-module of finite polynomial degree. Then the sequence
\[ \{H_i(N_n; A_n)\}_n\]
 is presented (and hence also generated) in finite degree for every fixed $i\in \N$.
\end{atheorem}

See \autoref{CoherenceExamples} and  \autoref{RemCoh} for examples of groups where degree-wise coherence is known. The condition \con{H3} is known for basically all families of groups that are known to satisfy homological stability.  In \autoref{mainCentralVersion}, we give a quantitative version of this theorem.  This is a generalization of a result of Randal-Williams--Wahl \cite{RWW} which shows that if $\G$ satisfies $\con{H3}$ and $A$ has finite polynomial degree, then $H_i(G_n;A_n)$ stabilizes. One obtains this result from our result by specializing to the case that the groups $Q_n$ are all trivial. An application of \autoref{main} is representation stability for the homology of the pure braid group with polynomial coefficients, such as the Burau representation (see \autoref{exampleBurau}). Perhaps surprisingly, this was not known before, even though the homology of the pure braid group was one of the examples which started the whole subject of representation stability. 

\begin{remark}
In \cite{Kra}, Krannich generalized the main stability results of Randal-Williams--Wahl \cite{RWW} to apply to families of spaces that are not classifying spaces of discrete groups. We expect that a similar generalization of \autoref{main} should be possible.

\end{remark}

\subsection{Secondary homological stability with polynomial coefficients}

Secondary homological stability is a stability pattern for the unstable homology of groups or spaces exhibiting homological stability. The prototypical example of this phenomena is Galatius--Kupers--Randal-Williams' result that the relative homology of mapping class groups with one boundary component $H_{i}(\Mod_{g,1},\Mod_{g-1,1})$ stabilize as $g$ increases by $3$ and $i$ increases by $2$ \cite{GKRW2}. In this paper, we show that the techniques of Galatius--Kupers--Randal-Williams \cite{GKRW1,GKRW2,GKRW3} can be used to prove secondary homological stability with coefficients in polynomial $U\G$-modules as well as improved stable ranges in classical homological stability with polynomial coefficients. 

Let $\G$ be a stability groupoid. Homological stability can be rephrased as the statement that the groups $H_i(G_n,G_{n-1})$ vanish in a range. Secondary homological stability is the statement that these relative homology groups themselves stabilize with respect to degree shifting maps in an even larger range. A \emph{secondary stability map} of bidegree $(a,b)$ is a certain kind of map that induces a degree shifting map
\[H_{i-b}(G_{n-a},G_{n-a-1}) \m H_i(G_n,G_{n-1}).\]
See \autoref{secSec} for a description. We show that secondary stability for untwisted coefficients implies secondary stability with polynomial coefficients if certain simplicial complexes called splitting complexes are highly connected.

The $n$th \emph{splitting complex} of a monoidal groupoid $\G$ is a semi-simplicial set whose set of $p$-simplices are given by: \[\bigsqcup_{a_0+\dots + a_{p+1}=n} G_n/(G_{a_0} \times \dots \times G_{a_{p+1}}).\] We say that $\G$ satisfies the standard connectivity assumptions if for all $n$ the $n$th splitting complex is $(n-3)$-acyclic. We prove that the standard connectivity assumption combined with secondary stability with untwisted coefficients implies secondary stability with polynomial coefficients.

\begin{atheorem} \label{mainSecondary}
Let $A$ be a polynomial $U\G$-module of degree $\leq r$ in ranks $>d$. Let $\lambda \leq 1$ and $c \geq 0$. Assume $\G$ satisfies the standard connectivity assumptions and that there is a secondary stability map $f$ of bidegrees $(a,b)$ which induces a surjection 
\[f_*\colon H_{i-b}(G_{n-a},G_{n-a-1}) \m H_{i}(G_{n},G_{n-1})\]
 for $i \leq \lambda (n-c)$ and an isomorphism for $i \leq \lambda (n-c)-1$. Then 
 \[f_*\colon H_{i-b}(G_{n-p},G_{n-a-1};A_{n-a},A_{n-a-1}) \m H_{i}(G_{n},G_{n-1};A_{n},A_{n-1})\]
 is a surjection for $i \leq \lambda (n-c-\max(r,d))$ and an isomorphism for $i \leq \lambda (n-c-\max(r,d))-1$.
\end{atheorem}

The standard connectivity assumption is satisfied for all families of groups known to exhibit secondary homological stability or homological stability with stable range larger than slope $\frac 12$. It is one of the main technical conditions that allows the techniques of Galatius--Kupers--Randal-Williams to apply to a family of groups.  See \autoref{AclyList} for a list of some groups which are known to satisfy the standard connectivity assumption.

\subsection{Stability for polynomial coefficients}
\label{SecC} Let $\G$ be a stability groupoid.
The category of $\G$-modules has a monoidal structure called the induction or convolution product. Using this monodial product, one can define rings, modules, $\Tor$ groups, etc. There is a formulation of representation stability in terms of vanishing of certain $\Tor$ groups. In \autoref{SecRing}, we describe how to associate to a $U\G$-module $A$, a $\G$-module called $\Tor^\uZ_i(A,\Z)$ whose vanishing controls representation stability for $A$. Concretely, \[ \Tor^\uZ_0(A,\Z)_n=\coker \big(\Ind_{G_{n-1}}^{G_n} A_{n-1} \m A_n\big)\] and the higher $\Tor$ groups are the higher derived functors of this functor. The $\G$-module $\Tor^\uZ_0(A,\Z)$ can be thought of as $U\G$-module indecomposables or minimal generators of $A$. In particular, vanishing of $\Tor^\uZ_0(A,\Z)_n$ for $n>d$ is equivalent to $A$ having generation degree $\leq d$ and vanishing of $\Tor^\uZ_0(A,\Z)_n$ and $\Tor^\uZ_1(A,\Z)_n$ for $n>r$ is equivalent to $A$ having presentation degree $\leq r$.  We say that $A$ has \emph{derived representation stability} if $\Tor^\uZ_i(A,\Z)_n \cong 0$ for $n \gg i$. See \autoref{prop:res of finite type} for the relationship between these $\Tor$-groups and resolutions of $U\G$-modules. Derived representation stability is equivalent to $A$ having a free resolution with each syzygy generated in finite degree.

 Galatius--Kupers--Randal-Williams \cite[Remark 19.11]{GKRW1}  asked if polynomial $U\G$-modules exhibit derived representation stability. We answer this in the affirmative if $\G$ satisfies \con{H3} or the standard connectivity assumptions.

\begin{atheorem} \label{derivedRepPolynomials}
Assume $\G$ satisfies \con{H3} or the standard connectivity assumptions and let $A$ be a polynomial $U\G$-module. Then $A$ exhibits derived representation stability. 
\end{atheorem}

While \autoref{main} and \autoref{mainSecondary} are about stability properties for the homology of groups with polynomial coefficients, \autoref{derivedRepPolynomials} is about the polynomial coefficients themselves. \autoref{derivedRepPolynomials} says that the polynomial $U\G$-modules have representation stability under mild assumptions on $\G$. For example, the sequence of Burau representations have representation stability with respect to the action of the braid groups (see \autoref{exampleBurau}).

\begin{remark}
\autoref{derivedRepPolynomials} was previously known in the case that $\G$ is the groupoid of symmetric groups, general linear groups of $PID$s, or the symplectic groups of $PID$s \cite{Pa2,MPW}. Those techniques do not apply to the braid groups (see \cite[Example 7.11]{Pa2}), or more generally any case where the stability groupoid is nontrivially braided monoidal instead of symmetric monoidal. However, the techniques of this paper do apply to the braid groups, which is in particular used in our study of moduli spaces of hyperelliptic curves. Moreover, in \autoref{PolynomialsCentral}, we give a quantitative version of \autoref{derivedRepPolynomials} which improves the stable ranges given in \cite{MPW} in the case of general linear groups and the symplectic groups of $PID$s. We recently learned that Andrew Putman has independently proven a theorem similar to \autoref{PolynomialsCentral} in the case that $G_n$ is a general linear group or a symplectic group. This is used in his work on congruence subgroups of mapping class groups. We additionally prove  \autoref{PolynomialsTor} which further improves the stable range for \autoref{derivedRepPolynomials} for many families of groups. This improved range is crucial for our applications to secondary homoloigcal stability.

\end{remark}

\subsection{Applications}

We now describe a few applications of our general stability theorems.

\vspace{.1in}
\noindent \textbf{Moduli space of hyperelliptic curves with twisted coefficients:}

\noindent 
Let $\mathcal H_g$ be the moduli space of closed hyperelliptic curves (i.e.\ curves which admit a degree $2$ map onto $\mathbb P^1$), and $\mathcal M_g$ denote the moduli space of closed curves. We consider $\mathcal H_g$ as a closed substack of $\mathcal M_g$. Any representation of the symplectic group $\Sp_{2g}(\Z)$ defines a local coefficient system on $\mathcal H_g$ via the composition 
\[\pi_1(\mathcal H_g) \m \pi_1(\mathcal M_g) \m \Sp_{2g}(\Z).\]
 The algebraic representations of the symplectic group are indexed by partitions, and we let $V_\lambda$ be the local coefficient system on $\mathcal H_g$ associated to a partition $\lambda$. The following is a corollary of \autoref{main}.

\begin{theorem}\label{hyperelliptic-thm}
	Let $\{V_\lambda\}_g$ denote the collection of local systems on the moduli spaces of hyperelliptic curves $\mathcal H_g$ corresponding to the partition $\lambda$. There are isomorphisms $H^k(\mathcal H_g;V_\lambda) \cong H^k(\mathcal H_{g+1};V_{\lambda})$ for $g \gg k$ which are moreover compatible with the natural mixed Hodge structure on these cohomology groups and with the structure of $\ell$-adic Galois representation obtained after tensoring with $\Q_\ell$. 
\end{theorem}

It is perhaps a bit inaccurate to refer to \autoref{hyperelliptic-thm} as a homological stability theorem, as there are in fact no natural maps relating the moduli spaces $\mathcal H_g$ for different genera. Nevertheless, one reason to be interested in such a result is the following. By the Grothendieck--Lefschetz trace formula for stacks (Behrend \cite[Corollary 6.4.10]{behrend-derived}), and using the fact that $\mathcal H_g$ is the complement of a simple normal crossing divisor in a smooth proper stack over $\mathbb Z[\tfrac 12]$, there is an equality 
\begin{equation}
\label{jonas} q^{\dim \mathcal H_g}\sum_k (-1)^k \operatorname{Tr}(\mathrm{\Phi}_q \mid H^k(\mathcal H_{g,\overline \Q};V_\lambda \otimes \Q_\ell)) = \sum_{x \in \mathcal H_g(\mathbb F_q)} \frac 1 {\vert \mathrm{Aut}(x)\vert} \operatorname{Tr}(\mathrm{\Phi}_q \mid x^\ast V_\lambda \otimes \Q_\ell),
\end{equation} 
where $\Phi_q$ denotes the arithmetic Frobenius at the odd prime power $q$, $V_\lambda \otimes \Q_\ell$ denotes the lisse $\Q_\ell$-sheaf corresponding to the local system $V_\lambda$, and $x^\ast V_\lambda \otimes \Q_\ell$ denotes the stalk of this sheaf at (a geometric point over) $x$. For example, if $\lambda=0$, so $V_\lambda$ is the trivial local system, then the left hand side becomes the trace of Frobenius on the cohomology of $\mathcal H_g$ and the right hand side becomes the number of $\mathbb F_q$-points of $\mathcal H_g$ weighted by their automorphisms (which turns out to always equal $q^{2g-1}$). Note that the right hand side can be calculated for any given $q$ if one can make a list of all $\mathbb F_q$-isomorphism classes of hyperelliptic curves of given genus, the size of their automorphism groups, and their Frobenius eigenvalues (which determine the quantity $\operatorname{Tr}(\mathrm{\Phi}_q \mid x^\ast V_\lambda \otimes \Q_\ell)$). Bergstr\"om \cite{bergstrom09} studied the quantity \eqref{jonas} by very direct methods, using that all hyperelliptic curves admit an affine equation $y^2=f(x)$ with squarefree $f$ and summing over all $f$, and discovered in the process curious recursive formulas in the genus for the quantity \eqref{jonas}. A particular consequence of Bergstr\"om's recursions is that when $q$ is fixed and $g \to \infty$, the expression
\[\sum_k (-1)^k \operatorname{Tr}(\mathrm{\Phi}_q \mid H^k(\mathcal H_g;V_\lambda \otimes \Q_\ell))\]
converges exponentially fast to a power series in $q^{-1}$, which is in fact given by a rational function with all poles on the unit circle (and in particular it converges on the unit disk), depending only on $\lambda$. Moreover, Bergstr\"om gave an algorithmic procedure to compute this rational function for any $\lambda$, which has been carried out for all $\vert \lambda \vert \leq 30$ (pers.\ comm.); the answers are highly nontrivial. 

Given the above it is natural to expect that there should be homological stability for the spaces $\mathcal H_g$ with coefficients in $V_\lambda$, and that the rational functions calculated by Bergstr\"om are really giving the trace of Frobenius on the stable cohomology. A result of Randal-Williams--Wahl \cite[Theorem D]{RWW} very nearly verifies this expectation, except they deal with hyperelliptic surfaces with boundary (i.e.\ the braid group) instead of closed hyperelliptic surfaces. Our  \autoref{hyperelliptic-thm} fills in this gap and shows that one also has homological stability for closed surfaces, which in particular (combined with an easy bound on the unstable Betti numbers, e.g.\ the one obtained from the Fuks stratification of the configuration space of points in $\mathbb C$) implies that Bergstr\"om's calculations are indeed giving the trace of Frobenius on stable cohomology. 

\vspace{.1in}
\noindent \textbf{Congruence subgroups:}

\noindent 
In \cite{Pu}, Putman proved that the homology of congruence subgroups satisfy representation stability as $U\fS$-modules and asked if a similar statement was true with respect to the action of general linear groups. When the quotient ring is finite, this was resolved by Putman--Sam \cite[Theorem G]{PS}. In the case that the quotient ring is a PID, this was resolved by \cite[Theorem C]{MPW}. We extend the result, removing these assumptions on the quotient ring and improving all known stable ranges. 

Let $J \subset R$ be an ideal in a commutative ring. Let $\GL_n(J)$ denote the kernel of $\GL_n(R) \m \GL_n(R/J)$ and let $\GL_n^{\mathfrak U}(R/J)$ denote the group of matrices with determinant in the image of $R^\times \m R/J$. When the map $\GL_n(R) \m \GL^{\mathfrak U}_n(R/J)$ is surjective, the homology groups $\{H_i(\GL_n(J))\}_n$ assemble to form a $U\GL^{\mathfrak U}(R/J)$-module which we denote by $H_i(\GL(J))$. Recall that a $U\G$-module $A$ is presented in degree $\leq r$ if $$\Tor^\uZ_0(A,\Z)_n\cong \Tor^\uZ_1(A,\Z)_n \cong 0\text{ for all } n>r.$$ The following is an application of \autoref{derivedRepPolynomials}.

\begin{theorem} \label{thmCong} Let $R$ be a commutative ring and $J \subset R$ an ideal. Let $\mathfrak U$ be the units in $R/J$ which lift to units in $R$. Let $t$ be the stable rank of $R/J$ and $s$ the stable rank of $R$. If $\GL_n(R) \m \GL_n^{\mathfrak U}(R/J)$ is surjective for all $n$, then the $U\GL^{\mathfrak U}(R/J)$-module $H_i(\GL(J))$ has presentation degree $\leq  \max(8i+4s+t+8,4i+2s+2t-1)$.
%Then, the  has presentation degree $\leq   8i+4s+8$.
\end{theorem}

The range established here is roughly three times better than that of \cite[Theorem C]{MPW}. See Bass \cite[Section 4]{Bass} for a definition of stable rank.

\vspace{.1in}
\noindent \textbf{Diffeomorphism groups:}

\noindent 
Let $\Diff^\delta(M)$ denote the group of $C^\infty$ diffeomorphisms of a manifold which fix a neighborhood of the boundary point-wise, topologized with the discrete topology. The cohomology groups of $B\Diff^\delta(M)$ are characteristic classes of flat bundles with fiber $M$ and have many applications to foliation theory and realization problems. Let $M_{g,1}$ denote an orientable surface with one boundary component. Nariman  \cite{Nar} proved that the groups $\Diff^\delta(M_{g,1})$ have homological stability. We prove that they also exhibit secondary homological stability. The following is an application of \autoref{mainSecondary}. 

\begin{theorem} \label{DiffSec}
There is a map 
\[H_{i-2}(B\Diff^\delta(M_{g-3,1}),B\Diff^\delta(M_{g-4,1});\Z[\textstyle\frac{1}{10}] ) \m H_{i}(B\Diff^\delta(M_{g,1}),B\Diff^\delta(M_{g-1,1});\Z[\textstyle\frac{1}{10}] )\]
 which is a surjection for $i <
\frac{3}{4}g$ and an isomorphism for $i < \frac{3}{4}g-1$.
\end{theorem}

\vspace{.1in}
\noindent \textbf{Homotopy automorphisms and general linear groups of the sphere spectrum:}

%In addition to proving a general secondary stability theorem for polynomial coefficients, we also prove a general theorem for improved stable ranges for (primary) homological stability with polynomial coefficients; see \autoref{improvedrangeGeneral}. We apply this result to study the general linear groups of the sphere spectrum. For $X$ a based space, let $\hAut(X)$ denote topological group of based homotopy automorphisms. Let $S^d$ denote the $d$-dimensional sphere and let $B$ denote the classifying space functor. Our model for the general linear group of the sphere spectrum is the $\Omega$-spectrum of the infinite loop space
%\[ \colim_{d\to \infty} \hAut(\bigvee_n S^d).\]

%\begin{theorem} \label{HAutThm}The natural map $H_i( B\GL_{n-1}(\mathbb S); \Z[\textstyle\frac 12] ) \m  H_i( B\GL_n(\mathbb S); \Z[\textstyle\frac 12]) $ is surjective for $i < \frac{2}{3}n$ and an isomorphism for $i \leq \frac{2}{3}n-1$.
%\end{theorem}

%Although not explicitly stated, the methods of Dwyer \cite[Section 4]{DwyerTwisted} are sufficient to establish a slope $\frac 12$ stable range. One reason to care about these spaces is their connection to $A$-theory. They play the same role in $A$-theory as classical general linear groups play in $K$-theory.
%\[\colim_{d \m \infty} H_i(B\hAut(\bigvee_{n} S^d)) \cong H_i(B\GL_n(\mathbb S))\] with $\mathbb S$ the sphere spectrum. Thus the homology of these spectra also stabilize.

\noindent In addition to proving a general secondary stability theorem for polynomial coefficients, we also prove a general theorem for improved stable ranges for primary homological stability with polynomial coefficients; see \autoref{improvedrangeGeneral}. We apply this result to study homotopy automorphism monoids. For $X$ a based space, let $\hAut(X)$ denote the topological monoid of based homotopy automorphisms topologized with the compact open topology. Let $B$ denote the bar construction for topological monoids. The space $B\hAut(X)$ can be viewed as the moduli space of spaces homotopy equivalent to $X$ with a choice of marked point. Let $S^d$ denote the $d$-dimensional sphere. We prove the following stability result.

\begin{theorem} \label{HAutThm} For $d \geq 3$, the natural map $H_i(B\hAut(\bigvee_{n-1} S^d) ; \Z[\frac 12] ) \m  H_i(B\hAut(\bigvee_{n} S^d); \Z[\frac 12]) $ is surjective for $i \leq \frac{2}{3}n$ and an isomorphism for $i \leq \frac{2}{3}n-1$.
\end{theorem}

One model of the $n$th general linear group of the sphere spectrum  is 
\[\GL_{n}(\mathbb S) := \underset{d\to \infty}{\colim} \hAut(\bigvee_n S^d).\] The following is a corollary of \autoref{HAutThm}.

\begin{corollary}\label{GLS cor}Let $\mathbb S$ denote the sphere spectrum. The natural map $H_i(B  \GL_{n-1}(\mathbb S); \Z[\textstyle\frac 12] ) \m  H_i(B  \GL_n(\mathbb S); \Z[\textstyle\frac 12]) $ is surjective for $i < \frac{2}{3}n$ and an isomorphism for $i \leq \frac{2}{3}n-1$.
\end{corollary}

One reason to care about general linear groups of ring spectra is their connection to $A$-theory. They play a similar role in $A$-theory as classical general linear groups play in $K$-theory. Although not explicitly stated, the methods of Dwyer \cite[Section 4]{DwyerTwisted} are sufficient to establish versions of these theorems with a slope $\frac 12$ stable range.

%One reason to care about these spaces is that \[\colim_{d \m \infty} H_i(B\hAut(\bigvee_{n} S^d)) \cong H_i(B\GL_n(\mathbb S))\] with $\mathbb S$ the sphere spectrum. Thus the homology of these spectra also stabilize.

\subsection{Outline of the paper}
In \autoref{Sec2}, we describe the categorical setup for our stability results. In \autoref{SecPolyn}, we prove \autoref{derivedRepPolynomials} which states that polynomial coefficient systems exhibit derived representation stability. We use this result in \autoref{Sec4} to prove our representation stability and secondary stability theorems with polynomial coefficients, \autoref{main} and \autoref{mainSecondary}. Finally, in  \autoref{Sec5}, we apply these general stability theorems to concrete examples and prove \autoref{hyperelliptic-thm} \autoref{thmCong}, \autoref{DiffSec}, and \autoref{HAutThm}.

In \autoref{appendix}, we give a summary of various kinds of stability arguments and give motivation for how one should think about the techniques of this paper. Readers interested in a qualitative and big picture view of this paper, should start with the appendix. Those who are instead only interested in rigorous proofs and precise statements should ignore the appendix.

\subsection{Acknowledgments}

We thank Zachary Himes, Manuel Krannich, Alexander Kupers, Rohit Nagpal, Sam Nariman, Andrew Putman, Oscar Randal-Williams, and Robin Sroka for helpful conversations.

\section{Categorical and algebraic preliminaries}
\label{Sec2}
In this section, we review the categorical framework for our stability results. Much of this setup has appeared in or is inspired by other papers such as \cite{DsurPoly,SSIntro,PS,RWW,Pa2,HepEdge,GKRW1}.

\subsection{Tensor products as coends}

We will describe a certain coend construction in this section, which can be thought of as the tensor product over a category. Let $\C$ be a small category. We write $c\in \C$ when $c$ is an object of $\C$ and $\C(c,c')$ for the set of morphisms from $c$ to $c'$ in $\C$. We call a functor from $\C$ to the category of sets a $\C$-set and a functor from $\C$ to modules over a commutative ring $\bk$ a $\C$-module. A morphism of $\C$-sets or $\C$-modules is a natural transformation. In later sections we will freely use the analogous terminology also for $\mathcal C$-chain complexes, $\mathcal C$-simplicial sets, $\mathcal C$-spaces, etc. 

Given a $\C$-module $M$ and a $\C^\op$-module $N$, their coend is the $\bk$-module $M \otimes_\C N$ defined as the coequalizer of the two maps
\[\bigoplus_{c,c' \in \C, f\in \C(c,c')} M(c) \otimes_\bk N(c') \rightrightarrows \bigoplus_{c\in \C} M(c) \otimes_\bk N(c)\]
given by $m\otimes n$ mapping to $f(m) \otimes n$ and $m\otimes f(n)$, respectively.

Similarly, given a $\C$-set $M$ and a $\C^\op$-set $N$, their coend is the set $M \otimes_\C N$ defined as the coequalizer of the two maps
\[ \coprod_{c,c' \in \C, f\in \C(c,c')} M(c) \times N(c') \rightrightarrows \coprod_{c\in \C} M(c) \times N(c)\]
given by $(m, n)$ mapping to $(f(m) , n)$ and $(m, f(n))$, respectively.

These construction work analogously to the tensor product or the balanced product and have many useful properties. For example, the $\C$-module 
\[ M \otimes_\C \bk\C(-,-)\]
 is canonically isomorphic to $M$, where $\bk\C(-,-)$ is the ``representable'' $\C^\op \times \C$-module given by taking the free $\bk$-module on the morphisms in $\C$. A more detailed introduction can be found in MacLane \cite[Chapter X]{maclane} and \cite[Section 2]{PWG}.

\subsection{Stability categories}

In this section, we review the framework of stability categories used in \cite{Pa2}. The following is {\cite[Def 3.1]{Pa2}}.

\begin{definition}\label{Def:stability groupoid}
Let $(\G,\oplus,0)$ be a monoidal skeletal groupoid whose monoid of objects is the natural numbers $\N$ with addition. The automorphism group of the object $n \in \N$ is denoted $G_n = \Aut^\G(n)$. Then $\G$ is called a \emph{stability groupoid} if it satisfies the following properties:
\begin{enumerate}
\item The monoidal structure
\[ \oplus \colon  G_m\times G_n \inject G_{m+n}\]
is injective for all $m,n\in \N$.
\item The group $G_0$ is trivial.
\item $(G_{l+m}\times 1)\cap (1\times G_{m+n}) = 1\times G_{m} \times 1\subset G_{l+m+n}$ for all $l,m,n\in \N$.
\end{enumerate}

A \emph{homomorphism} of stability groupoids is a monoidal functor sending $1$ to $1$.
\end{definition}

The following is a special case of a definition of Quillen.

\begin{definition}
Let $(\G,\oplus,0)$ be a monoidal groupoid. Let $U\G$ be the category which has the same objects as $\G$, and its morphisms $f\colon A \to B$ are equivalence classes of pairs $(\hat f,C)$ where $C$ is an object in $\G$ and $\hat f$ is an (iso)morphism  $C\oplus A \to B$ in $\G$. Two of these  pairs $(\hat f,C)$ and $(\hat f',C')$ are equivalent if there is an isomorphism $h\colon  C \to C'$ (in $\G$) such that the diagram
\[ \xymatrix{
C\oplus A \ar[r]^{\hat f} \ar[d]_{h\oplus \id_A} & B\\
C'\oplus A \ar[ru]_{\hat g'}
}\]
commutes. We will denote the equivalence class of $(\hat f,C)$ by $[\hat f,C]$. Composition is defined by
\[ [\hat f, C] \circ [ \hat g, D] = [ \hat f \circ (\id_C \oplus \hat g), C \oplus D ] \]
for $f = [\hat f , C] \colon A' \to A''$ and $g = [ \hat g, D] \colon A \to A'$.
\end{definition}

\begin{remark}
Let $\underline{*}$ be the $\G$-set that is a singleton for every $n$. Note that
\[ U\G(-_1,-_2) \cong \underline{*} \otimes_\G \G( - \oplus -_1,-_2)\]
as $\G^\op\times \G$-sets.
\end{remark}

\begin{definition}\label{Def:stability category}
	If $\G$ is a braided stability groupoid, then we call $U\G$ the \emph{stability category} of $\G$. We will denote the braiding by $b_{m,n} \colon m \oplus n \to n\oplus m$.
\end{definition}

\begin{example}
The following is a list of some stability groupoids.
\begin{itemize}
\item Trivial groups $1=(1)_{n \in \N}$. Functors from $U1$ to the category of abelian groups are the same data as a graded modules over a polynomial ring $\Z[x]$.
\item Symmetric groups $\fS = (\fS_n)_{n\in\N}$.  The category $U \fS$ is equivalent to the category of finite sets and injections that is denoted $\FI$ by Church--Ellenberg--Farb \cite{CEF}.
\item Braid groups $\Br = (\Br_n)_{n\in\N}$.
\item Pure braid groups $\PBr= (\PBr_n)_{n\in\N}$. Note that this stability groupoid is not braided.
\item General linear groups $\GL(R)=(\GL_n(R))_{n \in \N}$. If $R$ is commutative and $U$ is a subgroup of the group of units in $R$, we let $\GL^U(R)_n$ denote the subgroup of  $\GL_n(R)$ of matrices with determinant in $U$ and let $\GL^U(R)=(\GL_n^U(R))_{n \in \N}$. The categories $U\GL(R)$ and $U\GL^U(R)$ are equivalent to the categories denoted $\VIC(R)$ and  $\VIC^U(R)$ respectively by Putman--Sam \cite{PS}.
\item Congruence subgroups $\GL(J)= (\GL_n(J))_{n\in \N}$ for an ideal $J$. Note that this stability groupoid is generally not braided.
\item Symplectic groups $\Sp(R) =(\Sp_{2n}(R))_{n \in \N}$. The category $U\Sp(R)$ is equivalent to the category $\SI(R)$ of Putman--Sam \cite{PS}.
\item Mapping class groups of orientable surfaces with one boundary component $\Mod=(\Mod_{g,1})_{g \in \N}$.
\end{itemize}
See Randal-Williams--Wahl \cite{RWW} for more examples of stability groupoids and details on the monoidal structure and braidings.

\end{example}

The following proposition summarizes some results about stability categories.

\begin{proposition}\label{prop:stabcat}
Let $\G$ be a braided stability groupoid. Then:
\begin{enumerate}
\item \cite[Proposition 1.8(i)]{RWW} $0$ is initial in $U\G$. Let $\iota_m$ denote the unique map from $0$ to $m$.
\item \cite[Proposition 1.8(ii)]{RWW} $U\G$ is a pre-braided monoidal category. That is, the diagram
\[\xymatrix{
m \ar[r]^{\id_m \oplus \iota_n} \ar[rd]_{\iota_n \oplus \id_m} & m \oplus n \ar[d]^{b_{m,n}}\\
& n \oplus m
}\]
commutes.
\item \cite[Proposition 3.11(a)]{Pa2} Every map in $U\G$ is a monomorphism.
\item \label{item:semisimpl} \cite[Theorem 2.3]{Pa2} Let $\Delta_{inj,+}$ be the category of ordered finite sets and injective ordered maps. There is a unique monoidal functor $\Delta_{inj,+} \to U\G$ that sends $\{0\}$ to $1$.
\end{enumerate}
\end{proposition}

Lastly, we add one new technical observation that we will need later in the paper.

\begin{lemma}
 Let $\G$ be a braided stability groupoid. Let $f \colon l \to m$ be a 
 map in $U\G$ and $n$ an object in $U\G$. Then the diagram
 \[\xymatrix{ n \oplus l \ar[r]^{\id_n \oplus f} \ar[d]_{b_{n,l}} & n 
 \oplus m \ar[d]^{b_{n,m}} \\
 l\oplus n \ar[r]_{f \oplus \id_n} & m\oplus n}\]
 commutes.
 \end{lemma}
 
 \begin{proof}
 By construction of $\UG$, we can find a morphism $\hat f\colon (m-l) \oplus l \to 
 m$ in $\G$ such that the composition
 \[ l \stackrel{ \iota_{m-l} \oplus \id_m } \longrightarrow (m-l) \oplus 
 l \stackrel{\hat f} \longrightarrow m\]
 is $f\colon l \to m$. Thus we can split the problem in proving that both  
 squares
 \[\xymatrix{ n \oplus l \ar[rr]^<<<<<<<<{\id_n \oplus \iota_{m-l} \oplus 
 \id_{l}} \ar[d]_{b_{n,l}} &&n \oplus (m-l) \oplus l 
 \ar[rr]^<<<<<<<<{\id_n \oplus \hat f} \ar[d]_{ b_{n,(m-l) \oplus l}} && 
 n \oplus m \ar[d]^{b_{n,m}} \\
 l\oplus n \ar[rr]_<<<<<<<<{ \iota_{m-l} \oplus \id_{l}\oplus \id_n } && 
 (m-l) \oplus l \oplus n \ar[rr]_<<<<<<<<{\hat f \oplus \id_n}&& m\oplus 
 n}\]
 are commutative. The right square commutes because $b$ is a braiding of 
 $\G$. For the left square, note that
 \[ b_{n,(m-l) \oplus l} = (\id_{m-l} \oplus b_{n,l}) \circ (b_{n,m-l} 
 \oplus \id_{l} )\]
 by the hexagon relation in a braided monoidal category. Using the pre-braiding property, we see that
  \[\xymatrix{
 n \oplus l \ar[rr]^<<<<<<<<{\id_n \oplus \iota_{m-l} \oplus \id_{l}} 
 \ar[d]_{b_{n,l}}\ar[rrd]_<<<<<<<<{\iota_{m-l} \oplus\id_n \oplus \id_{l} 
 }  &&n \oplus (m-l) \oplus l \ar[d]^{b_{n,m-l} \oplus \id_{l}} \\
 l\oplus n\ar[rrd]^<<<<<<<<{\iota_{m-l} \oplus \id_{l} \oplus \id_{n}} 
 &&(m-l) \oplus n \oplus l \ar[d]^{ \id_{m-l} \oplus b_{ n,l}}\\
 &&(m-l) \oplus l\oplus n} \]
 commutes, which is the left square.
 \end{proof}

The precomposition of $b_{n,m}$ gives a map $\UG(m\oplus n, - ) \to \UG(n \oplus m, -)$ of $\UG$-sets. The lemma shows that this is in fact functorial in $m$. But it is not functorial in $m$ and $n$, unless $b$ is a braiding of $\UG$.

\begin{corollary}\label{cor:newbraidiso}Let $\G$ be a braided stability groupoid.
Precomposing with the braiding gives an isomorphism
\[ \UG(-\oplus n, -) \longrightarrow \UG(n \oplus -, -)\]
as $\UG^\op\times \UG$-sets.
\end{corollary}

\subsection{Central stability homology and degree-wise coherence}

In this subsection, we recall the definition of central stability homology and how it relates to the generation degree of syzygies of a $U\G$-module. To define this we need that $\G$ is a braided stability groupoid and we assume this for the remainder of this section.

We recall the notion of central stability homology of $U\G$-modules.

\begin{definition}\label{Def:central stability chains}
For a $U\G$-module $A$, consider
\[ A \otimes_{U\G} \Z U\G(- \oplus -_1,-_2)\]
as a $\Delta_{inj,+}^\op \times U\G$-module using the functor $\Delta_{inj,+} \to \UG$ from \autoref{prop:stabcat} \ref{item:semisimpl}. In other words, it is an augmented semisimplicial $U\G$-module. This gives rise to a $U\G$-chain complex that we denote by $\widetilde C^\G_*(A)$.
Call $\widetilde C^\G_*(A)$ \emph{central stability chains} of $A$. We write $\widetilde H^\G_i(A)$ for $H_i(\widetilde C^\G_*(A))$ and refer to it as the \emph{central stability homology} of $A$.
\end{definition}

\begin{remark}\label{rem:CS}
The objects in $\Delta_{inj,+}$ are the integers $n\ge -1$. Given an augmented semisimplicial abelian group $A_\bullet\colon \Delta_{inj,+}^\op \to \Ab$, the associated chain complex $C_*(A)$ is given by $C_p(A) = A_p$, so it is supported in degrees $\ge -1$. Because the inclusion $\Delta_{inj,+} \to U\G$ sends $n$ to $n+1$, we get a degree shift in the indexing:
\[ \tilde C^\G_p(A)_n =  A \otimes_{U\G} \Z U\G(- \oplus (p+1), n)  \cong \Ind_{G_{n-p-1}}^{G_n} A_{n-p-1}\]
This is consistent with the notation in \cite{Pa2}.
\end{remark}

In this paper, a $U\G$-module is called \emph{free} if it is isomorphic to the direct sum of representable functors $\Z U\G(m,-)$ for $m\in \N$. 

\begin{definition} \label{defPresentation}
A free $U\G$-module is said to be generated in degrees $\le d$ if it is a direct sum of representable functors $\Z U\G(m,-)$ with $m \leq d$. A $U\G$-module $A$ is said to be generated in degrees $\le d$ if there is a free $U\G$-module generated in degrees $\le d$ that surjects onto $A$.  A $U\G$-module is said to be presented in degree $\le r$ if $A$ is the cokernel of a map $P_1 \to P_0$ between free $U\G$-modules $P_0,P_1$ that are generated in degrees $\le r$. 
\end{definition}

It will follow from \autoref{prop:res of finite type} that these definitions of generation degree and presentation degree coincide with those given in the introduction.

The following condition on a stability category will allow us to relate vanishing of central stability homology with the generation degree of syzygies of $U\G$-modules. It is a condition that is known to hold for a variety of stability categories and implies homological stability. We will use it to establish representation stability results. 

\begin{definition}\label{Def:H3}
We say a stability category $U\G$ satisfies \con{H3($k,a$)} if 
\[ \widetilde H^\G_i(\Z U\G(0,-))_n \cong 0\]
for all $n> i\cdot k+a$.
\end{definition}

The following proposition is a list of a few stability categories that satisfy \con{H3}. This list is far from exhaustive and instead focuses primarily on those categories that will be relevant later in the paper. 

\begin{proposition} 
\label{H3examples}\ 
\begin{enumerate}
\item  $U1$, $U \fS$ and $U\Br$ all satisfy \con{H3($1,1$)}.
\item  $U\GL^{ U}(R)$ satisfies \con{H3($2,s+1$)}, where $s$ denotes the stable rank of $R$. 
\end{enumerate}
%The stability categories $U1$, $U \fS$ and $U\Br$ all satisfy \con{H3($1,1$)}. The category $U\GL^{\mathfrak U}(R)$ satisfies \con{H3($2,s+1$)}, where $s$ denotes the stable rank of $R$. 
\end{proposition}

\begin{proof}
For $U1$, note that $\widetilde C^1_p(\Z U\G(0,-))_n \cong \Z$ if $n>p$ and zero otherwise. For $n>p$, the differentials  $\widetilde C^1_p(\Z U\G(0,-))_n \to \widetilde C^1_{p-1}(\Z U\G(0,-))_n$ are given by the identity map if $p$ is even and by the zero map if $p$ is odd. Therefore $\widetilde H^1_p(\Z U\G(0,-))_n \cong \Z$ if $p=n-1$ is odd and zero otherwise, and thus $\widetilde H^1_p(\Z U\G(0,-))_n \cong 0$ for all $n> p+1$.

For the two cases $U \fS$ and $U\Br$ see \cite[Remark 5.6]{Pa2}. For $U\GL^{ U}(R)$ see \cite[Proposition 3.20 iii]{MPW}. All of these results follow quickly from high connectivity results for certain $CW$-complexes due to Farmer \cite[Theorem 5]{Fa}, Hatcher--Wahl \cite[Proposition 7.2]{hatcherwahl}, and Randal-Williams--Wahl \cite[Lemma 5.10]{RWW}.
%
%%The reduced cellular chains on the complex of injective words on an alphabet of size $n$ is isomorphic to $\widetilde C_*^\fS(T)_n$ for $T$ the trivial $U \fS$-module. By Farmer \cite[Theorem 5]{Fa}, the complex of injective words on an alphabet of size $n$ is $n-2$-connected so $\widetilde H_i^\fS(T)_n \cong 0$ for $i \leq n-2$.
%
%%Now let $T$ denote the trivial $UBr$-module. The chain complex $\widetilde C_*^\fS(T)_n$ is isomorphic to the reduced cellular chains on one of the arc complex considered \cite[Proposition 7.2]{hatcherwahl}. Since that arc complex is $n-2$-connected, $\widetilde H_i^{Br}(T)_n \cong 0$ for $i \leq n-2$.
%
%
\end{proof}

%\todo[Jeremy]{Here or in the introduction we should define generation degree and presentation degree.}
%
%\todo[Jeremy]{Define free}
%%

The following theorem links the generation degree of the syzygies of a $U\G$-module to its central stability homology. See \cite[Theorem 5.7]{Pa2}.

\begin{theorem}\label{thm:res of finite type}
Assume $U\G$ satisfies \con{H3($k,a$)}. Let $A$ be a $U\G$-module and $d_0,d_1,\dots \in\Z$ with $d_{i+1}-d_i \ge \max(k,a)$,
then the following statements are equivalent.
\begin{enumerate}
\item There is a resolution
\[ \dots \to P_1 \to P_0 \to A \to 0\]
with $P_i$ that are freely generated in ranks $\le d_i$.
\item The homology
\[ \widetilde H^\G_i(A)_n \cong 0\]
for all $i\le -1$ and all $n > d_{i+1}$.
\end{enumerate}
\end{theorem}

%\todo[Jeremy]{I think we might never use the following theorem and proposition. If that is the case, we should cut it. }
%
%The following follows from Church--Ellenberg \cite[Theorem A]{CE} and \cite[Proposition 4.2]{CE}. Also see the proof of Theorem 3.9 in Version 1 of \cite{CE}. 
%
%\begin{theorem}[Church--Ellenberg \cite{CE}]
%Let $A$ be a $U\fS$-module generated in degree $\leq d$ and and presented in degree $r$. Then there is a free resolution \[ \dots \to P_1 \to P_0 \to A \to 0\] with $P_0$ generated in degree $d$, $P_1$ generated in degree $r$ and $P_i$ generated in degree $i + d +r -1$ for $i \geq 2$. 
%\end{theorem}
%
%The following is immediate from \cite[Proposition 2.4]{CMNR} and  \cite[Proposition 4.2]{CE}.
%
%\begin{proposition}
%Let $A$ be a $U \fS$-module. Then $A$ has generation degree $\leq d$ if and only if $ H^{\fS}_{-1}(A)_n \cong 0$ for $n>d$ and $A$ has presentation degree $\leq r$ if and only if $ H^{\fS}_{-1}(A)_n \cong H^{\fS}_{0}(A)_n \cong 0$ for $n>r$.
%
%
%\end{proposition}

Recall that a stability category $U\G$ is \emph{degree-wise coherent} for all $U\G$-submodules $A \subseteq B$, if $B$ is presented in finite degree and $A$ is generated in finite degree, then $A$ is presented in finite degree. Equivalently, a stability category $U\G$ is degree-wise coherent if the presentation degree of a $U\G$-module can be used to find finite bounds for the generation degree of higher syzygies (see e.g.\ \cite[Corollary 2.36]{MillerWilson2}). The following definition quantifies this. %\todo{Jeremy: Should we add superscript $\G$s everywhere for central stability homology?}

\begin{definition}\label{Def:coherence}
A function $\Theta\colon \N^3 \m \N$ is called an \emph{coherence function} for $U\G$ if all $U\G$-modules $A$ that have the property that $\widetilde H_{-1}^\G(A)_n=0$ for $n>g$ and $\widetilde H^\G_{0}(A)_n=0$ for $n>r$ also have the property that $\widetilde H^\G_i(A)=0$ for $n>\Theta(g,r,i)$.
\end{definition}

The following is a list of some stability categories that satisfy degree-wise coherence that will be used later in the paper.

\begin{proposition} \label{CoherenceExamples}
For $\G=1$, we can take $\Theta(g,r,i)=\max(g+1,r)+i$. %$\Theta(g,r,i)=\max(g,r)+i+1$. 
For $\G = \fS$, we may take $\Theta(g,r,i)=g+\max(g,r)+i$. %$\Theta(g,r,i)=\max(g+\max(g,r)+i,g+i+1)$.
\end{proposition} 

\begin{proof}
Let us first consider $\G = 1$. Assume that $\widetilde H^1_{-1}(A)_n \cong 0$ for all $n>g$ and $\widetilde H^1_{0}(A)_n \cong 0$ for all $n>r$. The central stability complex is defined as $\widetilde C^1_i(A)_n \cong A_{n-i-1}$ and the differential $\widetilde C^1_i(A)_n \to \widetilde C^1_{i-1}(A)_n$ is the transition map for $i$ even and the zero map for $i$ odd. This implies that $A_{n-1} \to A_n$ is surjective for $n> g$ and injective for $n>r$. 
%Because $U1$ satisfies \con{H3($1,1$)}, \autoref{thm:res of finite type} implies that $A$ is presented in degrees $\le \max(r,g+1)$. Therefore $A_n \cong A_{n+1}$ for all $n\ge \max(r,g+1)$ 
%Let us denote the colimit of all $A_n$ by $A_\infty$. 
Calculating central stability homology, we get that $\widetilde H^1_i(A)_n \cong 0$ for even $i$ if the transition map $A_{n-i-1} \to A_{n-i}$ is injective, which it is if $n-i >r$. For odd $i$, $\widetilde H^1_i(A)_n \cong 0$ when the transition map $A_{n-i-2} \to A_{n-i-1}$ is surjective, which it is if $n-i-1>g$. Therefore, $\widetilde H^1_i(A)_n \cong 0$ for all $n>i +\max(r,g+1)$.

Now consider the case that $\G = \fS$ and let $A$ be a $U\fS$-module such that $\widetilde H^\fS_{-1}(A)_n \cong 0$ for all $n>g$ and $\widetilde H^\fS_{0}(A)_n \cong 0$ for all $n>r$. This implies that $A$ is generated in degrees $\le g$ by \cite[Proposition 5.4]{Pa2}. Furthermore, $A$ is presented in degrees $\le \max(g,r)$, because \cite[Proposition 6.2(c)]{Pa2} says that $\widetilde H^\fS_{-1}$ and $\widetilde H^\fS_{0}$ can be computed with the chain complex that is used by Church--Ellenberg \cite{CE} to compute $\FI$-homology and $\FI$-homology detects presentation degree by \cite[Proposition 4.2]{CE}. Further, Church--Ellenberg \cite[Proof of Theorem A]{CE} implies that the $i$-th syzygies are generated in degrees $\le i + g + \max(g,r) -1$. Using \autoref{thm:res of finite type}, we deduce that $\widetilde H^\fS_i(A)_n \cong 0$ for $n> i + g + \max(g,r)$.
\end{proof}

\begin{remark} \label{RemCoh} In contrast to \con{H3}, degree-wise coherence is not known for many stability categories. Church--Ellenberg's result for $\G=\fS$ was generalized to $\G = \fS \ltimes G$ by Ramos \cite{RamosCoherence}. In \cite{MillerWilson2}, degree-wise coherence for $\G = \GL^U(\F_q)$ and $\G = \Sp(\F_q)$ over characteristic zero was established. The examples from this remark and \autoref{CoherenceExamples} summarize the current literature. 

%In characteristic zero, for $\G = \GL(\mathbb F_q)$ or $\Sp(\mathbb F_q)$, we may take $\Theta(g,r,i)=\max(g,r) \cdot 3^i$. In characteristic zero, for $\G = \GL^{U}(\mathbb F_q)$, we may take $\Theta(g,r,i)=(\max(g,r)+\frac{1}{2}) \cdot 3^i -\frac{1}{2}$.
\end{remark}

\subsection{Stability short exact sequences}

In this subsection, we introduce stability short exact sequences. This will be the context of our general representation stability theorem, \autoref{main}. 

\begin{definition}\label{Def:stability SES}
Let $\mathcal N, \G, \mathcal Q$ be stability groupoids and assume that $\G$ and  $\mathcal Q$ are braided. Let $F\colon  \mathcal N \to \mathcal G$ and $F'\colon  \mathcal G \to \mathcal Q$ be homomorphisms of stability groupoids and assume that $F'$ is symmetric. We call this data a \emph{stability short exact sequence} if
\[ 1 \longrightarrow N_n  \longrightarrow G_n \longrightarrow  Q_n \longrightarrow 1 \]
is a short exact sequence for all $n\in \N$.
\end{definition}

By \cite[Lemma 8.3]{Pa2}, the actions of the groups $Q_n$ on $H_i(N_n)$ for varying $n$ assemble to form a $U\cQ$-module which we will call $H_i(\mathcal N)$. Moreover, if $A$ is a $U\G$-modules, there is a $U\cQ$-module $H_i(\mathcal N; A)$ with $H_i(\mathcal N; A)_n = H_i(N_n; A_n)$. The following spectral sequence is similar to the spectral sequence Quillen considered to prove homological stability and was first used in this generality by Putman--Sam \cite{PS}. We use the formulation from \cite[Proposition 8.4]{Pa2}. 

\begin{proposition}  \label{PeterSS}
Let $1 \m \cN \m \cG \m \cQ \m 1$ be a stability short exact sequence. Let $A$ be a $U\G$-module over $\bk$. There are two homologically graded spectral sequences converging to the same thing, one with $(E^2_{p,q})_n \cong \widetilde H_p^\cQ( H_q(\cN;A)  )_n$ and the other with $( \overline E_{p,q}^1)_n \cong (\bk \G_n)^{\otimes p+1} \otimes_{\bk \cN_n} \widetilde H_q^\G(A)_n$.
\end{proposition}

This gives the following corollary.

\begin{corollary}
\label{SSSScora}

Let $A$ be a $U\G$-module with $\widetilde H^\G_i(A)_n \cong 0$ all $n>d_i$. There is a spectral sequence with $(E^2_{p,q})_n \cong \widetilde H_p^\cQ( H_q(\cN;A)  )_n$ and with $(  E_{p,q}^\infty)_n \cong 0$ for $n > \max(d_{-1}, d_0, \dots, d_{p+q})$.

\end{corollary}

\begin{proof}
Consider the spectral sequences from \autoref{PeterSS}. Let us consider the diagonal $p+q=k$ in the spectral sequence $\overline E^1_{p,q}$: the entries $(\overline E^1_{0,k})_n, \dots, (\overline E^1_{k+1,-1})_n$ all vanish if $n> \max(d_{-1}, \dots, d_k)$. Thus $(E^\infty_{p,q})_n$ vanishes on the same diagonals.
\end{proof}

\subsection{Rings, modules, and Tor groups}
\label{SecRing}

The category of $\cG$-modules has a monoidal structure known as the induction tensor product, or Day convolution. This monoidal structure allows one to define rings and modules in the category of $\cG$-modules. Using this, we describe a general context for representation stability. This is somewhat redundant with the framework of stability categories. However, we include both setups in this paper because some arguments and definitions are easier in one than in the other. 

We first give definitions and general properties about these notions and in the end of the section, we connect the Tor groups to generation properties of $U\G$ modules.

\begin{definition}
Let $A$ and $B$ be $\cG$-modules. We define the $\G$-module
\[A \oast_\G B = (A \boxtimes B) \otimes_{\G \times \G} \Z\G(- \oplus -, -),\] 
where $\G \times \G$ acts on $A \boxtimes B$ with the first $\G$ acting on $A$ and with the second $\G$ acting on $B$.
\end{definition}

\begin{remark}
A more elementary way of writing the induction product is
\[(A \oast_\G B)_n \cong \bigoplus_{a+b=n} \Ind_{G_a \times G_b}^{G_n} A_a \boxtimes B_b.\]
\end{remark}

The braiding induces a isomorphism that swaps the factors. The following observation can be found already in Joyal--Street's original  work introducing braided monoidal categories \cite[p.11]{joyalstreet}.

\begin{lemma}\label{lem:oastbraiding}
Assume $\G$ is braided. Let $A$ and $B$ be $\G$-modules. Then there is an isomorphism $b\colon A\oast B \to B \oast A$ induced by the braiding $b$ of $\G$:
\begin{align*}
(A \boxtimes B) \otimes_{\G\times\G} \Z\G(- \oplus - , -)&\longrightarrow (B \boxtimes A) \otimes_{\G\times\G} \Z\G(- \oplus - , -)\\
(a_m \otimes b_n) \otimes g_{n+m} & \longmapsto (b_n \otimes a_m) \otimes  g_{n+m}\circ b_{m,n},
\end{align*}
which makes the category of $\G$-modules braided monoidal with respect to $\oast$. 
\end{lemma}
%
%\begin{proof}
%Let $h_m\in G_m$ and $h_n\in G_n$, then $(h_m \oplus h_n) \circ b_{n,m} = b_{n,m} \circ (h_n \oplus h_m)$. Therefore
%\[ (h_n(b_n) \otimes h_m(a_m)) \otimes  g_{n+m}\circ b_{m,n}  = (b_n \otimes a_m) \otimes  g_{n+m}\circ b_{n,m} \circ (h_n \oplus h_m) =  (b_n \otimes a_m) \otimes  g_{n+m}\circ (h_m \oplus h_n) \circ b_{n,m}.\]
%This proves that the map is well-defined. It is obviously an isomorphism. We omit the vier
%\end{proof}

This monoidal structure allows us to define ring and module objects in the category $\Mod_\G$ of $\G$-modules.

\begin{definition}
A \emph{$\G$-ring} is a monoid object in $(\Mod_\cG,\oast_\G)$. Given a $\G$-ring $R$, a \emph{(left/right) $R$-module} is a (left/right) module object over that ring. 
\end{definition}

The main example of $\G$-ring we want to consider in this paper is $R = \uZ$ which sends all objects to $\Z$ and all morphisms to the identity on $\Z$. The ring structure $\uZ \oast_\G \uZ \to \uZ$ is induced by the multiplication map $\Z \otimes \Z$.

\begin{proposition}
The category of left $\uZ$-modules is equivalent to the category of $\UG$-modules. 
\end{proposition}

\begin{proof}
Let $A$ be a left $\uZ$-module and denote its structure map by $\mu \colon \uZ \oast A \to A$. Observe that
\[ ( \uZ \boxtimes A) \otimes_{\G\times \G} \Z\G(- \oplus -,-) \cong A \otimes_\G \Z U\G(-,-)\]
because $U\G(-,-) \cong \underline * \times_\G G(-,-)$ as $\G^\op\times \G$-sets. Let $B$ be the $U\G$-module defined by $B_n = A_n$ and if $f\in U\G(m,n)$ then $f(a_m) = \mu( a_m\otimes f)$. It is easy to check that this is a well-defined and functorial assignment $A \mapsto B$. Likewise, it is easy to find an inverse functor.
\end{proof}

A left $\uZ$-module can be naturally considered as a two-sided $\uZ$-module.

\begin{proposition}\label{prop:leftrightuZmod}
Let $A$ be a left $\uZ$-module. Consider $A$ as a right $\uZ$-module via $ A \oast_\G \uZ \stackrel{b^{-1}}{\longrightarrow} \uZ \oast_\G A \to A$. These two actions commute, i.e.
\[ \xymatrix{ \uZ \oast_\G A \oast_\G \uZ \ar[r]\ar[d] & A \oast_\G \uZ \ar[d]\\
\uZ \oast_\G A \ar[r] & A}\]
is a commutative diagram.
\end{proposition}

\begin{proof}
Let us consider $A$ as a $U\G$-module. Then the left action is given by 
\[(1_l \otimes a_m) \otimes g_{l+m} \stackrel{\mu}{\longmapsto} [g_{l+m} \circ (\iota_l \oplus \id_m)] (a_m).\]
Therefore the right action is given by
\[(a_m \otimes 1_n) \otimes g_{m+n} \stackrel{b^{-1}}{\longmapsto} (1_n \otimes a_m) \otimes g_{m+n} \circ b^{-1}_{n,m} \stackrel{\mu}{\longmapsto} [g_{m+n} \circ b^{-1}_{n,m} \circ (\iota_n \oplus \id_m)] (a_m).\]
Notice that this is $[g_{m+n} \circ  (\id_m  \oplus \iota_n)] (a_m)$ because $U\G$ is prebraided. We conclude that both compositions of the square are given by
\[ (1_l \otimes a_m \otimes 1_n) \otimes g_{l+m+n} \longmapsto [g_{l+m+n} \circ (\iota_l \oplus \id_m \oplus \iota_n) ](a_m).\qedhere\]
\end{proof}

If $A$ is a $U\G$-module, we will consider it simultaneously as two-sided $\uZ$-module via the actions described above. Next we introduce $\oast$-products over $\G$-rings.

\begin{definition}\label{def:indprodoverR}
If $A$ is a right $R$-module and $B$ is a left $R$-module, we define $A \oast_R B$ to be the coequalizer of the two natural maps: 
\[A \oast_\G R \oast_\G B \rightrightarrows A \oast_R B.\] 
Let $\Tor^R_i( -,B)\colon  \Mod_R \m \Mod_\G$ be the $i$th left derived functor of $- \oast_R B\colon  \Mod_R \m \Mod_\G$.
\end{definition}

If $A$ is a $\UG$-module and $B$ a $\G$-module, then $A \oast_\G B$ has a left $\uZ$-module structure given by
\[ (\uZ \oast_\G A) \oast_\G B \longrightarrow A \oast_\G B.\]
If both $A$ and $B$ are $U\G$-modules, then this action descends to $A \oast_{\uZ} B$ because of \autoref{prop:leftrightuZmod}. Note that $b \colon A\oast_\G B \to B \oast_\G A$ is not $\uZ$-equivariant because in the codomain, $\uZ$ acts on $B$. Note that $b$ descends to a morphism $b \colon A\oast_\uZ B \to B \oast_\uZ A$ using the commutative diagram
\[ \xymatrix{ 
A \oast_\G \uZ \oast_\G B \ar@<2pt>[rrr]\ar@<-2pt>[rrr]\ar[d]_{(\id_B \oast b) \circ (b \oast \id_{\uZ})\circ (\id_A \oast b)}^{=( b\oast \id_A) \circ ( \id_{\uZ}\oast b)\circ (b\oast \id_B)} &&& A \oast_\G B\ar[d]^b\ar[r] & A \oast_\uZ B\ar@{-->}[d]\\
B \oast_\G \uZ \oast_\G B \ar@<2pt>[rrr]\ar@<-2pt>[rrr] &&& B \oast_\G A\ar[r]&B \oast_\uZ A
}\]
where the equality comes from the braid relations.

\begin{lemma}\label{lem:comm tensor}
The swap $b \colon A \oast_\uZ B \to B \oast_\uZ A$ is $\uZ$-equivariant. In particular, $U\G$ acts on $A \oast_\uZ \Z$ via identity with isomorphisms and via zero with non-isomorphisms.
\end{lemma}

\begin{proof}
The following commutative diagram proves the first assertion.
\[\xymatrix{
\uZ \oast_\G A \oast_\G \ar[r]\ar[d]_{b^{-1}\oast \id_B}\ar[rd] & \uZ \oast_\G A \oast_\uZ B \ar[rd] \ar@/^2.0pc/[dd]_<{b^{-1}\!\!} \\
A \oast_\G \uZ \oast_\G B \ar@<2pt>[r]\ar@<-2pt>[r] & A\oast_\G B \ar[r] &A \oast_\uZ B\\
A \oast_\G B \oast_\G \uZ \ar[u]^{\id_A \oast b} \ar[r]\ar[ru] & A \oast_{\uZ} B \oast_\G \uZ \ar[ru]
}\]

The second assertion follows because that is how $\UG$ acts on $\Z$.
\end{proof}

We will prove the same statement for the Tor-groups. To do this, we first will define an $\uZ$-action on $ \Tor^\uZ_*(A,B)$. Let $P_* \to A$ be a free $\UG$-resolution, then the chain complex $P_* \oast_\uZ B$ computes $ \Tor^\uZ_*(A,B)$. It is clear from the analysis above that the differentials are $\uZ$-equivariant (considering the left $\uZ$-action on $P_*$). From the previous lemma, we see that $b\colon P_* \oast_{\uZ} B  \to B \oast_{\uZ} P_*$ is a $\uZ$-equivariant isomorphism, where the codomain computes $\Tor^\uZ_*(B,A)$. We thus get the following corollary.

\begin{corollary}\label{cor:Tor(A,Z)}
The swap $b\colon \Tor^\uZ_*(A,B) \to \Tor^\uZ_*(B,A)$ is $\uZ$-equivariant. In particular, $U\G$ acts on $\Tor^\uZ_*(A,\Z)$ via identity with isomorphisms and via zero with non-isomorphisms.
\end{corollary}

We now give a connection between generation properties and the Tor groups of $\uZ$-modules. We first define the $\uZ$-module $\Z$.

\begin{definition}
Given a $\G$-module $X$, let $X_+$ be $(X_+)_n=X_n$ for $n>0$ and $(X_+)_0=0$. Let $\Z$ denote the $\uZ$-bimodule $\uZ/\uZ_+$.
\end{definition}

The following is an analogue of \autoref{thm:res of finite type} for Tor-groups.

\begin{proposition}\label{prop:res of finite type}
Let $A$ be a $U\G$-module and $d_0,d_1,\dots \in\Z$  with $d_{i} \leq d_{i+1}$ for all $i$. Then the following statements are equivalent:
\begin{enumerate}
\item\label{item:partialres} There is a resolution
\[ \dots \to P_1 \to P_0 \to A \to 0\]
with $P_i$ freely generated in ranks $\le d_i$.
\item The groups  $\Tor_i^\uZ(A,\Z)_n \cong 0$ for $n > d_i$.
\end{enumerate}
\end{proposition}

\begin{proof}
Free modules are isomorphic to modules of the form $X \oast_\G \uZ $ with $X$ a $\G$-module with each $X_n$ free as a $\Z G_n$-module. A standard argument shows that $\Tor_0^\uZ(X \oast_\G \uZ,\Z) \cong X$ and $\Tor_i^\uZ(X \oast_\G \uZ,\Z) \cong 0$ for $i>0$. Applying this fact and the hyper-homology spectral sequence associated to a free resolution shows that i) implies ii).

Now we prove that ii) implies i). Let $A$ be a $U\G$-module with $\Tor_i^\uZ(A,\Z)_n \cong 0$ for $n > d_i$ for all $i$. Let $X_m$ be a free $\Z G_m$ module with a choice of surjection $X_m \m A_m$ and let $Y$ be the $\G$-module which is $X_m$ in degrees $\leq d_0$ and $0$ in higher degrees. Let $P_0=Y \oast_\G \uZ$ and let $P_0 \m A$ be the natural map. Then $\Tor^\uZ_0(P_0,\Z)_n$ surjects onto $\Tor^\uZ_0(A,\Z)_n$ for all $n\in \N$. It follows that the cokernel $W = \coker(P_0 \to A)$ has the property that $\Tor^\uZ_0(W,\Z)_n \cong 0$ for all $n\in\N$ and thus itself must be zero. This proves that $P_0\to A$ is surjective. Continuing, $P_0$ is generated in degrees $\le d_0$ and thus $\Tor^\uZ_i(P_0,\Z)_n \cong 0$ for $n>d_0$ if $i=0$ and for all $n\in \N$ if $i>0$. Let $K_0 = \ker (P_0 \to A)$. The long exact sequence of $\Tor$ groups associate to the short exact sequence \[0\m K_0 \m P_0 \m A \m 0\] implies that $\Tor^\uZ_0(K_0,\Z)_n \cong \Tor^\uZ_1(A,\Z)_n$ for $n>d_0$.  The same argument as before ensures the existence of a surjection $P_1 \to K_0$ from a free $U\G$-modules $P_1$ generated in degrees $\le d_1$. 
 We proceed by induction. \end{proof}

\subsection{Splitting complexes and the Koszul complex}

In this subsection, we recall the Koszul complex of \cite[Example 19.5]{GKRW1}. We will need some details concerning its construction so we repeat the augments of \cite[Example 19.5]{GKRW1} here. We begin by recalling the two-sided bar construction. 

\begin{definition}
Given a $\G$-ring $R$, a right $R$-module $A$, and a left $R$-module $B$, let $B_p(A,R,B) =A \oast_\G R^{\oast_\G p} \oast_\G B$.
\end{definition}

The natural maps $A \oast_\G R \m A$, $R \oast_\G R \m R$, and $R \oast_\G B \m B$ give face maps making $B_\bullet(A,R,B)$ into a semi-simplicial $\G$-module. If $R$ is a unital ring, then  $B_\bullet(A,R,B)$  has the structure of a simplicial $\G$-module via the unit map $\Z \m R$. Let $B_*(A,R,B)$ denote the chain complex associated to $B_\bullet(A,R,B)$ whose differential is the alternating sum of the face maps. If $R$ is unital, let $\bar B_*(A,R,B)$ be the quotient of $ B_*(A,R,B)$ by the images of the degeneracies. The proof of the following proposition is the same as the analogous proof in the classical setting.

\begin{proposition}
Let $R$ be a unital $\G$-ring, let $A$ be a right $R$-module, and let $B$ be a left $R$-module. Then the following statements are true.
\begin{enumerate}
\item $\bar B_*(R,R,B) \to B \to 0$ is an exact sequence.
\item If $R_n$ and $B_n$ are free abelian groups for all $n$, then $- \oast_\G R$ and $- \oast_\G B$ are exact functors and thus $\bar B_p(R,R,B)$ are flat left $R$-module for all $p$.
\item $\bar B_*(A,R,B) \cong A \oast_R \bar B_*(R,R,B)$.
\item If $R_0 \cong \Z$, then $\bar B_*(A,R,B) = B_*(A,R_+,B)$.
\end{enumerate}
\end{proposition}

We then get the following corollary for $R = \uZ$. 

\begin{corollary}
$H_i(B_*(A,\uZ_+,\Z)) \cong \Tor^\uZ_i(A,\Z)$
\end{corollary}

The following is a chain-level enhancement of \autoref{cor:Tor(A,Z)}.

\begin{lemma} \label{changeOrder}  Let $A$ be a $\UG$-module. There is a zig-zag of quasi-isomorphisms of $U\G$-modules between $B_*(A,\uZ_+,\Z)$ and $B_*(\Z,\uZ_+,A)$. 
\end{lemma}

\begin{proof}
Consider the double complex $B_*(B_*(\uZ,\uZ_+,\Z), \uZ_+, A)$ and its associated spectral sequences. Note that $B_*(\uZ,\uZ_+,\Z)$ is a free resolution of $\Z$. Both spectral sequences collapse on the first page. One of the spectral sequences has
\[ B_p(\Z,\uZ_+,A)\]
in the row $q=0$ and zero everywhere else on the first page. The other spectral sequence has
\[ B_p(\uZ,\uZ_+,\Z) \oast_{\uZ} A\]
in the row $q=0$ and zero everywhere else on the first page. Together with \autoref{lem:comm tensor}, this implies that there is a zig-zag of quasi-isomorphisms
\[ B_*(A,\uZ_+,\Z) \cong A \oast_{\uZ} B_*(\uZ,\uZ_+,\Z) \leftarrow B_*(B_*(\uZ,\uZ_+,\Z), \uZ_+, A) \rightarrow B_*(\Z,\uZ_+,A).\qedhere\]
\end{proof}

We now recall the relationship between the reduced bar complex
\[ B_p(A,\uZ_+,\Z)_n 
%\cong ( A \boxtimes \uZ_+ \boxtimes \dots \boxtimes \uZ_+ \boxtimes \Z) \otimes_{\G^{\times (p+2)}} \Z\G(- \oplus \dots \oplus -, -) 
\cong \bigoplus_{\substack{n_0+\dots +n_{p}=n\\ n_i > 0}} \Ind_{G_{n_0} \times \dots \times G_{n_p}}^{G_n}A_{n_0}\]%\boxtimes \Z \boxtimes \dots \boxtimes \Z\]
 and a semi-simplicial set known as the $E_1$-splitting complex. It has previously been considered by Charney \cite{CharStab}, Hepworth \cite{HepEdge}, and Galatius--Kupers--Randal-Williams \cite{GKRW1,GKRW2,GKRW3}. 

\begin{definition}
Let \[S^{E_1}_p(\G)_n = \bigsqcup_{\substack{n_0+\dots +n_{p+1}=n\\ n_i > 0}}  G_n/(G_{n_0} \times \cdots \times G_{n_{p+1}} ).\]
This has the structure of a semi-simpicial set with $i$th face map induced by the map $G_{n_i} \times G_{n_{i+1}} \m G_{n_i+n_{i+1}}$. We denote this by $S^{E_1}_\bullet(\G)_n$.
\end{definition}

The superscript $E_1$ refers to the associative operad (or operads equivalent to it in various categories). The semi-simplicial set $S^{E_1}_\bullet(\G)_n$ has dimension equal to $n-2$. 
%Note that an isomorphic description of $S^{E_1}_\bullet(\G)_n$ is given by
%\[ S^{E_1}_p(\G)_n \cong \{ (V_0, \dots, V_{p+1}) \mid \text{ $V_i$ is a subobject of $n$ in $U\G$ and } V_0\oplus \dots \oplus V_{p+1} = n\}\]
%and the $i$th face map sends $(V_0, \dots, V_{p+1})$ to $(V_0, \dots, V_i \oplus V_{i+1} ,\dots,V_{p+1})$. 
Recall that the reduced homology of the realization of a semi-simplicial set can be computed by its augmented Moore complex $\Z[ S^{E_1}_*(\G)_n] \to \Z$.  It is immediate that this complex is isomorphic to $B_{*+2}(\Z,\uZ_+,\Z)_n$. Thus we obtain an isomorphism 
\[ \Tor_{i+2}^{\uZ}(\Z,\Z) \cong \widetilde H_{i}( \Vert S^{E_1}_\bullet(\G)_n\Vert)\]
for $i\ge 0$. As in \cite{GKRW1}, we make the following definition. 
\begin{definition}
We say $\G$ satisfies the \emph{standard connectivity assumption} if $\widetilde H_i(\Vert S^{E_1}_\bullet(\G)_n\Vert) \cong 0$ for $i \leq n-3$. If $\G$ satisfies the standard connectivity assumptions, then we denote  $\widetilde H_{n-2}(\Vert S^{E_1}_\bullet(\G)_n\Vert)$ by $\SSt_n$ and call this the $n$th \emph{split Steinberg module} of $\G$. 
\end{definition}

The standard connectivity assumption is equivalent to the statement that $\Tor_i^{\uZ}(\Z,\Z)_n\cong 0$ for $i \neq n$. This can be interpreted as Koszulness of $\uZ$.

For $\G=\GL_n(R)$, the splitting complex is Charney's split version of the Tits building. The name split Steinberg module is in analogy with the fact that the top reduced homology of the classical Tits building is the classical Steinberg module. %View $S^{E_1}_p(\G)_n$ as the set of tuples $(V_0,\dots, V_{p+1})$ with $V_i$ a subobject of $n$ and $n = V_0 \oplus \dots \oplus V_{p+1}$. 

The following theorem lists some groupoids that are known to satisfy the standard connectivity assumption. 

\begin{theorem}[Charney, Hepworth, Galatius--Kupers--Randal-Williams] \label{AclyList}
For $\G=\fS$, $\Br$, $\GL(R)$ for $R$ a PID, or $\Mod$,
 $\G$ satisfies the standard connectivity assumption. \end{theorem}
\begin{proof}
The case of $\GL(R)$ for $R$ a PID is due to Charney \cite[Theorem 1.1]{CharStab}. The case of $\G=\fS$ or $\Br$ is due to Hepworth \cite[Proposition 4.1 and 4.11]{HepEdge}. The case of $\Mod$ is due to Galatius--Kupers--Randal-Williams \cite[Theorem 3.4]{GKRW2}.
\end{proof}

Using the standard connectivity assumption, we will now construct the \emph{Koszul complex} $K_*(A)$ as a quasi-isomorphic subcomplex of $B_*(A,\uZ_+,\Z)$.

Let $A^{\ge m}$ denote the submodule of $A$ with
\[ A^{\ge m}_n = \begin{cases} A_n &n \ge m\\ 0 & n<m.\end{cases}\]
Consider the filtration
\[ F_pB_q(A,\uZ_+,\Z)_n = B_q(A^{\ge n-p}, \uZ_+,\Z)_n .\]
%\cong \bigoplus_{\substack{n_0+\dots +n_{q}=n\\ n_i >0, n_0\ge n-p}} \Ind_{G_{n_0} \times \dots \times G_{n_q}}^{G_n}A_{n_0}.\]
This gives us a spectral sequence
\[ E^0_{pq} =F_p B_{p+q}(A,\uZ_+,\Z)_n/ F_{p-1} B_{p+q}(A,\uZ_+,\Z)_n \cong A_{n-p} \oast B_{p+q}(\Z,\uZ_+,\Z)_p,\]
where $A_{n-p}$ is the $\G$-module that is $A_{n-p}$ in degree $(n-p)$ and zero elsewhere.
%  \bigoplus_{\substack{n_0+n_1+\dots +n_{p+q}=n\\ n_i >0, n_0 = n-p}} \Ind_{G_{n_0} \times \dots \times G_{n_{p+q}}}^{G_n}A_{n_0} \]
This spectral sequence converges to $H_{p+q}(B_*(A,\uZ_+,\Z)_n) \cong \Tor^{\uZ}_{p+q}(A,\Z)_n$.
Because the degree of $A_{n-p}$ cannot change and 
%the $E^0$-page is isomorphic to
%\[ \Ind_{G_{n-p} \times G_p}^{G_n} A_{n-p} \boxtimes  \bigoplus_{\substack{n_1+\dots +n_{p+q}=p\\ n_i >0}} \Ind_{G_{n_1} \times \dots \times G_{n_{p+q}}}^{G_p}\Z \cong \Ind_{G_{n-p} \times G_p}^{G_n} A_{n-p} \boxtimes \Z[ S^{E_1}_{p+q-2}(\G)_p].\]
%Due to 
the standard connectivity assumption, we conclude that
\[ E^1_{pq} \cong A_{n-p} \oast \Tor^{\uZ}_{p+q}(\Z,\Z)_p\]
%\Ind_{G_{n-p} \times G_p}^{G_n} A_{n-p} \boxtimes  \widetilde H_{p+q-2}(|| S^{E_1}_\bullet(\G)_p ||)\]% \cong \Ind_{G_p \times G_{n-p}}^{G_n} \SSt_p \otimes A_{n-p}\]
is concentrated in the row $q =0$ 
as $\Tor^{\uZ}_{p+q}(\Z,\Z)_p$ vanishes unless $q= 0$, and then it is $\Tor^{\uZ}_{p}(\Z,\Z)_p = \SSt_p$.
%because $\widetilde H_{p+q-2}(|| S^{E_1}_\bullet(\G)_p ||) \cong \SSt_p$ if $q=0$ and zero otherwise.
Therefore $E^1_{*,0}$ computes $\Tor^{\uZ}_{*}(A,\Z)$. We denote this complex by 
\[K_*(A)_n = E^1_{*,0} \cong (A_{n-p} \oast \SSt_p)_n = (A \oast \SSt_p)_n \cong \Ind_{G_{n-p} \times G_p}^{G_n} A_{n-p} \boxtimes \SSt_p\]
and call it the \emph{Koszul complex} of $A$. Here, by abuse of notation, $\SSt_p$ also denotes the $\G$-module that is $\SSt_p$ in degree $p$ and zero everywhere else.
%Observe that
%\[ K_p(A)_n = E^1_{p,0} \inject E^0_{p,0} = F_pB_p(A,\uZ_+,\Z)_n \subset B_p(A,\uZ_+,\Z)_n\]
%because $F_pB_{p+q}(A,\uZ_+,\Z)_n \cong 0$ for $q>0$, as $n = n_0 + \dots + n_{p+q} \ge (n-p)+ p+q$.

Let us also describe the boundary map in $K_*(A)$. In $B_*(A,\uZ_+,\Z)$, the boundary map is given by the alternating sum of the maps induced by all ways of multiplying two adjacent factors into one factor.  Then $\SSt_p$ is in the kernel of $B_p(\Z,\uZ_+,\Z) \to B_{p-1}(\Z,\uZ_+,\Z)$. The restriction of the boundary map of $B_*(A,\uZ_+,\Z)$ to $K_*(A) \subset B_*(A,\uZ_+,\Z)$ therefore fits into the commutative diagram
\begin{equation}\label{Kd}\vcenter{\xymatrix{
A \oast \SSt_p \ar@{^(->}[r]\ar[d] & B_p(A,\uZ_+,\Z) = A \oast \uZ_+ \oast \uZ_+ \oast \dots \oast \uZ_+ \oast \Z \ar[d]^{m \oast \id}\\
A \oast \SSt_{p-1} \ar@{^(->}[r] & B_{p-1}(A,\uZ_+,\Z) =  A \oast \uZ_+ \oast \dots \oast \uZ_+ \oast \Z 
}}\end{equation}
where $m \colon A \oast \uZ_+ \to A$ denotes the multiplication map of $A$ as a right $\uZ$-module.

We end this section by considering a $U\G$-action on $K_*(A)$.
 %and $B_*(A,\uZ_+,\Z)$. We start with $K_*(A)$. 
 %Note that $K_p(A) = A \oast_\G \SSt_p$, where $\SSt_p$ is considered as the $\G$-module that is $\SSt_p$ in degree $p$ and zero everywhere else. 
 Because $K_p(A) = A \oast \SSt_p$, we can use the $\UG$-action described after \autoref{def:indprodoverR}.
 %This allows us to equip $K_p(A)$ with a $U\G$ action as described after \autoref{def:indprodoverR}. 
 Note that the inclusion 
 \[ K_p(A) \inject B_p(A,\uZ_+,\Z) \cong A \oast_\uZ B_p(\uZ,\uZ_+,\Z)\]
  is $\UG$-equivariant with $U\G$ action on $A \oast_\uZ B_*(\uZ,\uZ_+,\Z)$ induced by the action on $A$. Therefore, the differentials of $K_*(A)$ are also $\UG$-equivariant and the action on its homology coincides with the action of $\UG$ on $\Tor^\uZ_*(A,\Z)$ as described in \autoref{cor:Tor(A,Z)}.

% Using this action, we have the following proposition.
%
%\begin{proposition}
%$K_*(A)$ is a $U\G$ chain complex and on its homology all non-isomorphisms act via zero.
%\end{proposition}
%
%\begin{proof} 
%The boundary map commutes with the $U\G$ action defined above:
%\[ \xymatrix{
%a \otimes (L_1, \dots, L_p) \ar@{|->}[r]\ar@{|->}[d] & (C \subset C \oplus L_1)^*(a) \otimes (L_2, \dots, L_p)\ar@{|->}[d]\\
%(C \to f(C) \subset C_f \oplus f(C))^*(a)  \otimes (f(L_1), \dots, f(L_p)) \ar@{|->}[r] & (C \to f(C)\subset C_f\oplus f(C) \oplus f(L_1))^*(a) \otimes (f(L_2), \dots, f(L_p))\\
%}\]
%Similarly, $B_*(A,\uZ_+,\Z) \cong A \otimes_\uZ B_*(\uZ,\uZ_+,\Z)$ is a $U\G$ chain complex. And the inclusion $K_p(A) \inject B_p(A,\uZ_+,\Z)$ is $U\G$-equivariant. Because $B_*(\uZ,\uZ_+,\Z) \to \Z$
%
%
%
%
%\end{proof}

\subsection{Polynomial coefficient systems}
\label{secPoly}
We now recall the definition of polynomial coefficient systems and describe their basic properties.

We start with the shift endofunctor. Let $(\C,\oplus, 0)$ be a small monoidal category, then we can precompose a $\C$-module $A$ with the functor $\oplus\colon \C \times \C \to \C$ and we obtain the $\C\times \C$-module
\[ A \otimes_\C \Z\C(-, - \oplus -).\]
Fixing an object $c\in \C$, 
\[ \Sigma^c A := A \otimes_\C \Z\C(-, - \oplus c)\]
is a $\C$-module and clearly $\Sigma^c$ is an endofunctor on the category of $\C$-modules. Note that from the construction it is clear that also $\Sigma$ can be understood as a functor from $\C$ to the category of endofunctors of $\C$-modules. So in particular, if there is map $0\to c$ in $\C$, then it induces a morphism of $\C$-modules 
\[ A = \Sigma^0 A\longrightarrow \Sigma^c A.\]

Let $\G$ be a braided stability groupoid and $A$ a $\G$-module, then $\Sigma^p A$ is the $p$-shift of $A$. If $A$ is a $U\G$-module, then $\Sigma^p A$ is a $U\G$-module whose underlying $\G$-module is the same as the $p$-shift of the underlying $\G$-module of $A$. We unambiguously also call it the $p$-shift of $A$. Additionally, $0$ is initial in $U\G$, so we obtain a canonical map
\[ A \longrightarrow \Sigma^p A.\]
We will use the shorthand $\Sigma$ for $\Sigma^1$. Note that $\Sigma^p = (\Sigma)^p$, the endofunctor $\Sigma$ iterated $p$ times.

%%%\begin{remark}Concretely, if $\G$ is braided and $A$ is a $\G$-module, then there are isomorphisms of $G_n$-representations 
%%%\[(\Sigma A)_n = \mathrm{Res}^{G_{n+1}}_{1 \times G_n} A_{n+1} \cong  \mathrm{Res}^{G_{n+1}}_{ G_n \times 1} A_{n+1}  .\] 
%%%\end{remark}
%%\begin{definition}
%%%	The same formula defines an endofunctor $\Sigma^p$ on $U\mathcal G$ and on the category of $U\G$-modules. View $\id$ and $\Sigma^p$ as functors $U\G \to U\G$. 
%%Define a natural transformation $\id \to \Sigma^p$ on the category of $U\G$-modules, defined by precomposition with the natural transformation given by
%%\[ 0\oplus n \longrightarrow p\oplus n.\] 
%%\end{definition}
%
%\begin{convention}\label{ConShi}
%Given a suboject $C \subseteq n$ and $\G$-module $A$, we define $\Sigma^C A$ to be $\Sigma^{n_C} A$ as in  \autoref{SubCon}. For $D$ contained in the complement of $C \subseteq n$, we will identify $(\Sigma^{C} A)_D$ with $A_{C \oplus D}$.
%\end{convention}

\begin{definition} Given a $U\G$-module $A$,  define $U\G$-modules \[\qquad \quad  \ker A := \ker(A \to \Sigma A) \qquad \text{and} \]  \[\coker A := \coker(A\to \Sigma A).\] 
\end{definition}

\begin{definition} \label{defpoly}
We say that a $U\G$-module $A$ has \emph{polynomial degree $-\infty$ in ranks $>d$} if $A_n = 0$ for all $n>d$. For $r\ge0$, we say $A$ has \emph{polynomial degree $\le r$ in ranks $>d$} if $(\ker A)_n = 0$ for all $n>d$  and $\coker A$ has polynomial degree $\le r-1$ in ranks $>d-1$. 

We say $A$ has \emph{polynomial degree $\le r$} if it has polynomial degree $\le r$ in all ranks $>-1$.
\end{definition}

\begin{remark}
Note that if $A$ has polynomial degree $-\infty$ in ranks $>d$, then $(\ker A)_n = 0$ for all $n>d$ and $\coker A$ has polynomial degree $-\infty = -\infty -1$ in ranks $>d-1$.

If $A$ has polynomial degree $\le 0$ in ranks $>d$, then we require $\coker A$ has polynomial degree $\le -1$ in ranks $>d-1$. This forces $\coker A$ to be polynomial degree $-\infty$ in ranks $>d-1$.
\end{remark}

\begin{remark}\label{rem:ZXpoly}
Recall that $1$ denotes the braided stability groupoid whose automorphism groups are all trivial. Note that $U1 \subset U\G$ for every stability category $U\G$ by mapping the unique map from $m \to n$ to $\iota_{n-m}\oplus \id_m$. We remark that the notion of polynomiality of a $U\G$-module only depends on the underlying $U1$-module structure.
\end{remark}

\begin{lemma}\label{lem:polysigma}
Let $A$ be a $U\G$-module of polynomial degree $\le r$ in ranks $>d$ and $p\ge 1$. Then $\Sigma^p A$ has polynomial degree $\le r$ in ranks $>d-p$.
\end{lemma}

\begin{proof}
This is clearly true for $r=-\infty$. For $r\ge 0$, it follows by induction because $\coker(\Sigma^pA) \cong \Sigma^p \coker(A)$. \end{proof}

The following lemma appears in a very similar form in \cite[Lem 7.3(a)]{Pa2}, but the notion of polynomial degree in that paper is slightly different than the one used here. 

\begin{lemma}\label{lem:polySES}
Let $A'$, $A$, and $A''$ be $U\G$-modules, where  $A'$ has polynomial degree $\le r$ in ranks $>d$ and $A''$ has polynomial degree $\le r$ in ranks $>d-1$. Assume there are maps $A' \to A \to A''$ such that 
\[0 \to A'_n \to A_n \to A''_n \to 0\]
are short exact sequences for all $n>d$. Then $A$ has polynomial degree $\le r$ in ranks $>d$.
\end{lemma}

\begin{proof}
We prove this by induction over $r$. Let us start with $r = -\infty$. This means that $A'_n \cong A''_n \cong 0$ for all $n>d$ and thus $A_n \cong 0$ for all $n>d$. Let $r\ge 0$ and consider the exact sequence
\[ 0 \to \ker A' \to \ker A \to \ker A'' \to \coker A' \to \coker A \to \coker A'' \to 0\]
coming from the snake lemma for $n>d$. Because $(\ker A')_n \cong (\ker A'')_n \cong 0$ for all $n>d$, $(\ker A)_n \cong 0$ for all $n>d$. Because $(\ker A'')_n \cong 0$ for all $n>d-1$, 
\[ 0 \to \coker A'_n \to \coker A_n \to \coker A''_n \to 0\]
is a short exact sequence for $n>d-1$. We can therefore apply the induction hypothesis, showing that $\coker A$ has polynomial degree $\le r-1$ ($-\infty$ if $r=0$) in ranks $>d-1$. This implies that $A$ has polynomial degree $\le r$ in ranks $>d$.
\end{proof}

\begin{lemma}\label{lem:cokerhighersigma}
Let $A$ be a $U\G$-module of polynomial degree $\le r$ in ranks $>d$ and $p\ge 1$. Then $\ker (A \to \Sigma^p A)_n \cong 0$ for all $n>d$ and $\coker( A \to \Sigma^pA)$ is of polynomial degree $\le r-1$ in ranks $>d-1$.
\end{lemma}

\begin{proof}
If $A_n \to (\Sigma A)_n$ is injective for every $n >d$, then so is the composition $A_n \to (\Sigma^p A)_n$. This proves that $\ker (A \to \Sigma^p A)_n \cong 0$ for all $n>d$.

We will prove the second assertion by induction over $r$. For $r = -\infty$, then \[0 \cong (\Sigma^pA)_n\to \coker( A \to \Sigma^p A)_n\] is surjective for $n>d - p$, and thus $\coker( A \to \Sigma^p A)$ has polynomial degree $-\infty = -\infty -1$ in ranks $>d-1$.

For $r\ge 0$ and $n> d-1$,  $\coker( A \to \Sigma^p A)_n$ has a filtration whose factors are $\coker(\Sigma^iA)_n = \coker( \Sigma^i A \to \Sigma^{i+1} A)_n$, because $A_n \to (\Sigma A)_n$ is injective for $n>d$.  Therefore, $\coker( A \to \Sigma^p A)$ has polynomial degree $\le r-1$ in ranks $>d-1$ using \autoref{lem:polysigma} and \autoref{lem:polySES}. \end{proof}

\section{Polynomial modules and derived representation stability}

\label{SecPolyn}

In this section, we prove quantitative versions of \autoref{derivedRepPolynomials}. These theorems give sufficient conditions for polynomial modules to exhibit derived representation stability.

\subsection{Via central stability complexes}

Our goal is to bound the following quantity:

\begin{definition}
Let $\psi\colon (\Z_{\geq -1})^3  \m \N \cup \{ \infty \}$ be the smallest number such that for any polynomial $U \G$-module $A$ of degree $\leq r$ in ranks $>d$, then 
\[\widetilde H^\G_i(A)_n \cong 0 \text{ for } n> \psi(r,d,i).\]
\end{definition}

\begin{definition}
Let $A$ be $U\G$-module. Let 
\[X(A) = (A \otimes \Z U\G(-_1,-)) \otimes_{U\G} \Z U\G(- \oplus -_2,-_3)\]
be a $U\G^\op \times U\G^\op \times U\G$-module, where we consider $A \otimes \Z U\G(-_1,-)$ as an $U\G$-module using the diagonal action. Using the inclusion $\Delta_{inj,+}\to U\G$, we can consider $X(A)$ as a bi-semisimplicial $U\G$-module with augmentations. We denote the corresponding double chain complex by $X_{*,*}(A)$ such that
\[ X_{p,q}(A)_n = (A \otimes \Z U\G(p+1,-)) \otimes_{U\G} \Z U\G(- \oplus q+1,n).\]
(The index shift is explained in \autoref{rem:CS}.) Let $\bar X_{*,*}(A)$ be the transpose of $X_{*,*}(A)$.
\end{definition}

\begin{proposition}\label{prop:descriptionX}
For fixed $q$, there is a  chain complex isomorphism
\[ X_{*,q}(A) \cong  (A \otimes \widetilde C^\G_*(\Z U\G(0,-))) \otimes_{\UG} \Z\UG(-\oplus (q+1),-).\]
For fixed $p$, there is a  chain complex isomorphism
\[ X_{p,*}(A) \cong \widetilde C^\G_*(\Sigma^{p+1} A) \otimes_{\UG} \Z\UG(-\oplus (p+1),-).\]
\end{proposition}

\begin{proof}
The first isomorphism follows immediately the definition and the fact that the corresponding chain complex to the semisimplicial $\UG$-module $\Z\UG(-,-)$ is precisely $\widetilde C^\G_*(\Z\UG(0,-))$.

The second isomorphism is considerably harder to prove. For simplicity, we replace $p+1$ by $p'$ during the proof. First, we will first prove that
\[\Sigma^{p'} A \otimes_{\UG} \Z\UG(- \oplus p',-) \cong A \otimes \Z\UG(p',-)\]
as $\UG$-modules for fixed $p'$. Let $a_{l+p'} \in A_{l+p'}$ and $f\in U\G( l\oplus p', n)$, then we send
\[ a_{l+p'} \otimes f \longmapsto f(a_{l+p'}) \otimes f \circ (\iota_l \oplus \id_{p'}).\]
Let us first check that this map is well-defined for fixed $n$. To see this, let $a_{l+p'} \in A_{l+p'}$, $h\in U\G(l,m)$ and $f\in U\G( m\oplus p', n)$. We need to show that $(h \oplus \id_{p'})(a_{l+p'}) \otimes f$ is sent to the same thing as $a_{l+p'} \otimes f \circ (h\oplus \id_{p'})$. This is the case, as
\[ f( (h\oplus \id_{p'})(a_{l+p'})) \otimes f \circ (\iota_m \oplus \id_{p'}) =( f \circ (h\oplus \id_{p'}))(a_{l+p'}) \otimes f \circ (h\oplus \id_{p'}) \circ (\iota_l \oplus \id_{p'}).\]
That this map is $U\G$-equivariant is obvious as on both sides it acts by postcomposition on $f$. Next, we will describe the inverse map. So let $a_n \in A_n$, $g\in U\G(p', n)$ and $\hat g \in G_n$ such that $g = \hat g \circ (\iota_{n-p'} \oplus \id_{p'})$. Then, we send
\[ a_n \otimes g \longmapsto {\hat g}^{-1}(a_n) \otimes \hat g.\]
As $\hat g$ only well defined up to precomposition with $(h \oplus \id_{p'})$ for $h \in G_{n-p'}$, we have to check that 
\[ {\hat g}^{-1}(a_n) \otimes \hat g = (\hat g\circ (h \oplus \id_{p'}))^{-1}(a_n) \otimes \hat g \circ (h \oplus \id_{p'}),\]
which is clear from the definition of the coend $\otimes_{\UG}$. Again, this map is clearly $U\G$-equivariant, as $U\G$ acts via postcomposition. The only thing left to check is that these maps are inverses. Let us start with $a_{l+p'} \in A_{l+p'}$ and $f\in U\G( l\oplus p', n)$. In this case, $a_{l+p'} \otimes f$ is sent through both maps to
\[ {\hat g}^{-1}( f(a_{l+p'})) \otimes \hat g,\]
with $\hat g\in G_n$ such that 
\[ f = \hat g \circ (\iota_{n-l-p'} \oplus \id_{l +p'})\]
because then
\[ f \circ (\iota_l \oplus \id_{p'}) =  \hat g \circ (\iota_{n-p'} \oplus \id_{p'}).\]
Therefore, 
\[ {\hat g}^{-1}( f(a_{l+p'})) \otimes \hat g =  (\iota_{n-l-p'} \oplus \id_{l +p'})(a_{l+p'}) \otimes \hat g = a_{l+p'} \otimes \hat g \circ (\iota_{n-l-p'} \oplus \id_{l +p'}) = a_{l+p'} \otimes f. \]
If on the other hand $a_n \in A_n$, $g\in U\G(p', n)$, and $\hat g \in G_n$ such that $g = \hat g \circ (\iota_{n-p'} \oplus \id_{p'})$, then $a_n \otimes g$ is sent through both maps to
\[ \hat g({\hat g}^{-1}(a_n)) \otimes \hat g \circ (\iota_{n-p'} \oplus \id_{p'}) = a_n \otimes g.\]

Using this isomorphism, we conclude that
\begin{multline*} (A \otimes \Z U\G(p',-)) \otimes_{U\G} \Z U\G(- \oplus -,-) \cong \Sigma^{p'} A  \otimes_{\UG} \Z\UG(- \oplus p',-)\otimes_{U\G} \Z U\G(- \oplus -,-) \\
\cong \Sigma^{p'} A  \otimes_{\UG} \Z U\G(- \oplus p' \oplus -,-).\end{multline*}

Next, note that using \autoref{cor:newbraidiso},
\[ \Z U\G(-_1 \oplus p' \oplus -_2,-) \cong \Z\UG(p'\oplus -_2,-) \otimes_{\UG} \Z\UG(-_1 \oplus -,-) \] \[\cong\Z\UG(-_2 \oplus p', -) \otimes_{\UG} \Z\UG(-_1\oplus -, -)\cong \Z U\G(-_1  \oplus -_2\oplus p',-)  \] \[  \cong \Z\UG(-_1 \oplus -_2 , -) \otimes_{\UG} \Z\UG(- \oplus p', -).\]

Since 
\[ \Sigma^{p'} A  \otimes_{\UG} \Z\UG(- \oplus - , -) \cong K(\Sigma^{p'} A),\]
we deduce the asserted isomorphism
\[ X_{p,*} \cong \widetilde C^\G_*(\Sigma^{p'} A) \otimes_{\UG} \Z\UG(-\oplus p',-).\qedhere\]

%
%The face maps define chain complexes in both directions. It is enough to prove that 
%\[ \xymatrix{
%X_{p,q}(A)_n \ar[r]^{\delta_i}\ar[d]^{\bar\delta_j} & X_{p,q-1}(A)_n \ar[d]^{\bar\delta_j} \\
%X_{p-1,q}(A)_n \ar[r]^{\delta_i} & X_{p-1,q-1}(A)_n
%}\]
%commutes to show that it is a double complex. In particular, the summands are indexed by pairs of maps
%\[ \xymatrix{
%(f,g) \ar[r]\ar[d] &( (C_g \subset C_g') \circ f, g\circ(\id_i \oplus \iota_1 \oplus \id_{q-i}) ) \ar[d]\\
%(f\circ (\id_j \oplus \iota_1 \oplus \id_{p-j}),g) \ar[r] & ( (C_g \subset C_g') \circ f (\id_j \oplus \iota_1 \oplus \id_{p-j}), g\circ(\id_i \oplus \iota_1 \oplus \id_{q-i}) )
%}\]
%and the maps on the summands
%\[ \xymatrix{
%A_{C_g} \ar[r]\ar[d] & A_{C_{g'}}\ar[d]\\
%A_{C_{g}} \ar[r] & A_{C_{g'}}
%}\]
%commute.
%
%The first isomorphism of chain complexes is immediate from the definition.
%
%For the second isomorphism we use the bijection from \autoref{prop:swaporder} to change the order of the sums:
%\begin{multline*} X_{p,q}(A)_n  \cong \bigoplus_{g \in  U\G(q+1,n)} \bigoplus_{f \in  U\G(p+1,C_g)}A_{C_{g}} \cong  \bigoplus_{\bar f \in  U\G(p+1,n)} \bigoplus_{\bar g \in  U\G(q+1,C_f)} A_{ C_{\bar g} \oplus \im \bar f} \\\bigoplus_{\bar f \in U\G(p+1,n)} \widetilde C^\G_*(\Sigma^{\im \bar f} A)_{C_{\bar f}} \cong \bigoplus_{\bar f \in U\G(p+1,n)} \widetilde C^\G_*(\Sigma^{p+1} A)_{C_{\bar f}}.\end{multline*}
%
%\todo[Peter]{Need to show that this iso commutes with the face maps.}
\end{proof}

\begin{proposition}
Assume $U\G$ satisfies \con{H3($k,a$)}. Let $E^s_{p,q}(A)_n$ denote the spectral sequence associated to $X_{*,*}(A)_n$ and let $\bar E^s_{p,q}(A)_n$ denote the spectral sequence associated to the transpose $\bar X_{*,*}(A)_n$. Then $\bar E^1_{p,q}(A)_n \cong 0$ for $p+q < \frac{n-a}k -1$. In particular, $E^\infty_{p,q}(A)_n \cong 0$ for $p+q < \frac{n-a}k-1$. 
%   $p\cdot k+a +q+1 \leq n$
\label{EinfinityVanish}
\end{proposition} 

\begin{proof} We have that 
\[ \bar E^0_{p,q}(A)_n \cong  (A \otimes \widetilde C^\G_q(\Z U\G(0,-))) \otimes_{\UG} \Z\UG(-\oplus (p+1),-).\]
The homology $\widetilde H^\G_q(\Z U\G(0,-))_{n-p-1}$ vanishes for $n-p-1 > k\cdot q +a$ since $\G$ satisfies \con{H3($k,a$)}. Therefore, so does the homology of $(A \otimes \widetilde C^\G_*(\Z U\G(0,-)))$ by the K\"unneth spectral sequence. The functor $- \otimes_{\UG} \Z\UG(-\oplus (p+1),-)$ is exact, so we conclude that $E^1_{p,q}(A)_n \cong 0$ for $n-p-1 > k\cdot q +a$. In particular, they vanish when $p+q < \frac{n-a}k-1$, because then
\[p+kq = k(p+q) - (k-1)p <  k\left(\frac{n-a}k-1\right) + (k-1) = n-a -1.\] The second claim follows from the fact that the two spectral sequences both converge to the homology of the total complex of $X_{*,*}(A)_n$. 
\end{proof}

\begin{definition}
Let $A$ be a $U\G$-module. Define 
\[Y(A) = A \otimes_{\UG} \Z\UG(-\oplus -_1, -) \otimes_{\UG} \Z\UG(- \oplus -_2, -_3)\]
and consider it as a augmented simplicial-augmented simplicial-$U\G$-module. Let $Y_{*,*}(A)$ be the associated double complex.
\end{definition}

It is immediate that $Y_{*,*}(A) = \widetilde C^\G_*(\widetilde C^\G_*(A))$.

\begin{proposition}
Let $A$ be $U\G$-module and let $\widehat E^s_{p,q}(A)_n$ be the double complex spectral sequence associated to $Y_{*,*}(A)_n$. Then $d_1\colon \widehat E^1_{0,q}(A)_n \m \widehat E^1_{-1,q}(A)_n$ is the zero map. 
\end{proposition}

\begin{proof}
Note that 
\[\widehat E^0_{p,q}(A)_S \cong A \otimes_{\UG} \Z\UG(- \oplus (p+1) \oplus (q+1), n).\]
%\widetilde C^\G_p( \widetilde C^\G_q(A))_n  \cong \widetilde C^\G_q( \widetilde C^\G_p(A))_n  \cong \bigoplus_{g \in  U\G(q+1,n)} \left(\bigoplus_{f \in  U\G(p+1,C_g)} A_{C_f} \right).\] 
Clearly, there is an isomorphism of abelian groups 
\[\varphi \colon  \widehat E_{0,q}^0  \cong A \otimes_{\UG} \Z\UG(- \oplus 1 \oplus (q+1), n) = A \otimes_{\UG} \Z\UG(- \oplus 0 \oplus (q+2), n)  \cong\widehat E_{-1,q+1}^0.\] 
%For $(g,f)$ with $g\in  U\G(q+1,n)$ and $f\in  U\G(1,C_g)$, let $h\in  U\G(q+2,n)$ be such that $h \circ(\id_{q+1} \oplus \iota_1) = g$ and $h\circ(\iota_{q+1} \oplus \id_1) = f$. Then $C_h = C_f$ and $\psi$ is given by $A_{C_f} \to A_{C_h}$.

We will prove that $\varphi$ is a chain homotopy from the differential $d^h\colon  \widehat E_{0,*}^0 \to \widehat E_{-1,*}^0$ to the zero map. Recall that
\begin{align*}
d^h\colon  \widehat E_{0,q}^0 \cong A \otimes_{\UG} \Z\UG(- \oplus 1 \oplus (q+1), n)  &\longrightarrow A \otimes_{\UG} \Z\UG(- \oplus 0 \oplus (q+1), n) \cong \widehat E_{-1,q}^0\\
a_m \otimes (f\colon m \oplus 1 \oplus (q+1) \to n) & \longmapsto a_m \otimes f \circ (\id_m \oplus \iota_1 \oplus \id_{q+1})
\end{align*}
and $d^v = d_0-d_1+\dots+ (-1)^q d_q \colon E^0_{p,q} \to E^0_{p,q-1}$ is given by
\begin{align*}
d_i \colon \widehat E_{p,q}^0 \cong A \otimes_{\UG} \Z\UG(- \oplus (p+1) \oplus (q+1), n)  &\longrightarrow A \otimes_{\UG} \Z\UG(- \oplus (p+1) \oplus q, n) \cong \widehat E_{p,q-1}^0\\
a_m \otimes (f \colon m \oplus (p+1) \oplus (q+1) \to n) &\longmapsto a_m \otimes f \circ (\id_m \oplus \id_{p+1} \oplus \id_{i} \oplus \iota_1 \oplus \id_{q-i}).
\end{align*}

From these definitions it is clear that
\[\xymatrix{
\widehat E^0_{-1,q+1} \ar[d]^{d_{i+1}} &\\
\widehat E^0_{-1,q} &\widehat E^0_{0,q} \ar[ul]_{\varphi} \ar[d]^{d_i} \\
 &\widehat E^0_{0,q-1} \ar[ul]^{\varphi}
}\]
commutes for $0\le i\le q$. Further the diagram
\[\xymatrix{
\widehat E^0_{-1,q+1} \ar[d]^{d_{0}} &\\
\widehat E^0_{-1,q} &\widehat E^0_{0,q} \ar[ul]_{\varphi} \ar[l]^{d^h}
}\]
is commutative.
Thus $d^h = \varphi d^v + d^v\varphi$.
\end{proof}

We will construct a comparison map $F\colon Y(A) \to X(A)$ using the following lemma.

\begin{lemma}\label{lem:mapconstr}
There is a map
\begin{align*}
A \otimes_{\UG} \Z\UG(- \oplus -, -) & \longrightarrow A \otimes \Z\UG(-,-)\\
a_m \otimes f\colon m \oplus p \to n & \longmapsto f( (\id_m \oplus \iota_p)(a_m) ) \otimes f \circ (\iota_m \oplus \id_p)
\end{align*}
of $\UG^\op\times \UG$-modules.
\end{lemma}

\begin{proof}
We first observe that the map is well defined on objects. For that, let $a_l\in A_l$, $h \colon l \to m$ and $f\colon m \oplus p \to n$. Then
\[  f( (\id_m \oplus \iota_p)( h(a_l)) ) \otimes f \circ (\iota_m \oplus \id_p) =  (f\circ (h \oplus \id_p)) (\id_l \oplus \iota_p)(a_l) ) \otimes (f\circ (h \oplus \id_p)) \circ (\iota_l \oplus \id_p),\]
which proves well-definedness.

The map is clearly $\UG$-equivariant. To check that it is also $\UG^\op$-equivariant, let $a_m\in A_m$, $f\colon m \oplus p \to n$, and $h \colon p' \to p$. Then
\[ f( (\id_m \oplus \iota_p)(a_m) ) \otimes f \circ (\iota_m \oplus \id_p) \circ h = (f \circ (\id_m \oplus h))( (\id_m \oplus \iota_{p'})(a_m) ) \otimes (f \circ (\id_m \oplus h)) \circ (\iota_m \oplus \id_p),\]
which proves $\UG^\op$-equivariance.
\end{proof}

\begin{definition}\label{def:comparisonmap}
Using the map of \autoref{lem:mapconstr}, we define the comparison map
\[ F \colon Y(A) = A \otimes_{\UG} \Z\UG(-\oplus -_1, -) \otimes_{\UG} \Z\UG(- \oplus -_2, -_3) \longrightarrow (A \otimes \Z U\G(-_1,-)) \otimes_{U\G} \Z U\G(- \oplus -_2,-_3) = X(A).\]
\end{definition}

%\begin{lemma}\label{lem:Fisshift}
%For fixed $p$, 
%\[ F \colon Y_{p,*}(A) \cong \widetilde C^\G_*(A) \otimes_{\UG} \Z\UG(-\oplus (p+1),-) \longrightarrow \widetilde C^\G_*(\Sigma^{p+1}A) \otimes_{\UG} \Z\UG(-\oplus (p+1),-) \cong X_{p,*}(A)\]
%is induced by the shift map $A \to \Sigma^{p+1} A$.
%\end{lemma}
%
%\begin{proof}
%
%\end{proof}

This map of double complexes induces a map of spectral sequences. For polynomial $U\G$-modules, it turns out to be an isomorphism on the $E^1$-page in a range.

\begin{proposition}\label{prop:EhatE}
 Let $A$ be a polynomial $U\G$-module of degree $\le r$ in ranks $>d$. Then 
 \[F\colon \widehat E^1_{p,q}(A)_n  \m  E^1_{p,q}(A)_n\]
 is an isomorphism for 
 \[n> \max(d+ p+q+2, \psi(r-1,d-1,q+1)+p +1) \] 
 and a surjection for 
 \[n> \max( d+p+q+1,\psi(r-1,d-1,q)+p+1).\]
\end{proposition}

\begin{proof}
We see that the map from \autoref{def:comparisonmap} 
\[ F\colon \widehat E^1_{p,q} \cong \widetilde H^\G_q(A)\otimes_{\UG} \Z\UG(-\oplus (p+1),n)  \longrightarrow \widetilde H^\G_q(\Sigma^{p+1}A)\otimes_{\UG} \Z\UG(-\oplus (p+1),-)  \cong E^1_{p,q}\]
is induced by the map
\[\widetilde H^\G_q(A)_{n-p-1}   \longrightarrow \widetilde H^\G_q(\Sigma^{p+1}A)_{n-p-1}. \]
We will prove that this is an isomorphism or a surjection in the asserted ranges.

Let $B$ and $D$ be the kernel and cokernel of $A \to \Sigma^{p+1} A$, respectively. Since $\widetilde C^\G_i$ is an exact functor and $\widetilde C^\G_*$ is functorial with respect to $U \G$-morphisms, we have an exact sequence of chain complexes: 
\[0 \m \widetilde C^\G_*(B)_{n-p-1} \m \widetilde C^\G_*(A)_{n-p-1} \m \widetilde  C^\G_*(\Sigma^{p+1}A)_{n-p-1} \m \widetilde C^\G_*(D)_{n-p-1} \m 0.\]
Let $E$ be the cokernel of $B\to A$, which is incidentally the kernel of $\Sigma^{p+1} A \to D$. As explained in \autoref{lem:cokerhighersigma}, $B_n \cong 0$ for $n>d$. That means that $\widetilde C^\G_q(B)_{n-p-1} \cong 0$ for $n>d+p+q+2$. From the long exact sequence associated to: 
\[ 0 \longrightarrow \widetilde C^\G_*(B)_{n-p-1} \longrightarrow \widetilde C^\G_*(A)_{n-p-1} \longrightarrow \widetilde C^\G_*(E)_{n-p-1} \longrightarrow 0,\]
we conclude that \[ \widetilde H^\G_q(A)_{n-p-1} \longrightarrow \widetilde H^\G_q(E)_{n-p-1},\]
is an isomorphism for $n>d+p+q+2$ and a surjection for $n> d+p+q+1$. Further, $D$ has polynomial degree $\le r-1$ in ranks $>d-1$ by \autoref{lem:cokerhighersigma}. Thus $\widetilde H^\G_q(D)_{n-p-1} \cong 0$ if $n> \psi(r-1,d-1,q)+p+1$. From the long exact sequence associated to:
\[ 0 \longrightarrow \widetilde C^\G_*(E)_{n-p-1} \longrightarrow \widetilde C^\G_*(\Sigma^{p+1}A)_{n-p-1} \longrightarrow \widetilde C^\G_*(D)_{n-p-1} \longrightarrow 0,\]
we deduce that \[ \widetilde H^\G_q(E)_{n-p-1} \longrightarrow \widetilde H^\G_q(\Sigma^{p+1}A)_{n-p-1}\]
is an isomorphism for $n> \psi(r-1,d-1,q+1)+p+1$ and a surjection for $n> \psi(r-1,d-1,q)+p+1$. 

Therefore, the composition $ \widetilde H^\G_q(A)_{n-p-1} \to \widetilde H^\G_q(E)_{n-p-1} \to  \widetilde H^\G_q(\Sigma^{p+1}A)_{n-p-1}$ is an isomorphism or a surjection in the given ranges.
\end{proof}

\begin{corollary} \label{corE2central}
 Let $A$ be a polynomial $U\G$-module of degree $\le r$ in ranks $>d$. Then $E^2_{-1,i}(A)_n \cong \widetilde H^\G_i(A)_n$ if 
\[ n> \max( d+i+1, \psi(r-1,d-1,i)+1).\] 
\end{corollary}

\begin{proof}
In this case $F\colon \widehat E^1_{0,i}(A)_n \m E^1_{0,i}(A)_n$ is surjective and thus  $d_1\colon  E^1_{0,i}(A)_n \m  E^1_{-1,i}(A)_n$ is zero because $d_1\colon \widehat E^1_{0,i}(A)_n \m \widehat E^1_{-1,i}(A)_n$ is the zero map. Therefore, $E^2_{-1,i}(A)_n \cong E^1_{-1,i}(A)_n \cong \widetilde H^\G_i(A)_n.$
\end{proof}

We now prove a quantitative version of the \con{H3} portion of \autoref{derivedRepPolynomials}. Note that by \autoref{thm:res of finite type} and \autoref{prop:res of finite type}, vanishing of central stability homology implies derived representation stability. Thus, it suffices to give a vanishing line for the central stability homology of poloynomial $U\G$-modules.

\begin{theorem} \label{PolynomialsCentral}
Let $A$ be a polynomial $U\G$-module of degree $\leq r$ in ranks $>d$. If $U\G$ satisfies \con{H3($k,a$)} with $k\ge 2$, then $\widetilde H_i^\G(A)_n \cong 0$ for $n >\max(d+i+1,ki+a+r)$.
\end{theorem}

\begin{proof}
We will prove the theorem by nested induction---the first over $r$ and the second over $i$. 

We first give a proof for $r=-\infty$. We have that $\widetilde H^\G_i(A)_n \cong 0$ for all $n>d+i+1$ since central stability chains also vanish in that range.

Let us now proceed with $r\ge 0$ and assume the theorem is true for all $U\G$-modules $A$ of polynomial degree $\le s< r$ in ranks $>d$ and all homological degrees. Fix  $i\ge -1$ and assume that the theorem is true for all  $U\G$-modules $A$ of polynomial degree $\le r$ in ranks $>d$ and all homological degrees $j<i$. In particular, 
\[ \max(d+j+1,kj+a+s) \ge \psi(s,d,j)\]
if $s<r$, or $s=r$ and $j<i$.

Let $A$ be a polynomial $U\G$-module of degree $\le r$ in ranks $>d$. We will now prove that $\widetilde H_i^\G(A)_n \cong 0$ for all $n> \max(d+i+1,ki+a+r)$. 
By \autoref{corE2central}, $\widetilde H^\G_i(A)_n \cong E^2_{-1,i}(A)_n$ if 
\[n>  \max(d+i+1,ki+a+r)\]
because
\begin{multline*} \max(d+i+1,ki+a+r) = \max (d+i+1, \max((d-1)+i+1+1, ki+a+(r-1)+1) )\\ \ge \max( d+i+1, \psi(r-1,d-1,i)+1) .\end{multline*}
We will prove that $E^2_{-1,i}(A)_n = E^{\infty}_{-1,i}(A)_n$ in the asserted range. To do this, we will look at $E^1_{p,q}(A)_n$ for $p+q = i$ and $q<i$ and show that these groups vanish. Observe that
\[E^1_{p,q}(A)_n \cong  \widetilde H^\G_q(\Sigma^{p+1} A) \otimes_{\UG} \Z\UG(- \oplus (p+1),n).\]
Because $\Sigma^{p+1} A$ has polynomial degree $\le r$ in ranks $>d-p-1$,  \autoref{lem:polysigma} implies that $E^1_{p,q}(A)_n \cong 0$ for 
\[n>\max(d+i+1,ki+a+r) \stackrel{k\ge2}{\ge} \max(d-p-1+q +1, kq+a+r)+p+1\ge \psi(r,d-p-1,q) + p+1.\] We finish the proof by invoking \autoref{EinfinityVanish} that says that $E^{\infty}_{-1,i}(A)_n$ vanishes for \[n> \max(d+i+1,ki+a+r) \ge ki+ a. \qedhere\] \end{proof}

\begin{remark} \label{H1H2}
\con{H3($1,a$)} implies \con{H3($2,a+1$)}. Therefore, if $k=1$ and if $A$ is a polynomial $U\G$-module of degree $\leq r$ in ranks $>d$, we conclude  $\widetilde H^\G_i(A)_n \cong 0$ for $n >\max(d+i+1,2i+a+1+r)$.
\end{remark}

%\begin{remark}
%For several families of groups, \autoref{PolynomialsCentral} was previously known. For $U\fS$ it follows from \cite[Theorem 3.30]{MPW}. For $U\GL(R)$ and $U\Sp(R)$ for $R$ a PID, this is \cite[Corollaries 3.34 and 3.36]{MPW}. However, \autoref{PolynomialsCentral} proves a better range. 

%This allows one to improve the ranges in \cite[Theorem A, Theorem B, and Theorem C]{MPW} as well as allows Theorem C to apply to all rings with finite Bass stable rank. In personal communication, Andrew Putman informed us that he has an argument improving the ranges in \cite[Theorem A and Theorem B]{MPW}.

%\end{remark}

\subsection{Via the Koszul resolution}

We now prove an improved stable range when $\G$ satisfies the standard connectivity assumptions. The arguments will be similar to those of the previous subsection.

%\begin{definition}
%Let $A$ be $U\G$-module. Let $P_{p,q}(A)=\bigoplus_{V \oplus W=n, |V|=p} K_q(T^V A)_W \otimes \SSt^{\G}_W$.
%\end{definition}

\begin{definition}
Let $A$ be $U\G$-module. Define the double complex $P_{*,*}(A) = K_*(A \otimes K_*(\uZ))_n$, where $U\G$ acts on $A \otimes K_*(\uZ)$ diagonally and $P_{p,q} = K_q(A \otimes K_p(\uZ))$. Let $F^r_{pq}$ denote the spectral sequence associated to $P_{*,*}$.
\end{definition}

We start the discussion of this section by showing that the spectral sequence of the double complex $P_{*,*}$ converge to zero in a range.

\begin{lemma}\label{lem:Fvanish}
$F^\infty_{pq} \cong 0$ for $p+q<n$.
\end{lemma}

\begin{proof}
Let $\bar F^r_{p,q}$ be the spectral sequence associated to the transpose of $P_{*,*}$. Then
\[ \bar F^0_{pq} \cong K_p(A \otimes K_q(\uZ))_n.\]
From the standard connectivity assumption, the homology of $K_*(\uZ)_{n-p}$ vanishes in degrees $*<n-p$. Using the K\"unneth spectral sequence, this implies that the homology of $A_{n-p} \otimes K_*(\uZ)_{n-p}$ vanishes in the same degrees $*<n-p$. Because $K_p$ is exact,
\[ F^1_{pq} \cong H_q(K_p(A \otimes K_*(\uZ))_n) \cong \Ind^{G_n}_{G_{n-p} \times G_p} H_q(A_{n-p} \otimes K_*(\uZ)_{n-p}) \otimes \SSt_p \cong 0\]
for $p+q < n$.
\end{proof}

We proceed by comparing this double complex to the following double complex.

\begin{definition}
Define the double complex $Q_{*,*}(A) = K_*(K_*(A))$, where $Q_{p,q} = K_q(K_p(A))$. Let $\widehat F^r_{pq}$ denote the spectral sequence associated to $Q_{*,*}$.
\end{definition}

Note that this double complex is symmetric as proved in the following lemma.

\begin{lemma}\label{lem:Kcom}
The braiding induces an isomorphism of double complexes swapping the two complex gradings of $Q_{*,*}(A)$. In particular,
\[ K_p(K_q(A)) \cong K_q(K_p(A)).\]
\end{lemma}

\begin{proof}
We see that 
\[ K_p(K_q(A)) \cong (A \oast \SSt_q) \oast \SSt_p \cong (A \oast \SSt_p) \oast \SSt_q \cong  K_q(K_p(A))\]
because $ \SSt_q \oast \SSt_p \cong \SSt_p \oast \SSt_q$ by \autoref{lem:oastbraiding}. As the boundary maps are defined using the right $\uZ$-action on $A$, \autoref{prop:leftrightuZmod} shows that the braiding induces an isomorphism of chain complexes.
\end{proof}

We now uses these lemmas and corollary to study the spectral sequence $\widehat F^s_{p,q}(A)$. 

\begin{proposition} \label{specCol} Let $A$ be a $U\G$-module and assume $\G$ satisfies the standard connectivity assumptions.
The spectral sequence  $\widehat F^*_{*,*}(A)_n$ collapses at the first page. For $s \geq 1$, we have that $\widehat F^s_{p,q}(A) \cong K_p(\Tor_q^{\uZ}(\Z,A))$. In particular, for $s\geq 1$, $\widehat F^s_{0,q}(A) \cong \Tor_q^{\uZ}(\Z,A)$.

%For $s \geq 1$, we have that $\widehat F^s_{0,q}(A)_n \cong \Tor^\uZ_{q}(\Z,A)$. For $s \ge 2$ and $p>0$, we have $\widehat F^s_{p,q}(A)_n \cong 0$, and $\widehat F^1_{1,q}(A)_n \to \widehat F^1_{0,q}(A)_n$ is the zero map.
\end{proposition}

\begin{proof}
It is clear that $\widehat F^1_{p,q}(A) \cong K_p(\Tor_q^{\uZ}(\Z,A))$. By \autoref{cor:Tor(A,Z)}, all non-isomorphisms of $\UG$ act on $\Tor_q^{\uZ}(\Z,A)$ by zero. From Equation \eqref{Kd} for the differentials in the Koszul complex, we conclude that the $d^1$ differentials are zero.
%
%For $s=1$, it is clear that $\widehat F^s_{p,q}(A) \cong K_p(\Tor_q^{\uZ}(\Z,A))$. By \autoref{cor:Tor(A,Z)}, $d_1$ vanishes. 
To finish the proof, it suffices to show the differentials $d^i$ vanish for $i>1$.

The quasi-isomorphism \[ K_*(A) \m B_*(A,\uZ_+,\Z)\] induces a quasi-isomorphism 
\[Q_{*,*}(A) = K_*(K_*(A)) \m B_*(B_*(A,\uZ_+,\Z),\uZ_+,\Z).\] By \autoref{changeOrder}, there is a zig-zag of quasi-isomorphisms of double complex maps between  $B_*(B_*(A,\uZ_+,\Z),\uZ_+,\Z)$ and $B_*(B_*(\Z,\uZ_+,A),\uZ_+,\Z)$. A zig-zag of quasi-isomorphisms of double complex maps induces an isomorphism on double complex spectral sequences starting at the second page.  Thus, it suffices to show the higher differentials vanish for the double complex spectral sequences associated to $B_*(B_*(\Z,\uZ_+,A),\uZ_+,\Z)$. Note that $B_*(\Z,\uZ_+,A)$ is a $U\G$-module where all non-isomorphisms act via zero. This means that $B_*(B_*(\Z,\uZ_+,A),\uZ_+,\Z)$ is isomorphic to $B_*(\Z,\uZ,A) \oast_\G   B_*(\Z,\uZ,\Z))$. The double complex spectral sequence associated to $B_*(\Z,\uZ,A) \oast_\G   B_*(\Z,\uZ,\Z))$ is a K\"{u}nneth spectral sequence and hence collapses at the second page since the terms of the right tensor factor are free as abelian groups.
\end{proof}

%The spectral sequence $\widehat F^s_{p,q}(A)$ agrees with a Grothendieck spectral sequence where the two functors are the same functor. Thus it collapses at the second page and is concentrated on the leftmost column starting at the first page. A straight forward calculation also shows that
%\[ \widehat F^1_{0,q}(A)_n \cong K_0(\Tor_q^{\uZ}(A,Z))_n \cong \Tor_q^{\uZ}(A,Z)_n.\]
%In particular, the $d^1$ differential into $\widehat F^1_{0,q}(A)_n$ has to be zero.
%
%Let $'F^s_{p,q}(A)_n$ be the spectral sequence defined analagously to  $\widehat F_{p,q}(A)_n$ except with $\uZ$ replacing $\uZ_+$ in all of the bar constructions. Since reduced and unreduced bar constructions are quasi-equivalent, $'F^s_{p,q}(A)_n$ and $ \widehat F^s_{p,q}(A)_n$ agree for $s \geq 2$ and hence it suffices to prove the claim for $'F^s_{p,q}(A)_n$.
%
%We have that $'F^1_{p,q}(A)_n=\Tor_{q+1}^\Z(A_+,\Z^{\otimes_\G (p+1)} \otimes_\G \Z  )_n$. For $p \geq 0$, $\uZ^{\otimes_\G (p+1)} \otimes_\G \Z$ is free as a $\uZ$-module and so $'F^1_{p,q}(A)_n \cong 0$ for $p \geq 0$. Thus the spectral sequence collases at the first page and so  $'F^s_{-1,q}(A)_n \cong \Tor^\uZ_{q+1}(\Z,A_+)$ for $s \geq 1$.

We will now construct a map of double complex $Q_{*,*} \to P_{*,*}$. This map is induced by the map of the following lemma.

\begin{lemma}
There is map
\begin{align*}
K_p(A) \cong (A \boxtimes \SSt_p) \otimes_{\G\times\G} \Z\G(- \oplus - , -) &\longrightarrow A \otimes (\uZ \boxtimes \SSt_p) \otimes_{\G\times\G} \Z\G(- \oplus - , -) \cong A \otimes K_p(\uZ)\\
(a_{n-p} \otimes w_p) \otimes g_n &\longmapsto g_n (\id_{n-p} \oplus p) (a_{n-p}) \otimes (1 \otimes w_p) \otimes g_n
\end{align*}
of $U\G$-chain complexes.
\end{lemma}

\begin{proof}
We first check that this map is well defined. Let $h_{n-p} \in G_{n-p}$ and $h_p \in G_p$. Then 
\[ (a_{n-p} \otimes w_p) \otimes g_n\circ  (h_{n-p} \oplus h_p) = (h_{n-p}(a_{n-p}) \otimes h_p(w_p)) \otimes g_n\]
is sent to
\[ g_n (h_{n-p} \oplus h_p) (\id_{n-p} \oplus p) (a_{n-p}) \otimes (1 \otimes w_p) \otimes g_n\circ (h_{n-p} \oplus h_p) = g_n (h_{n-p} \oplus \id_p ) (\id_{n-p} \oplus p) (a_{n-p}) \otimes (1 \otimes h_p(w_p)) \otimes g_n.\] 

The $U\G$-equivariance is easy to see.

To see that this is a chain map, we recall that $\SSt_p \subset \Z_1^{\oast p}$ with $\Z_1$ the $\G$-module $\Z$ in degree $0$ and zero in all other degrees. The differential is given by
\[ A \oast \Z_1 \oast \Z_1^{\oast p-1} \stackrel{m \oast \id}{\longrightarrow} A \oast \Z_1^{\oast p-1}.\]
The above map is the restriction of the map
\begin{align*}
A \oast  \Z_1^{\oast p} & \longrightarrow A \otimes (\uZ \oast \Z_1^{\oast p})\\
a_{n-p} \otimes w_p \otimes g_n &\longmapsto g_n (\id_{n-p} \oplus \iota_p)(a_{n-p}) \otimes ( 1 \otimes w_p \otimes g_n).
\end{align*}
This gives a commutative diagram
\[ \xymatrix{
A \oast  \Z_1^{\oast p}  \ar[r]\ar[d] & A \otimes (\uZ \oast \Z_1^{\oast p})\ar[d]\\
A \oast  \Z_1^{\oast p-1} \ar[r] &  A \otimes (\uZ \oast \Z_1^{\oast p-1})
}\]
that restricts to the assertion on $K_*(A) \to A \otimes K_*(\uZ)$.
\end{proof}

We see that this map induces a map of double complexes
\[ Q_{pq} \cong K_q(K_p(A)) \longrightarrow K_q( A \otimes K_p(\uZ)).\]
We want to prove that this comparison map induces an isomorphism on homology in a range.

% H_*(coker(Q_p* -> P_p*)) = 0 in a range so hat F^1_pq-> F^1_pq iso/epi in a range

% hat F^1 collapses

\begin{definition}
Let $\eta\colon (\Z_{\ge -1})^3  \m \N \cup \{ \infty \}$ be the smallest number such that for any polynomial $U \G$-module $A$ of degree $\leq r$ in ranks $>d$, then 
\[ \Tor^\uZ_i(A,\Z)_n \cong 0 \text{ for } n> \eta(r,d,i).\]
\end{definition}

%A proof identical to \autoref{prop:EhatE} gives the following.
The proof of the following proposition is modeled after the proof of \autoref{prop:EhatE}. 
\begin{proposition}\label{prop:FhatF}
 Let $A$ be a polynomial $U\G$-module of degree $\le r$ in ranks $>d$. Then 
 \[G\colon \widehat F^1_{p,q}(A)_n  \m  F^1_{p,q}(A)_n\]
is an isomorphism when 
\[n> \max(d+ p+q, \eta(r-1,d-1,q+1)+p) \] 
and a surjection when 
\[n> \max( d+p+q-1,\eta(r-1,d-1,q)+p).\]
\end{proposition}

\begin{proof}
Let us first consider 
\[ \widehat F^0_{p,q}(A)_n = K_q(K_p(A))_n \cong \Ind^{G_n}_{G_{n-p-q} \times G_p \times G_q} A_{n-p-q} \boxtimes \SSt_p \boxtimes \SSt_q.\]
Using the braiding of $\G$, we see that this is isomorphic to 
\[ K_p(K_q(A))_n \cong \Ind^{G_n}_{G_{n-p} \times G_p} \bigg(\Ind^{G_{n-p}}_{G_{n-p-q} \times G_q}A_{n-p-q} \boxtimes \SSt_q\bigg) \boxtimes \SSt_p.\]
There is a noncanonical isomorphism from the above group to
\[ \bigoplus_{G_n/G_{n-p} \times G_p} \SSt_p \otimes \bigg(\Ind^{G_{n-p}}_{G_{n-p-q} \times G_q}A_{n-p-q} \boxtimes \SSt_q\bigg).\]
Observe that the boundary maps in $q$-direction restrict to 
\[ \Ind^{G_{n-p}}_{G_{n-p-q} \times G_q}A_{n-p-q} \boxtimes \SSt_q\]
(with varying $q$) because $p$ stays fixed.

For $F^0_{p,*}$, there is a similar decomposition. First, note that
\[ F^0_{p,q} \cong \Ind^{G_n}_{G_{n-p-q} \times G_p \times G_q} \bigg(\Res^{G_{n-q}}_{G_{n-p-q} \times G_p} A_{n-q} \otimes ( \Z \boxtimes \SSt_p)\bigg) \boxtimes \SSt_q.\]
Using the braiding, we obtain an inclusion of $G_{n-p-q} \times G_p \times G_q$ into $G_{n-p} \times G_p$ and the induction factors as
\[ \Ind^{G_n}_{G_{n-p} \times G_p} \Ind^{G_{n-p}\times G_p}_{G_{n-p-q} \times G_p \times G_q} \bigg(\Res^{G_{n-q}}_{G_{n-p-q} \times G_p} A_{n-q} \otimes ( \Z \boxtimes \SSt_p)\bigg) \boxtimes \SSt_q.\]
We now restrict the inner induction to $G_{n-p}$ and get
\begin{multline*} \Res^{G_{n-p} \times G_p}_{G_{n-p}}  \Ind^{G_{n-p}\times G_p}_{G_{n-p-q} \times G_p \times G_q} \bigg(\Res^{G_{n-q}}_{G_{n-p-q} \times G_p} A_{n-q} \otimes ( \Z \boxtimes \SSt_p)\bigg) \boxtimes \SSt_q \\\cong \SSt_p \otimes \Ind^{G_{n-p}}_{G_{n-p-q} \times G_q} \Res^{G_{n-q}}_{G_{n-p-q}} A_{n-q} \boxtimes \SSt_q\end{multline*}
using Mackey's double coset formula.
Thus there is a noncanonical isomorphism
\[ F^0_{p,q} \cong  \bigoplus_{G_n/G_{n-p} \times G_p} \SSt_p \otimes \bigg( \Ind^{G_{n-p}}_{G_{n-p-q} \times G_q} \Res^{G_{n-q}}_{G_{n-p-q}} A_{n-q} \boxtimes \SSt_q\bigg)\]
and again, the boundary maps in $q$-direction restrict to 
\[  \Ind^{G_{n-p}}_{G_{n-p-q} \times G_q} \Res^{G_{n-q}}_{G_{n-p-q}} A_{n-q} \boxtimes \SSt_q\]
(with varying $q$) because $p$ stays fixed.

Observe that the map of double complexes $G$ restricts to a map of chain complexes
\[ K_q(A)_{n-p} \cong \Ind^{G_{n-p}}_{G_{n-p-q} \times G_q}A_{n-p-q} \boxtimes \SSt_q \longrightarrow \Ind^{G_{n-p}}_{G_{n-p-q} \times G_q} \Res^{G_{n-q}}_{G_{n-p-q}} A_{n-q} \boxtimes \SSt_q \cong K_q(\Sigma^p A)_{n-p}\]
that is induced by the shift map $A \to \Sigma^p A$.

Let $B$ and $C$ denote the kernel and cokernel of $A \to \Sigma^p A$, respectively. 
If we show that 
%Because $\SSt_p$ is free abelian, it is to show that 
\[ \Tor^\uZ_q(B, \Z)_{n-p} \cong  \Tor^\uZ_q(C, \Z)_{n-p} \cong 0\]
in a range, then $K_q(A)_{n-p} \to K_q(\Sigma^p A)_{n-p}$ is a quasi-isomorphism in a range and thus 
\[ \bigoplus_{G_n/G_{n-p} \times G_p} \SSt_p \otimes \bigg( \Ind^{G_{n-p}}_{G_{n-p-q} \times G_q} \Res^{G_{n-q}}_{G_{n-p-q}} A_{n-q} \boxtimes \SSt_q\bigg) \longrightarrow \bigoplus_{G_n/G_{n-p} \times G_p} \SSt_p \otimes \bigg(\Ind^{G_{n-p}}_{G_{n-p-q} \times G_q}A_{n-p-q} \boxtimes \SSt_q\bigg)\]
 induces an isomorphism on homology because $\SSt_p$ is free abelian. For a precise formulation, let $D$ be the cokernel of $B \to A$ which is also the kernel of $\Sigma^p A \to C$.

 \autoref{lem:cokerhighersigma} implies that $B_n \cong 0$ for $n> d$ and that $C$ is polynomial of degree $\le r-1$ in ranks $>d-1$. Thus 
\[ \Tor^{\uZ}_q(D,\Z)_{n-p} \cong 0\quad \text{if}\quad n> d+q+p\]
and
\[ \Tor^{\uZ}_q(C,\Z)_{n-p} \cong 0\quad \text{if}\quad n> \eta(r-1,d-1,q)+p.\]
 This implies that
\[ \Tor^\uZ_q(A,\Z)_{n-p} \longrightarrow \Tor^{\uZ}_q(D, \Z)_{n-p}\]
is surjective for $n> d+q+p+1$ and bijective for $n> d+q+p$ and
\[ \Tor^\uZ_q(D,\Z)_{n-p} \longrightarrow \Tor^{\uZ}_q(\Sigma^p A, \Z)_{n-p}\]
is surjective for $n> \eta(r-1,d-1,q)+p$ and bijective for $n> \eta(r-1,d-1,q)+p-1$. Putting both maps together implies the assertion.
\end{proof}

\begin{remark}
A key feature of the proof of \autoref{prop:FhatF} is that we construct isomorphisms that are not equivariant. It would be interesting to know if there is an equivariant reformation of the arguments.
\end{remark}

Similarly to \autoref{corE2central}, we have the following.

\begin{corollary} \label{corF2Tor}
 Let $A$ be a polynomial $U\G$-module of degree $\le r$ in ranks $>d$ and assume $\G$ satisfies the standard connectivity assumptions. Then $F^\infty_{0,i}(A)_n \cong \Tor^\uZ_i(A,\Z)_n$ for 
\[n> \max( d+i, \eta(r-1,d-1,i)+1,\eta(r-1,d-1,i-1)+2, \dots, \eta(r-1,d-1,0)+i+1).\] 
\end{corollary}

\begin{proof} Fix $n> \max(d+i, \eta(r-1,d-1,i) + 1)$. In this range, we have that $\widehat F^1_{s,i+1-s}(A)_n \to F^1_{s,i+1-s}(A)_n$ is surjective for all $s \geq 1$. Assume by induction that $d^t: F^t_{s,i+1-s} \m F^t_{0,i}$ vanishes for 
\[n> \max( d+i, \eta(r-1,d-1,i)+1,\eta(r-1,d-1,i-1)+2, \dots, \eta(r-1,d-1,0)+i+1) \]  for all $1 \leq t <s$.  Consider the commutative diagram
\[ \xymatrix{
\widehat F^1_{0,i}(A)_n=\widehat F^s_{0,i}(A)_n \ar[d] & \widehat F^s_{s,i+1-s}(A)_n  =\widehat F^1_{s,i+1-s}(A)_n\ar@{->>}[d] \ar[l]_{d^s} \\
F^1_{0,i}(A)_n  = F^s_{0,i}(A)_n & F^s_{s,i+1-s}(A)_n = F^1_{s,i+1-s}(A)_n.\ar[l]_{d^s}
}\] By \autoref{specCol}, $d^s: \widehat F^s_{s,i+1-s}(A)_n   \m  \widehat F^s_{0,i}(A)_n$ vanishes and hence so does $d^s:  F^s_{s,i+1-s}(A)_n   \m   F^s_{0,i}(A)_n$. By induction, we see that there are no differentials into the group $F^s_{0,i}$ for any $s \geq 1$. Therefore, in the given range, \[\Tor^{\uZ}_i(A,\Z)_n \cong F^1_{0,i}(A)_n \cong F^\infty_{0,i}(A)_n.\qedhere\]  \end{proof}

We can now prove an improved version of \autoref{PolynomialsCentral} under the assumption that $\G$ satisfies the standard connectivity assumptions. The following is a quantitative version of \autoref{derivedRepPolynomials} under the assumption that $U\G$ satisfies the standard connectivity assumptions. 

\begin{theorem} \label{PolynomialsTor}
Let $A$ be a polynomial $U\G$-module of degree $\leq r$ in ranks $>d$. If $\G$ satisfies the standard connectivity assumptions, then $\Tor^\uZ_{i}(A,\Z)_n \cong 0$ for $n>i+\max(d,r)$.

%$n >i+d$.

\end{theorem}

\begin{proof}
We will prove by induction that $\eta(r,d,i) \leq i+\max(d,r)$ for all $r$, $d$, and $i$. The induction beginning is straightforward since $K_i(A)_n \cong 0$ in a range if $A_n \cong 0$ in a range. 

Now fix $n >i+\max(d,r)$. We will assume by induction that the claim is true for modules of polynomial degree $<r$. By \autoref{corF2Tor}, $F^\infty_{0,i}(A)_n \cong \Tor^\uZ_i(A,\Z)_n$. However, by \autoref{lem:Fvanish}, $F^\infty_{0,i}(A)_n \cong 0$. This establishes the induction step. 
\end{proof}

\section{Stability with polynomial coefficients}
\label{Sec4}
In this section, we give tools for proving representation stability and secondary stability for families of groups with polynomial coefficients.

\subsection{Representation stability with polynomial coefficients}

%Before we can prove \autoref{main}, we will need to review some properties of the following spectral sequence. 

In this subsection, we establish a general criterion for representation stability with polynomial coefficients. The following is a quantitative version of \autoref{main}. 

\begin{theorem}\label{mainCentralVersion}
Let
\[  1 \longrightarrow \mathcal N  \longrightarrow \G \longrightarrow  \cQ \longrightarrow 1 \]
be a stability short exact sequence. Assume that $U\G$ satisfies \con{H3($k,a$)}. Let $\Theta$ be a coherence function for $U\cQ$. Let $A$ be a $U\G$-module of polynomial degree $\le r$ in ranks $>d$. Let 
\[g_0 = \max(d,a-k+r), \, \,r_0 = \max(d+1,a+r+1),\]
\[g_i = \max(d+i,ki-k+a+r,\Theta(g_{i-1},r_{i-1},1), \dots,\Theta(g_{0},r_{0},i)), \text{ and}\]  
\[r_i = \max(d+i+1,ki+a+r,\Theta(g_{i-1},r_{i-1},2), \dots,\Theta(g_{0},r_{0},i+1)).\]
Then $\widetilde H^\cQ_{-1}(H_i(\cN;A))_n \cong 0$ for $n>g_i$ and $\widetilde H^\cQ_{0}(H_i(\cN;A))_n \cong 0$ for $n>r_i$.
\end{theorem}

\begin{proof}

In \autoref{PolynomialsCentral}, we showed that $\widetilde H^\G_i(A) \cong 0$ for all $n>\max(d+i+1,ki+a+r)$. By \autoref{SSSScora}, there is a spectral sequence with $(E^2_{p,q})_n \cong \widetilde H_p^\cQ( H_q(\cN;A)  )_n$ and with $(  E_{p,q}^\infty)_n \cong 0$ for $n >\max(d+p+q+1,kp+kq+a+r)$. Let 
\[g_0 = \max(d,a-k+r), \, \,r_0 = \max(d+1,a+r+1),\] 
\[g_i = \max(d+i,ki-k+a+r,\Theta(g_{i-1},r_{i-1},1), \dots,\Theta(g_{0},r_{0},i)), \text{ and}\]  
\[r_i = \max(d+i+1,ki+a+r,\Theta(g_{i-1},r_{i-1},2), \dots,\Theta(g_{0},r_{0},i+1)).\] 

We begin with the case $i=0$. Since there are no differentials into or out of the groups $(E^s_{-1,0})_n$ and $(E^s_{0,0})_n$ for $s \geq 2$, we have: 
\[\widetilde H^{\cQ}_{-1 } (H_i(\cN;A))_n \cong (E^2_{-1,0})_n = (E^\infty_{-1,0})_n \text{ and}\]
\[\widetilde H^{\cQ}_{0 } (H_i(\cN;A))_n \cong (E^2_{0,0})_n = (E^\infty_{0,0})_n.\]
 It now follows from the vanishing line for $(E^\infty_{p,q})_n$ that $\widetilde H^{\cQ}_{-1 } (H_0(\cN;A))_n \cong 0$ for $n>g_0$ and $\widetilde H^{\cQ}_{0} (H_0(\cN;A))_n \cong 0$ for $n>r_0$.

Now assume we have proven that $\widetilde H^{\cQ}_{-1 } (H_q(\cN;A))_n \cong 0$ for $n>g_q$ and $\widetilde H^{\cQ}_{0} (H_q(\cN;A))_n \cong 0$ for $n>r_q$ for all $q<i$. Recall that $(E_{p,q}^2)_n=\widetilde H^{\cQ}_p(H_{q}(\cN;A))_n$. For $q < i$, it follows from this description of the $E^2$-page and definition of the coherence function $\Theta$ that $(E^2_{p,q})_n=0$ for $n>\Theta(g_q,r_q,p)$. This rules out differentials into and out of $(E^s_{-1,i})_n$ and $(E^s_{0,i})_n$ for $s \geq 2$ and $n$ sufficiently large. In particular, $(E^2_{-1,i})_n = (E^\infty_{-1,i})_n$ for 
\[n>\max(\Theta(g_{i-1},r_{i-1},1), \dots,\Theta(g_{0},r_{0},i)).\] 
Additionally, $(E^\infty_{-1,i})_n\cong 0$ for $n>\max(d+i,ki-k+a+r)$. Since $\widetilde H^{\cQ}_{-1}(H_i(\cN;A))_n \cong (E^2_{-1,i})_n,$ it follows that $\widetilde H^{\cQ}_{-1}(H_i(\cN;A))_n \cong 0$ for $n>g_i$. An almost identical argument shows that $\widetilde H^{\cQ}_{0}(H_i(\cN;A))_n \cong 0$ for $n>r_i$.
\end{proof}

%
%We now describe some special cases of this theorem. They follow immediately by combining \autoref{mainCentralVersion}, \autoref{H3examples}, \autoref{H1H2}, and \autoref{CoherenceExamples}.

\begin{corollary}\label{cor:repstabPBr}
Let $A$ be a $U\Br$-module which has polynomial degree $\leq r$ in ranks $>d$. Then the $U\fS$-module $H_0(\PBr;A)$ has generation degree $\leq \max(d,r)$ and presentation degree $\leq \max(d+1,r+2)$. For $i>0$, the $U\fS$-module $H_i(\PBr;A)$ has generation degree 
\[\leq 2^{i-1} \cdot \Big(\max(d,r)+\max(d+1,r+2) +3 \Big)-2  \]
  and presentation degree 
  \[\leq 2^{i-1} \cdot \Big(\max(d,r)+\max(d+1,r+2)+3 \Big)-1  \]
  
\end{corollary}

\begin{proof}
We will use the stability short exact sequence:
\[ 1 \longrightarrow \PBr \longrightarrow \Br \longrightarrow \fS \longrightarrow 1.\] It follows from \autoref{H3examples} and \autoref{H1H2} that $U\Br$ satisfies \con{H3($2,2$)}. From \autoref{CoherenceExamples}, it follows that
\[\Theta(g,r,i) = g + \max(g,r) +i\]
is a coherence function for $U\fS$. We now apply \autoref{mainCentralVersion}. It is immediate that $g_0= \max(d,r)$ and $r_0 = \max(d+1,r+2)$. We will prove by induction that
\[ g_i = 2^{i-1} \cdot \Big(\max(d,r)+\max(d+1,r+2) +3 \Big)-2 \quad \text{and}\quad r_i = 2^{i-1} \cdot \Big(\max(d,r)+\max(d+1,r+2)+3 \Big)-1 \]
for $i>0$. First note that $g_j = r_j-1$ for all $j<i$ by induction. Thus
\begin{align*}&\Theta(g_{j},r_{j},m)\\=&g_j+\max(g_j,r_j)+m\\= &\bigg( 2^{j-1} \cdot \Big(\max(d,r)+\max(d+1,r+2) +3 \Big)-2\bigg)\\
& + \bigg( 2^{j-1} \cdot \Big(\max(d,r)+\max(d+1,r+2) +3 \Big)-1\bigg) +m \\=& 2^{j} \cdot \Big(\max(d,r)+\max(d+1,r+2) +3 \Big)-3 +m.\end{align*}
This implies that 
\begin{align*} g_i&= \max(d+i,2i+r,\Theta(g_{i-1},r_{i-1},1), \dots,\Theta(g_{0},r_{0},i))\\& = \Theta(g_{i-1},r_{i-1},1) \\&= 2^{i-1} \cdot \Big(\max(d,r)+\max(d+1,r+2) +3 \Big)-2 \end{align*}
and
\begin{align*} r_i&= \max(d+i+1,2i+2+r,\Theta(g_{i-1},r_{i-1},2), \dots,\Theta(g_{0},r_{0},i+1))\\& = \Theta(g_{i-1},r_{i-1},1) \\&= 2^{j} \cdot \Big(\max(d,r)+\max(d+1,r+2) +3 \Big)-1.\end{align*}
From \autoref{H3examples}, it follows that $U\fS$ satisfies \con{H3(1,1)}. Using \autoref{thm:res of finite type} and $r_i -g_i =1$, we deduce that $H_i(\PBr;A)$ is generated in degrees $\le g_i$ and presented in degrees $\le r_i$ as asserted.
\end{proof}

\begin{example} \label{exampleBurau}
Let $\Bur_n$ denote the Burau representation of $\Br_n$. By Randal-Williams--Wahl \cite[Examples 4.3, 4.15]{RWW}, the sequence $\Bur=\{\Bur_n\}_n$ assembles to form a polynomial $U \Br$-module of degree $1$. Thus, \autoref{cor:repstabPBr} implies that the $U\fS$-module $H_0(\PBr;\Bur)$ has generation degree $\leq 1$ and presentation degree $\leq 3$. For $i>0$, the $U\fS$-module $H_i(\PBr;A)$ has generation degree $\leq 7 \cdot 2^{i-1} -2  $ and presentation degree 
$\leq 7 \cdot  2^{i-1} -1. $

\autoref{PolynomialsTor} implies that $\Bur$ has generation degree $\leq 1$ and presentation degree $\leq 2$ as a  $U \Br$-module. Thus, it is reasonable to think of $\Bur$ as exhibiting a form of representation stability.

\end{example}

\begin{remark}
The exponential range produced in \autoref{cor:repstabPBr} can likely be improved to a quadratical range using ideas from \cite{CMNR}.

\end{remark}

Specializing \autoref{main} to the case $\cQ_n=1$, gives classical homological stability with twisted coefficients. In particular, \autoref{main} is a generalization of Randal-Williams--Wahl \cite[Theorem A]{RWW}. This also follows from \cite[Theorem D]{Pa2} and \autoref{PolynomialsCentral}.

\begin{corollary}
Let $U\G$ be a stability category that satisfies \con{H3($k,a$)} with $k\ge 2$. Let $A$ be a $U\G$-module which has polynomial degree $\leq r$ in ranks $>d$. Then 
\[ H_i(G_n; A_n) \longrightarrow H_i(G_{n+1};A_{n+1})\]
is surjective for
\[n\ge  \begin{cases}  \max(d,a-k+r) & i = 0\\ \max(d+2i,\max(k(i-1)+1, {2i})+a+r) &i>0\end{cases}\]
  and injective for 
\[n\ge  \max(d+2i+1,\max(ki,2i+1)+a+r).\]
\end{corollary}

\begin{proof}
In this proof, we will use the stability short exact sequence
\[ 1 \longrightarrow \G \longrightarrow \G \longrightarrow 1 \longrightarrow 1.\]
From \autoref{CoherenceExamples}, it follows that
\[\Theta(g,r,i) = \max(g+1,r) +i\]
is a coherence function of $U1$. We now apply \autoref{mainCentralVersion}. Then
\[ g_0 = \max(d,a-k+r) \text{ and}\]
\[ r_0 = \max(d+1,a+r+1) = \max(d+2\cdot 0+1,\max(ki,2\cdot 0+1)+a+r).\]
Next we want to prove
\[ g_i =  \max(d+2i,\max(k(i-1)+1,{2i})+a+r) \quad \text{and}\quad r_i = \max(d+2i+1,\max(ki,2i+1)+a+r) \]
by induction for $i>0$. First note that $g_0 +1 \le r_0$ and 
\[ g_j +1= \max(d+2j,\max(k(j-1)+1,{2j})+a+r) +1 \le  \max(d+2j+1,\max(kj,2j+1)+a+r) = r_j\]
for all $0<j<i$. Thus
\[\Theta(g_j,r_j,m) = r_j +m = \max(d+2j+1, \max(kj,2j+1)+a+r) + m.\]
Therefore 
\begin{align*} g_i&= \max(d+i,ki-k+a+r,\Theta(g_{i-1},r_{i-1},1), \dots,\Theta(g_{0},r_{0},i)) \\&= \max(d+2(i-1)+2,\max(k(i-1),2(i-1)+1)+a+r+1)  \\&= \max(d+2i,\max(k(i-1)+1,2i)+a+r) \end{align*}
and
\begin{align*} r_i&= \max(d+i+1,ki+a+r,\Theta(g_{i-1},r_{i-1},2), \dots,\Theta(g_{0},r_{0},i+1))\\& =   \max(ki+a+r, d+2(i-1)+1+2,\max(k(i-1),2(i-1)+1)+a+r+2))\\& =   \max(d+2i+1,\max(ki,k(i-1)+2,2i+1)+a+r)\\& =   \max(d+2i+1,\max(ki,2i+1)+a+r).\end{align*}
An argument as in the proof of \autoref{CoherenceExamples} shows the assertion.
\end{proof}

\subsection{Secondary stability and improved stable ranges with polynomial coefficients}

\label{secSec}

In this subsection, we prove that if the homology of a family of groups exhibits a certain form of secondary homological stability with untwisted coefficients, then it exhibits secondary stability with polynomial coefficients as well. We also describe how to use improvements to homological stability stable ranges with untwisted coefficients to deduce similar improved ranges with polynomial coefficients. We now recall some of the setup of secondary homological stability from \cite{GKRW1,GKRW2}.

Fix a commutative ring $\bk$ and stability groupoid $\G$. Let $\bR_{\bk}$ be the free $\bk$-module on the nerve of $\G$. The monoidal structure on $\G$ makes $\bR_\bk$ into an $E_1$-algebra. Let $\barR_\bk$ be a simplicial $\bk$-module which is homotopy equivalent to $\bR_\bk$ as $E_1$-algebras but is strictly associative (see \cite[Section 12.2]{GKRW1}). We will view $\barR_\bk$ as a graded simplicial $\bk$-module with the $n$th graded piece coming from $G_n$. Given a $U\G$-module $A$ over $\bk$, $\bigoplus_n (\bR_\bk)_n \otimes_{\bk G_n} A_n$ naturally has the structure of an $E_1$-module over $\bR_\bk$.  Let $\bR_A$ denote a strict $\barR_\bk$-module in the category of graded simplicial $\bk$-modules which is  homotopy equivalent to $\bigoplus_n (\bR_\bk)_n \otimes_{\bk G_n} A_n$ as an $E_1$-module (see \cite[Section 19.1]{GKRW1}). Define $H_{n,i}(\barR_\bk)$ to be the degree $n$ part of $\pi_i(|\barR_\bk|)$. Here $| \cdot |$ denotes geometric realization. We have that $H_{n,i}(\barR_\bk) \cong H_i(G_n;\bk)$. Similarly define $H_{n,i}(\bR_A)$ to be the degree $n$ part of $\pi_i(|\bR_A|)$. Then $H_{n,i}(\bR_A) \cong H_i(G_n;A_n)$.

Let $S^{a,b}_\bk$ denote the graded simplicial $\bk$-module which is the quotient of the free $\bk$-module on the simplicial set model of the  $b$-dimensional simplex modulo its boundary, where everything is concentrated in degree $a$. We have that $H_{n,i}(S^{a,b}_\bk)$ vanishes unless $n=a$ and $i=b$ in which case we have $H_{a,b}(S^{a,b}_\bk) \cong \bk$. Let $\sigma \in H_{1,0}(\barR_\bk)$ be the class of a point in $H_0(BG_1;\bk)$. Let 
\[\sigma \cdot -\colon  S^{1,0} \otimes \barR_\bk \m \barR_\bk\]
be multiplication by a lift of  $\sigma$. We have that $\sigma \cdot -$ is homotopic to the map induced by the inclusions $G_n \cong 1\times G_n  \m G_{n+1}$. Let $\barR_\bk/\sigma$ denote a $\barR_\bk$-module homotopy equivalent to the mapping cone of $\sigma \cdot -\colon \barR_\bk \m \barR_\bk$. Multiplication by a lift of $\sigma$ also gives a map $\sigma \cdot -\colon \bR_A \m \bR_A$ and we let $\bR_A/\sigma$ denote a $\barR_\bk$-module homotopy equivalent to the mapping cone of this map (see \cite[Section 19.2]{GKRW1}). Note that $H_{n,i}(\barR_\bk/\sigma) \cong H_{i}(G_n,G_{n-1};\bk)$ and $H_{n,i}(\bR_A/\sigma) \cong H_{i}(G_n,G_{n-1};A_n,A_{n-1})$. In \cite[Pages 192-193]{GKRW1}, they define a quantity $H^{\barR_\bk}_{n,i}(\bR_A)$ and prove it is naturally isomorphic to the hyper-homology groups $H_i(G_n;
\bk \otimes B_*(A,\uZ,\Z)_n)$. The following is Galatius--Kupers--Randal-Williams \cite[Theorem 19.2]{GKRW1} and relates vanishing of $H^{\barR_\bk}_{n,i}(\bR_A)$ to improved stable ranges with twisted coefficients. 

\begin{theorem}[Galatius--Kupers-Randal-Williams]\label{19.2}
Let $A$ be a $U\G$-module over $\bk$, $\lambda \leq 1$ and $c \in \R$. If $\G$ is braided, $H_{n,i}(\barR_\bk/\sigma)=0$ for $i< \lambda n$, and $H_{n,i}^{\barR_\bk}(\bR_A)=0$ for $i< \lambda(n-c)$, then $H_{n,i}(\bR_A/\sigma)=0$ for $i< \lambda(n-c)$.

\end{theorem}

We now apply this theorem to the case that $A$ has finite polynomial degree.

\begin{lemma} \label{PolynomialsHyper}
Let $A$ be a polynomial $U\G$-module of degree $\leq r$ in ranks $>d$ over $\bk$. If $U\G$ satisfies the standard connectivity assumptions, then $H_{n,i}^{\barR_\bk}(\bR_A)\cong H_i(G_n;\bk \otimes B_*(A,\uZ,\Z)_n)) \cong 0$ for $n >i+\max(d,r)$.
\end{lemma}

\begin{proof}
\autoref{PolynomialsTor} implies that $\Tor^{\uZ}_q(A,\Z)_n \cong H_q(B_*(A,\uZ,\Z)_n) \cong 0$ for $n> q+\max(d,r)$. Thus $H_q(B_*(A,\uZ,\Z)_n \otimes \bk) \cong 0$ for $n> q+\max(d,r)$. The hyperhomology spectral sequence \[ E^2_{pq} \cong H_p(G_n; H_q(B_*(A,\uZ,\Z)_n \otimes \bk) ) \implies H_{p+q}(G_n;B_*(A,\uZ,\Z)_n) \otimes \bk) \]
implies the assertion because $E^2_{pq} \cong 0$ for $n > q+\max(d,r)$.
\end{proof}

Combining \autoref{19.2} and \autoref{PolynomialsHyper} give the following.

\begin{theorem} \label{improvedrangeGeneral}
Let $A$ be a polynomial $U\G$-module of degree $\leq r$ in ranks $>d$ over $\bk$. Let $\lambda \leq 1$ and $c \in \R$. Assume $U\G$ satisfies the standard connectivity assumptions, $\G$ is braided, and that $H_i(G_n,G_{n-1};\bk) \cong 0$ for $i< \lambda n$. Then $H_i(G_n,G_{n-1};A_n,A_{n-1}) \cong 0$ for $i< \lambda n - \max(r,d)$.
\end{theorem}

Before we move on to secondary stability, we describe a few applications of this theorem.

\begin{theorem} \label{improvedRangeSn}
Let $A$ be a polynomial $U \fS$-module of degree $\leq r$ in ranks $>d$ over $\Z[\textstyle\frac 12]$. Then $H_i(\fS_n,\fS_{n-1};A_n,A_{n-1}) \cong 0$ for $i< n - \max(r,d)$.
\end{theorem}

\begin{proof}
It is well known from calculations of the homology of the symmetric groups (see e.g. Cohen--Lada--May \cite[Section 1]{CLM}) or from \cite[Theorem 1.4]{kupersmillerimprov} for $M=\R^\infty$ that 
\[H_i(\fS_n;\fS_{n-1};\Z[\textstyle\frac 12]) \cong 0 \text{ for }i<n. \] 
The claim now follows from \autoref{improvedrangeGeneral}.  \end{proof}

\begin{theorem} \label{improvedRangeGLn}
Let $A$ be a polynomial $U\GL(\Z)$-module of degree $\leq r$ in ranks $>d$ over $\Z[\textstyle\frac 12]$. Then $H_i(\GL_n(\Z),\GL_{n-1}(\Z);A_n,A_{n-1}) \cong 0$ for $i<\frac23n-\max(r,d)$.
\end{theorem}

\begin{proof}

In \cite[Section 18.2]{GKRW1} Galatius--Kupers--Randal-Williams proved that: 
\[H_i(\GL_n(\Z), \GL_{n-1}(\Z); \Z[\textstyle\frac 12]) \cong 0 \text{ for }i < \textstyle\frac{2}{3}n.\] 
The claim now follows from \autoref{improvedrangeGeneral}.
\end{proof}

We now give a general secondary stability theorem for twisted coefficients. A \emph{secondary stability map of bidegree $(a,b)$} is a map of $\barR_\bk$-modules 
\[f\colon \left( \barR_\bk/\sigma \right)\otimes_\bk S_\bk^{a,b} \m \barR_\bk/\sigma.\] 
Such a map induces a map 
\[f_*\colon H_{i-b}(G_{n-a},G_{n-a-1};A_{n-a},A_{n-a-1}) \m H_{i}(G_{n},G_{n-1};A_n,A_{n-1})\] 
for any $U\G$-module $A$ over $\bk$. The following theorem appears as \cite[Theorem 5.20]{GKRW2} for $\G = \Mod$ with a specific choice of secondary stability map but no changes to the proof are necessary for general $\G$ with the assumption given here.

\begin{proposition} \label{SecondaryGeneralHyper}
Let $\lambda \leq 1$ and $c \in \R$. Let $A$ be a $U\G$-module with $H_{n,i}^{\barR_\bk}(\bR_A)=0$ for $i< \lambda(n-c)$. Assume $U\G$ satisfies the standard connectivity assumptions, $\G$ is braided, and that there is a secondary stability map $f$ of bidegrees $(a,b)$ which induces a surjection 
\[f_*\colon H_{i-b}(G_{n-a},G_{n-a-1};\bk) \m H_{i}(G_{n},G_{n-1};\bk)\]
for $i \leq \lambda(n-c)$ and an isomorphism for $i \leq \lambda (n-c-a)$. Then 
\[f_*\colon H_{i-b}(G_{n-a},G_{n-a-1};A_{n-a},A_{n-a-1}) \m H_{i}(G_{n},G_{n-1};A_{n},A_{n-1})\] 
is a surjection for $i \leq \lambda (n-c)$ and an isomorphism for $i \leq \lambda (n-c)-1$. 
\end{proposition}

Combining  \autoref{PolynomialsHyper} and \autoref{SecondaryGeneralHyper} gives the following theorem.

\begin{theorem} \label{SecondaryGeneral}
Let $A$ be a polynomial $U\G$-module of degree $\leq r$ in ranks $>d$ over $\bk$. Let $ \lambda  \leq 1$ and $c \in \R$. Assume $U\G$ satisfies the standard connectivity assumptions, $\G$ is braided, and that there is a secondary stability map $f$ of bidegrees $(a,b)$ which induces a surjection 
\[f_*\colon H_{i-b}(G_{n-a},G_{n-a-1};\bk) \m H_{i}(G_{n},G_{n-1};\bk)\] 
for $i \leq \lambda (n-c)$ and an isomorphism for $i \leq \lambda  (n-c-a)$. Then 
\[f_*\colon H_{i-b}(G_{n-p},G_{n-a-1};A_{n-a},A_{n-a-1}) \m H_{i}(G_{n},G_{n-1};A_{n},A_{n-1})\]
is a surjection for $i \leq  \lambda(n-c-\max(r,d))$ and an isomorphism for $i \leq  \lambda(n-c-a-\max(r,d)) -1$. 
\end{theorem}

 \autoref{mainSecondary} is the $\bk=\Z$ case of \autoref{SecondaryGeneral}. We will apply this to the case of mapping class groups. For this application, \cite[Theorem 5.20]{GKRW2} would have sufficed. \autoref{SecondaryGeneral} has other applications such as secondary stability for braid groups with coefficients in the Burau representation. %This will appear in future work of Z. Himes. 

\begin{corollary} \label{mappingclassgroups}
Let $A$ be a polynomial $\Mod$-module of degree $\leq r$ in ranks $>d$. Then 
\[H_i(\Mod_{g-1,1};A_{g-1}) \m H_i(\Mod_{g,1};A_{g})\] 
is surjective for $i < \frac{2}{3}(g-\max(d,r))$ and an isomorphism for $i < \frac{2}{3}(g-\max(d,r))-1$. Moreover, there is a map 
\[H_{i-2}(\Mod_{g-3,1},\Mod_{g-4,1};A_{g-3},A_{g-4}) \m H_i(\Mod_{g,1},\Mod_{g-1,1};A_{g},A_{g-1})\] 
which is a surjection for $i <\frac{3}{4}(g-\max(d,r))$ and an isomorphism for $i <\frac{3}{4}(g-\max(d,r))-1$.
\end{corollary}

\begin{proof}
The homological stability portion of this corollary follows Galatius--Kupers--Randal-Williams \cite[Corollary 5.2]{GKRW2} and \autoref{improvedrangeGeneral} and the secondary stability portion from Galatius--Kupers--Randal-Williams \cite[Theorem A]{GKRW2} and \autoref{SecondaryGeneral}.
\end{proof}

\begin{corollary} \label{burSec}

Let $\Bur_n$ denote the Burau representation of the braid group $\Br_n$. Then there is a map 
\[H_{i-1}(\Br_{n-2},\Br_{n-3};\Bur_{n-2},\Bur_{n-3}) \m  H_i(\Br_{n},\Br_{n-1};\Bur_{n},\Bur_{n-1}) \] which
is surjective for $i < \frac{2}{3}(n-1)$ and an isomorphism for $i < \frac{2}{3}(n-1)-1$.
\end{corollary}

\begin{proof}
Recall that $\sigma$ denotes a cycle corresponding to a generator of $H_0(\Br_1)$. Let $Q^1 \sigma$ denote a cycle corresponding to the generator of $H_1(\Br_2)$ and let $f:\barR_\Z \m \barR_\Z $ denote multiplication by $Q^1 \sigma$. This gives a map \[f_*:H_{i-1}(\Br_{n-2},\Br_{n-3};\Bur_{n-2},\Bur_{n-3}) \m  H_i(\Br_{n},\Br_{n-1};\Bur_{n},\Bur_{n-1}) \] It follows from Cohen--Lada--May's \cite{CLM}
computations of the homology of free $E_k$-algebras that $f_*$ is surjective for $i < \frac{2}{3}n$ and an isomorphism for $i < \frac{2}{3}n-1$ (see Himes \cite[Theorem 1.1]{Himes}). By Randal-Williams--Wahl \cite[Examples 4.3, 4.15]{RWW}, the sequence $\Bur=\{\Bur_n\}_n$ assembles to form a polynomial $U \Br$-module of degree $1$. The claim now follows by \autoref{SecondaryGeneral}.
\end{proof}

\begin{remark}
\autoref{burSec} applies equally well to the specialization of $\Bur_n$ to $t$ a root of unity. For $t=-1$, this gives the sympletic representation of $\Br_n$ acting on a $H_1(\Sigma_n)$, as discussed in greater detail in the following Subsection \ref{hyperelliptic-curves-subsection}. Computations of Callegaro--Salvetti \cite[Table 1]{CallSal} imply that the relative homology groups $H_i(\Br_n;H_1(\Sigma_n))$ are often nonzero in the metastable range of \autoref{burSec}. In particular, slope $\frac{1}{2}$ stability is optimal for braid groups with coefficients in the Burau representation specialized to $t=-1$.

Let $\Conf_n(X)$ denote the configuration space of $n$ distinct unordered  points in a topological space $X$. The homology groups  $H_i(\Br_{n};H_1(\Sigma_n))$ are relevant because of their relation with the homology of the space \[ \mathcal E_n=\{(\{x_1,\ldots,x_n  \},z,y) \, | \, y^2=(z-x_1)\cdots (z-x_n) \} \subset \Conf_n(\mathbb C) \times \mathbb C \times \mathbb C,\]
which is the total space of the universal family of affine hyperelliptic curves. There is an isomorphism (see 
Callegaro--Salvetti \cite[Page 1]{CallSal}) \[H_i(\mathcal E_n )  \cong H_i(\Br_n ) \oplus H_{i-1}(\Br_n ; H_1(\Sigma_n) ).      \] Thus \autoref{burSec} specialized to $t=-1$ combined with Himes \cite[Theorem 1.1]{Himes} implies secodary homological stability for the space $\mathcal E_n$.

\end{remark}

\section{Applications}
\label{Sec5}
In this section, we apply our general stability theorems to concrete examples. Namely, we prove a twisted homological stability theorem for moduli spaces of hyperelliptic curves, establish an improved representation stability stable range for congruence subgroups, establish secondary stability for diffeomorphism groups of surfaces viewed as discrete groups, and prove an improved homological stability stable range for homotopy automorpihsms of wedges of spheres and general linear groups of the sphere spectrum.

\subsection{Moduli spaces of hyperelliptic curves}

\label{hyperelliptic-curves-subsection}

Let $\mathcal M_g$ denote the moduli space of smooth genus $g$ curves over $\mathbb C$, and let $\mathcal M_{g,\partial}$ denote the moduli space of smooth genus $g$ curves together with a marked point and a nonzero tangent vector at that point. Both $\mathcal M_g$ and $\mathcal M_{g,\partial}$ are $ K(\pi,1)$ spaces, where $\pi$ is the mapping class group of a closed genus $g$ surface, and a genus $g$ surface with one boundary component, respectively. (For this to be true we should either work rationally or consider $\mathcal M_g$ as a stack or orbifold; we will prefer the latter perspective.)

Gluing on a torus defines an embedding of the mapping class group of a genus $g$ surface with a boundary component into the mapping class group of a genus $g+1$ surface with boundary. Thus we get a continuous (nonalgebraic) map $\mathcal M_{g,\partial} \to \mathcal M_{g+1,\partial}$, which is well defined up to homotopy. This map induces homological stability. 

\begin{theorem}[\cite{harerstab},\cite{GKRW2}] The induced map $H^i(\mathcal M_{g+1,\partial};\Z) \to H^i(\mathcal M_{g,\partial};\Z)$ is an isomorphism for $i \leq \frac {2g-4}{3}$.
\end{theorem}

Algebraic geometers are often more interested in the case of closed surfaces. For this we need to consider the evident (algebraic) forgetful maps $\mathcal M_{g,\partial} \to \mathcal M_g$; Harer has proved that the induced morphism $H^i(\mathcal M_g;\Z) \to H^i(\mathcal M_{g,\partial};\Z)$ is an isomorphism for $i \leq \frac 2 3 g$ \cite{harerstab,Boldsen}. In particular, there are also isomorphisms $H^i(\mathcal M_{g+1};\Z) \cong H^i(\mathcal M_{g};\Z)$ for $i \le \frac{2g-4}3$. 

\begin{remark}Although $\mathcal M_{g,\partial} \to \mathcal M_{g+1,\partial}$ is not an algebraic map, $H^i(\mathcal M_{g+1,\partial}) \to H^i(\mathcal M_{g,\partial})$ still preserves algebraic structures such as the mixed Hodge structure, the comparison isomorphisms with the algebraic de Rham cohomology and the \'etale cohomology, including the structure of $\ell$-adic Galois representations. The reason is that $\mathcal M_{g+1,\partial}$ can be given a Deligne--Mumford style partial compactification in which a boundary stratum adjacent to $\mathcal M_{g+1,\partial}$ is isomorphic to $\mathcal M_{g,1} \times \mathcal M_{1,1+\partial}$; here $\mathcal M_{g,1}$ is the moduli of genus $g$ curves with a marked point, and $\mathcal M_{1,1+\partial}$ parametrizes genus $1$ curves with two marked points, one of which is equipped with a nonzero tangent vector. The complement of the zero section in the normal bundle of this boundary stratum is isomorphic to $\mathcal M_{g,\partial} \times \mathcal M_{1,\partial + \partial}$ with $\mathcal M_{1,\partial + \partial}$ the moduli space of genus one curves with two distinct marked points and a nonzero tangent vector at each of the marked points. After choosing a tubular neighborhood and a fixed point of $\mathcal M_{1,\partial + \partial}$, this defines a continuous embedding of  $\mathcal M_{g,\partial} $ into  $\mathcal M_{g+1,\partial} $ which coincides with the one defined up to homotopy by gluing on a torus in the mapping class group. Although tubular neighborhoods do not literally exist in algebraic geometry, the induced map on cohomology can still be defined purely algebraically via ``deformation to the normal cone'' (or ``Verdier specialization''). See e.g.\ \cite[Section 4]{hainlooijenga}.
\end{remark}

There is a version of the Harer stability theorem for twisted coefficients, due to Ivanov \cite{Ivanov}. The stable range was later improved by Boldsen \cite{Boldsen}. See \autoref{mappingclassgroups} for a further improvement to this stable range.

An important example of a polynomial coefficient system is the following. There is a natural rank $2g$ local system on $\mathcal M_{g,\partial}$ given by the first cohomology group of the curve, and as $g$ varies it defines a polynomial coefficient system of degree $1$. Its $r$-fold tensor power with itself is a polynomial coefficient system of degree $r$; more generally, any Schur functor (for the general linear group or for the symplectic group) applied to the standard degree $1$ coefficient system produces again a polynomial coefficient system. We denote these ``standard'' coefficient systems by $V_\lambda$; they are parametrized by partitions $\lambda$. 

Again one can ask what happens for closed surfaces. It is not true in general that  \[H^i(\mathcal M_g;V_\lambda) \to H^i(\mathcal M_{g,\partial};V_\lambda)\] is an isomorphism stably. However, at least rationally it will be true that the Leray--Serre spectral sequence for $\mathcal M_{g,\partial} \to \mathcal M_g$ behaves predictably in a stable range also for coefficients in $V_\lambda$, and a consequence is that there are isomorphisms $H^i(\mathcal M_g;V_\lambda \otimes \Q) \cong H^i(\mathcal M_{g+1};V_{\lambda} \otimes \Q)$ for $g$ large with respect to $i$ (see Looijenga \cite[Theorem 1.1]{LooijengaTwisted}). Moreover, these isomorphisms respect the natural mixed Hodge structure/Galois module structure etc., just as in the case of constant coefficients.

It is natural to ask whether there is a version of the above story if one replaces the usual mapping class group with the hyperelliptic mapping class group, or in algebro-geometric terms, if we replace $\mathcal M_g$ with the moduli space $\mathcal \mathcal \mathcal H_g$ of hyperelliptic curves. Again it is natural to consider the case of surfaces with boundary, in order to even have a stabilization map. We let $\mathcal H_{g,\partial}$ denote the moduli space parametrizing a hyperelliptic curve of genus $g$, the choice of a marked Weierstrass point (i.e.\ a fixed point of the hyperelliptic involution), and a nonzero tangent vector at the Weierstrass point. The space $\mathcal H_{g,\partial}$ is again a $ K(\pi,1)$, but the group $\pi$ is now (as we will later explain geometrically) the Artin braid group $\Br_{2g+1}$ on $2g+1$ strands, which we may think of as the hyperelliptic mapping class group of a genus $g$ surface with boundary. In particular, the hyperelliptic analogue of Harer stability is simply Arnold's theorem that the braid groups satisfy homological stability \cite{Ar}. The hyperelliptic analogue of Ivanov's theorem, i.e.\ homological stability for the braid groups with polynomial coefficients, is a more recent theorem of Randal-Williams--Wahl \cite[Theorem D]{RWW}. Of particular interest are the ``standard'' polynomial coefficients systems $V_\lambda$, restricted from $\mathcal M_{g,\partial}$ to $\mathcal H_{g,\partial}$ using the evident embedding $\mathcal H_{g,\partial} \hookrightarrow \mathcal M_{g,\partial}$. The pullback of the standard rank $2g$ local system $V_1$ on $\mathcal M_{g,\partial}$ to $\mathcal H_{g,\partial}$, considered as a representation of the braid group on $2g+1$ strands, is precisely the \emph{reduced Burau representation} specialized to $t=-1$, see e.g. Chen \cite{chen-burau}. 

\begin{remark}
	The inclusion $\mathcal H_{g,\partial} \hookrightarrow \mathcal M_{g,\partial}$ induces a map of fundamental groups from the braid group to the mapping class group of a surface with a boundary component. This map, and the induced map in (co)homology, has been studied in several papers over the years, see e.g.\ \cite{birmanhilden,songtillmann,segaltillmann,chen-burau,bianchi-braid}. 
\end{remark}

\begin{remark}
	Similar arguments as in the case of $\mathcal M_{g,\partial}$ imply the compatibility of the stabilization maps with mixed Hodge structure and Galois module structure. The partial compactification of $\mathcal H_{g,\partial}$ used is the one defined by admissible covers as in Abramovich--Corti--Vistoli \cite{acv03}. 
	\label{remarkMixHodgeH}
\end{remark}

Now we may consider instead \emph{closed} hyperelliptic surfaces. At this point we will restrict our attention to working rationally, i.e.\ we tensor all coefficient systems with $\Q$. It turns out that the reduced rational cohomology of $\mathcal H_g$ vanishes for all $g$ (see  \autoref{acyclic}), so homological stability for constant (rational) coefficients is uninteresting. But the cohomology of $\mathcal H_g$ with twisted coefficients is highly nontrivial in general and not much is known about it. Again there are no natural stabilization maps and the best we can ask for is that $H^i(\mathcal H_{g+1};V_{\lambda}) \cong H^i(\mathcal H_{g};V_\lambda)$ for $g \gg i$. It does not seem easy to deduce from the results of Randal-Williams--Wahl the existence of such an isomorphism in the case of closed surfaces. However if one knows not just homological stability for $\Br_n$ with twisted coefficients but representation stability for the \emph{pure} braid groups $\PBr_n$ with twisted coefficients, then one can deduce homological stability for closed surfaces, too. The following  is a corollary of \autoref{cor:repstabPBr}.

\begin{proposition} \label{CorBur}
%Let $\Bur_n$ denote the Burau representation. 
For a fixed $i \in\N$,  the $U\fS$-module
\[n \mapsto H_i(\PBr_n;V_\lambda),\]
is generated in degree $\le 2^{i-1}(2 \vert\lambda\vert+ 5) -2$ and presented in degree $\leq 2^{i-1}(2 \vert\lambda\vert+ 5) -1$. 

%In particular, the $U\fS$-module \[n \mapsto H_i(\PBr_n;V_1)\] is generated in degree $\le 7 \cdot 2^{i-1} -2$ and presented in degree $\leq 7\cdot 2^{i-1} -1$. 

\end{proposition}

\subsubsection{The various moduli spaces involved}

For the proof of  \autoref{hyperelliptic-thm} we will need to work with several closely related moduli spaces, which we will define now. In the process we also explain why $\mathcal H_{g,\partial}$ is algebraically isomorphic to the configuration space of $2g+1$ distinct unordered points in $\mathbb A^1$ up to translation, which explains in particular why $\mathcal H_{g,\partial}$ is a $\mathrm K(\pi,1)$ for the braid group. 

A hyperelliptic curve of genus $g$ is a double cover of $\mathbb P^1$ branched at $2g+2$ points. This gives a map of stacks $\mathcal H_g \to \mathcal M_{0,2g+2}/\fS_{2g+2}$, where $\mathcal M_{0,n}$ denotes the moduli space parametrizing $n$ distinct ordered points on $\mathbb P^1$ up the action of $\mathrm{PGL}(2)$. The hyperelliptic curve is determined up to isomorphism by the location of the branch points; moreover, the automorphism group of the hyperelliptic curve is a $\Z/2$-central extension\footnote{If we want to be careful about characteristic $2$ we are better off replacing $\Z/2$ here, and in the discussion which follows, with the group scheme $\mu_2$.} of the symmetry group of the configuration of branch points. This reflects the fact that $\mathcal H_g \to \mathcal M_{0,2g+2}/\fS_{2g+2}$ is not an isomorphism, but a $\mathbb Z/2$-gerbe, or in terms of geometric group theory, that the hyperelliptic mapping class group is a $\Z/2$-central extension of the mapping class group of a sphere with $2g+2$ unordered punctures.

We let $\mathcal H_{g,1}$ denote the moduli space parametrizing hyperelliptic curves with a marked Weierstrass point, i.e.\ a distinguished ramification point of the canonical double cover of $\mathbb P^1$. The space $\mathcal H_{g,1}$ is a $\mathbb Z/2$-gerbe over $\mathcal M_{0,2g+2}/\fS_{2g+1}$, by the same reasoning as the preceding paragraph.

Let $L$ denote the line bundle over $\mathcal M_{0,2g+2}/\fS_{2g+1}$ given by the tangent space of $\mathbb P^1$ at the $(2g+2)$nd marked point. Note that $L^\ast$, the complement of the zero section in $L$, is isomorphic to $\mathrm{Conf}_{2g+1}(\mathbb A^1)/\mathbb A^1$, the configuration space of $2g+1$ unordered distinct points in $\mathbb A^1$ modulo translation. Indeed, given a point of $\mathcal M_{0,2g+2}/\fS_{2g+1}$ we may use the gauge freedom to put the $(2g+2)$nd marked point at infinity, in which case we are considering $2g+1$ distinct unordered points of $\mathbb P^1 \setminus \{\infty\}$ modulo the subgroup of $\mathrm{PGL}(2)$ fixing $\infty$ and moreover fixing a nonzero tangent vector at $\infty$. But that subgroup is simply the group of affine translations.

Now there is a natural map $\mathcal H_{g,\partial} \to L^\ast \cong \mathrm{Conf}_{2g+1}(\mathbb A^1)/\mathbb A^1$, associating to a hyperelliptic curve together with a nonzero tangent vector at a Weierstrass point its set of branch points and the corresponding nonzero tangent vector at the branch point which is the image of the distinguished Weierstrass point. This, however, is \emph{not} a $\Z/2$-gerbe---the factor of $\mathbb Z/2$ arose previously since every hyperelliptic curve has the automorphism given by the hyperelliptic involution, but the hyperelliptic involution will not fix any nonzero tangent vector at a Weierstrass point. It follows that  $\mathcal H_{g,\partial} \to \mathrm{Conf}_{2g+1}(\mathbb A^1)/\mathbb A^1$ is in fact an isomorphism. 

\begin{remark}\label{acyclic}
	It is easy to see from the above considerations that $\mathcal H_g$ has the rational cohomology of a point, as mentioned earlier. Indeed we obtain that 
	\[ H^\ast(\mathcal H_g;\Q) \cong H^\ast(\mathcal M_{0,2g+2};\Q)^{\fS_{2g+2}} \subseteq H^\ast(\mathcal M_{0,2g+2};\Q)^{\fS_{2g+1}}.\]
	But the above discussion identified $\mathcal M_{0,2g+2}/\fS_{2g+1}$ with the quotient of $\mathrm{Conf}_{2g+1}(\mathbb A^1)/\mathbb A^1$ by the group $\mathbb G_m$. It is well known that the rational cohomology of the braid group is the same as the cohomology of a circle, an isomorphism being given by the action of $S^1$ on the configuration space by rotation. So $\mathcal M_{0,2g+2}/\fS_{2g+1}$ has the rational cohomology of a point, and then a fortiori so does $\mathcal M_{0,2g+2}/\fS_{2g+2}$.
\end{remark}

\begin{remark}Another perspective on the isomorphism $\mathcal H_{g,\partial} \cong \mathrm{Conf}_{2g+1}(\mathbb A^1)/\mathbb A^1$ is that the line bundle $L$ canonically acquires a square root when pulled back to $\mathcal H_{g,1}$: a square root is given by the line bundle given by the tangent line of the hyperelliptic curve at the distinguished Weierstrass point. Using this, one may in fact identify $\mathcal H_{g,1}$ with the $\Z/2$-gerbe over $\mathcal M_{0,2g+2}/\fS_{2g+1}$ parametrizing square roots of the line bundle $L$; see e.g.\ \cite[Section 2]{maslovgerbe} for this construction. Now $L$ is tautologically trivialized over $L^\ast$ and then so is the pullback of the corresponding gerbe of square roots, so $\mathcal H_{g,1} \times_{\mathcal M_{0,2g+2}/\fS_{2g+1}} L^\ast \cong B(\Z/2) \times L^\ast$. On the other hand this fibered product also equals the quotient of $\mathcal H_{g,\partial}$ by $\Z/2$, acting by multiplying the tangent vector at the Weierstrass point with $-1$. But this action is trivial because of the hyperelliptic involution, so $\mathcal H_{g,\partial} \cong L^\ast$. 
\end{remark}

\subsubsection{Uniform multiplicity stability}

Representation stability was first formulated by Church--Farb \cite{CF} in terms of multiplicities of irreducible representations stabilizing. Since the work of Church--Ellenberg--Farb \cite{CEF}, this approach has generally gone out of style in favor of more categorical forms of representation stability. However, to prove \autoref{hyperelliptic-thm} we will need to consider a kind of stability result for an $\fS$-module which does not naturally come from a $U\fS$-module. This does not fit neatly into the categorical formalism used in the rest of the paper, which will force us to switch back to this older form of representation stability. Thus we briefly review the theory here. In this subsection, all representations are assumed to be over $\Q$ and we will only consider representations of symmetric groups. Recall that the data of a $U \fS$-module is the same as the data of an $\FI$-module in the sense of Chruch--Ellenberg--Farb \cite{CEF}.

In characteristic zero, irreducible representations of $\fS_n$ are in bijection with partitions of $n$ and we denote the representation associated to a partition $\lambda$ by $\sigma_\lambda$. Given $\lambda=(l_1 \geq \dots \geq l_j)$ a partition of $k$ and $n \geq l_1 +k$, let $\lambda_n=(n-k \geq l_1 \geq \dots \geq l_j)$ be the partition of $n$ obtained by appending $(n-k)$ boxes above the top row of the Ferrers diagram of $\lambda$. Let $V(\lambda)_n$ be the representation given by $\sigma_{\lambda_n}$ for $n \geq l_1+k$ and $0$ for $n<l_1+k$.

\begin{definition}
	Let $A$ be an $\mathfrak S$-module over $\Q$. We say that $A$ has \emph{uniform multiplicity stability} starting at $N$ if there are numbers $c_{\lambda}$ such that \[A_n \cong \bigoplus_\lambda c_\lambda V(\lambda)_n \text{ for }n \geq N, \]
	and such that if $c_\lambda \neq 0$ then $V(\lambda)_N \neq 0$. 
	
\end{definition}

\begin{definition}
	Let $A_n$ be a $\Q[S_n]$-module. We say $A_n$ has \emph{weight} $\leq N$ if $A_n$ is isomorphic to a direct sum of $\Q[S_n]$-modules of the form $V(\lambda)_n$ with $\lambda$ a partition of $k$ and $k \leq N$.  Let $\{A_n\}_n$ be a sequence with $A_n$ a $\Q[S_n]$-module. We say $\{A_n\}_n$ has weight $\leq N$ if each $A_n$ has weight $\leq N$.
\end{definition}

\begin{definition}
	Let $A$ be a $U\fS$-module. We say $A$ has \emph{stability degree} $\leq N$ if for all $k \geq 0$ and $n \geq N$, the natural map $(A_{n+k})_{\fS_{n}} \m (A_{n+k+1})_{\fS_{n+1}}$ is an isomorphism.
\end{definition}

The following is well-known and can be proven by stringing together results of Church--Ellenberg--Farb \cite{CEF}.

\begin{proposition} \label{presUniform}
	Let $A$ be a $U \fS$-module over $\Q$ generated in degree $d$ and presented in degree $r$ with each $A_n$ finite dimensional. Then $A$ has uniform multiplicity stability starting at $d+r$.
\end{proposition}

\begin{proof}
	Church--Ellenberg--Farb \cite[Proposition 3.2.5]{CEF} implies that $A$ has weight $\leq g$. Since $A$ has generation degree $\leq g$ and presentation degree $r$, there is a resolution \[P_1 \m P_0 \m A \] with $P_0,P_1$ free, $P_0$ generated in degree $\leq d$ and $P_1$ generated in degree $\leq d$. Church--Ellenberg--Farb \cite[3.1.7]{CEF} implies that $P_1$ and $P_0$ have stability degree $\leq r$. Now, Church--Ellenberg--Farb \cite[Lemma 3.1.6]{CEF} implies that $A$ has stability degree $\leq r$. Church--Ellenberg--Farb \cite[Proposition 3.3.3]{CEF}  implies that the sequence $\{A\}_n$ has uniform multiplicity stability starting at $d+r$.
\end{proof}

The following lemma is perhaps known, although we do not know of a reference proving it in precisely this form. The implication (i) $\implies$ (ii) is Church--Farb \cite[Theorem 3.2]{CF}.

\begin{lemma} \label{deShift}
	Let $A$ be an $\fS$-module over $\Q$, and let $\Sigma A$ denotes its shift as in \autoref{secPoly}, i.e. $\Sigma A_n =\mathrm{Res}^{\fS_{n+1}}_{\fS_n} A_{n+1}$. The following are equivalent:
	\begin{enumerate}[\rm (i)]
		\item $A$ is uniformly multiplicity stable.
		\item $\Sigma A$ is uniformly multiplicity stable.
	\end{enumerate}
\end{lemma}

\begin{proof}Let $\lambda$ be a partition of $n$, and $\sigma_\lambda$ the corresponding representation of $\fS_n$. Pieri's formula says that $\mathrm{Res}^{\fS_{n}}_{\fS_{n-1}} \sigma_\lambda \cong \bigoplus_\mu \sigma_\mu$, where the summation runs over all partitions $\mu$ that can be obtained by removing a box from the Ferrers diagram of $\lambda$. We may interpret $\mathrm{Res}$ as a linear map $R(\fS_n) \to R(\fS_{n-1})$ between rings of virtual representations. This linear map is of course far from invertible, as the number of partitions of $n$ is larger than the number of partitions of $n-1$. 
	
	Note now that if either of the two sequences $\{A_n\}_n$ or  $\{\Sigma A_n\}_n$ is {uniformly} multiplicity stable then both sequences will consist only of representations corresponding to partitions whose corresponding Ferrers diagrams have at most $N$ boxes below the first row, for some $N$. Then for $n > N$ we have that the numbers of partitions of $n$ and of $n-1$ satisfying this restriction are the same. We may consider  $\mathrm{Res}$  as a linear map between the corresponding subspaces of $R(\fS_n)$ and $R(\fS_{n-1})$, and this linear map is now invertible: indeed, the spaces are free abelian groups of the same rank, and the linear map is upper triangular with ones on the diagonal with respect to the bases given by partitions ordered by dominance. The conclusion follows. 
\end{proof}

\subsubsection{Proof of \autoref{hyperelliptic-thm}}

	We know that the cohomology groups $H^\ast(\mathcal H_{g,\partial};V_\lambda)$ satisfy homological stability as $g \to \infty$ by Randal-Williams--Wahl \cite[Theorem D]{RWW}. We will now prove the same for $H^\ast(\mathcal H_g;V_\lambda)$. 

\begin{proof}[Proof of  \autoref{hyperelliptic-thm}]

	Let $\widetilde{\mathcal H}_g$ be the space parametrizing hyperelliptic curves with a total ordering on their sets of Weierstrass points, so that $\mathcal H_g=\widetilde{\mathcal H}_g/\fS_{2g+2}$.  Now it is enough to prove that the sequence of representations of $\fS_{2g+2}$ given by $H^\ast(\widetilde{\mathcal H}_g;V_\lambda)$ is uniformly multiplicity stable since this will imply that $H^\ast(\widetilde{\mathcal H}_g;V_\lambda)_{\fS_{2g+2}}$ stabilizes. Similarly let $\widetilde{\mathcal H}_{g,\partial}$ parametrize hyperelliptic curves with a total ordering on their set of Weierstrass points, as well as a distinguished tangent vector at the last Weierstrass point, so that $\mathcal H_{g,\partial} = \widetilde{ \mathcal H}_{g,\partial}/\fS_{2g+1}$. Now note that the sequence of representations of $\fS_{2g+1}$ given by $H_\ast(\widetilde{ \mathcal H}_{g,\partial};V_\lambda)$ is uniformly multiplicity stable, since they are simply the homology groups of the pure braid groups $\PBr_{2g+1}$ with certain polynomial coefficients, which satisfy representation stability by  \autoref{CorBur} and \autoref{presUniform}. Since the multiplicities of the irreducible subrepresentations of $H_\ast(\widetilde{ \mathcal H}_{g,\partial};V_\lambda)$ and $H^\ast(\widetilde{ \mathcal H}_{g,\partial};V_\lambda)$ agree, these cohomology groups are also uniformly multiplicity stable. We will deduce the result by combining this fact with the preceding lemma.

	The circle bundle $\widetilde {\mathcal H}_{g,\partial} \to \widetilde{ \mathcal H}_{g}$ induces a Gysin sequence
	\[ \dots \to H^i(\widetilde {\mathcal H}_{g,\partial};V_\lambda) \to  H^{i-1}(\widetilde{\mathcal H}_{g};V_\lambda) \to H^{i+1}(\widetilde{\mathcal H}_{g};V_\lambda) \to H^{i+1}(\widetilde{\mathcal H}_{g,\partial};V_\lambda) \to \dots \]
	where $H^{i-1}(\widetilde{\mathcal H}_{g};V_\lambda) \to H^{i+1}(\widetilde{\mathcal H}_{g};V_\lambda)$ is multiplication with the first Chern class of the circle bundle. The circle bundle is pulled back from $\mathcal H_{g,1}$, so by  \autoref{acyclic} its first Chern class vanishes rationally and we obtain isomorphisms $H^i(\widetilde{\mathcal H}_{g,\partial};V_\lambda) \cong H^i(\widetilde {\mathcal H}_{g};V_\lambda) \oplus H^{i-1}(\widetilde{\mathcal H}_{g};V_\lambda)$. These isomorphisms are clearly $\fS_{2g+1}$-equivariant. Then uniform multiplicity stability for the sequence of $\fS_{2g+1}$-representations $\{H^\ast(\widetilde{\mathcal H}_{g,\partial};V_\lambda)\}$ implies the same for $\{H^\ast(\widetilde{ \mathcal H}_{g};V_\lambda)\}$. By \autoref{deShift}, knowing uniform multiplicity stability for $\{H^\ast(\widetilde {\mathcal H}_{g};V_\lambda)\}$ considered as a sequence of representations of $\fS_{2g+1}$ implies stability also considered as a sequence of representations of $\fS_{2g+2}$, finishing the proof. 
\end{proof}

\subsection{Congruence subgroups}

Let $R$ be a ring and $J \subset R$ an ideal. Recall that $\GL_n(J)$ denotes the kernel of $\GL_n(R) \m \GL_n(R/J)$ and that $\GL_n^{\mathfrak U}(R/J)$ denotes the group of matrices with determinant in the image of $R^\times \m R/J$. When the map $\GL_n(R) \m \GL_n^{\mathfrak U}(R/J)$ is surjective, the homology groups $\{H_i(\GL_n(J))\}_n$ assemble to form a $U\GL^{\mathfrak U}(R)$-module which we denote by $H_i(\GL(J))$. We now prove \autoref{thmCong}.

\begin{proof}[Proof of \autoref{thmCong}]

By Gan--Li \cite[Theorem 11]{GanLiCongruenceSubgroups}, as a $U\fS$-module, $H_i(\GL(J))$ is generated in degree $4i+2s-1$ and presented in degree $4i+2s+4$. By \cite[Theorem 3.30]{MPW},  $H_i(\GL(J))$ is polynomial of degree $\leq 4i+2s-1$ in ranks $ >8i+4s+7$. It follows from \autoref{rem:ZXpoly} that a $U\GL^{\mathfrak U}(R/J)$-module has polynomial degree $\leq r$ in ranks $>d$ if and only if its underlying $U\fS$-module does. By \autoref{H3examples}, the category $U\GL^{\mathfrak U}(R/J)$ satisfies \con{H3($2,t+1$)}. By \autoref{PolynomialsCentral}, we have that  
\[ \widetilde H^{\GL^{\mathfrak U}(R/J)}_{-1}(H_i(\GL(J)))_n\quad \text{for $n> \max(8i+4s+7,4i+2s+t-2)$}\] 
and
\[ \widetilde H^{\GL^{\mathfrak U}(R/J)}_{0}(H_i(\GL(J)))_n\quad \text{for $n> \max(8i+4s+8,4i+2s+t)$.}\] 
Because $t\ge 1$ and the category $U\GL^{\mathfrak U}(R/J)$ satisfies \con{H3($2,t+1$)},  \autoref{thm:res of finite type} implies that $H_i(\GL(J))$ is presented in degree $\le \max(8i + 4s +t+8, 4i +2s + 2t -1)$.
\end{proof}

\subsection{Diffeomorphism groups}

In this subsection, we prove a secondary stability result for the homology of diffeomorphism groups viewed as discrete groups. Given an orientable smooth manifold $M$, let $\Diff(M)$ denote the topological group of smooth orientation preserving diffeomorphisms $M \m M$ topologized with the $C^\infty$-topology. If $M$ has boundary, we assume that the diffeomorphisms fix a neighborhood of the boundary. Let $\Diff^\delta(M)$ denote the group $\Diff(M)$ topologized with the discrete topology. We first prove a secondary stability result for moduli spaces of surfaces with highly connected $\theta$-structures. Then we use Mather--Thurston theory as in Nariman \cite{Nar} to deduce our results for diffeomorphism groups. 

\begin{figure}[h]
\labellist
\pinlabel $M_{1,1}$ at 150 -30
\pinlabel $M_{2,1}$ at 650 -30
\Huge
\pinlabel $\subset$ at 370 120
\endlabellist
\includegraphics[width = \textwidth]{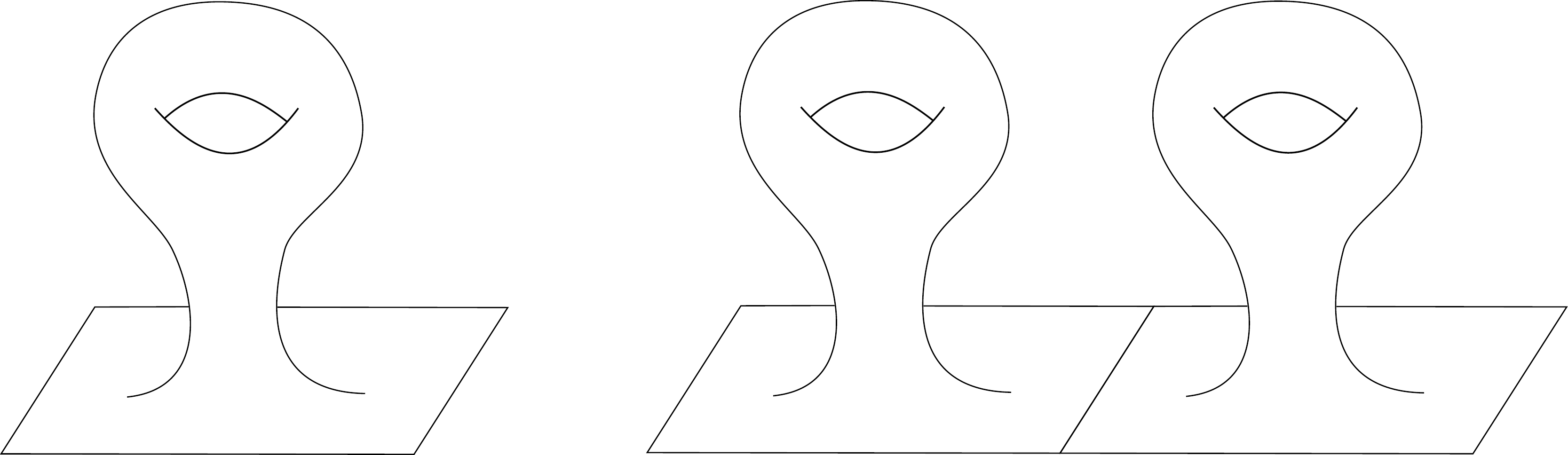}
\vspace{.5ex}
\caption{$M_{1,1} \subset M_{2,1}$}\label{fig:Minfty}
\end{figure}

Let $M_{1,1}$ be an orientable genus one surface with one boundary component. Let $M_{g,1}$ be the $g$-fold boundary connect sum of $M_{1,1}$ as in \autoref{fig:Minfty}. Let $\theta\colon B \m BSO(2)$ be a fibration, let $\gamma_2\colon ESO(2) \times_{SO(2)} \R^2 \m BSO(2)$ be the tautological bundle, and let $V=ESO(2) \times_{SO(2)} \R^2$ be the total space. Given $\pi\colon E \m B$ and $\pi'\colon E' \m B'$ vector bundles of the same dimension, let $\Bun(\pi,\pi')$ denote the space of bundle maps from $\pi$ to $\pi'$ topologized with the compact open topology. That is, an element of $\Bun(\pi,\pi')$ is a map $F\colon E \m E'$ covering a map $f\colon B \m B'$ such that $F$ restricts to give linear isomorphisms on each fiber. Fix $F_1$ covering $f_1$ in $\Bun(TM_{1,1},\theta^*\gamma_2)$ such that $F_1$ can be connect summed with itself to give a bundle map in $\Bun(TM_{2,1},\theta^*\gamma_2)$. Let $F_g \in \Bun(TM_{g,1},\theta^*\gamma_2)$ be the $g$-fold connect sum of $F_1$ and let $f_g\colon M_{g,1} \m B$ be the map that $F_g$ covers. Let $\Bun_c(TM_{g,1},\theta^*\gamma_2)$ denote the subspace of  $\Bun(TM_{g,1},\theta^*\gamma_2)$ of bundle maps that agree with $F_g$ in a neighborhood of the boundary. The group $\Diff(M_{g,1})$ acts on $\Bun_c(TM_{g,1},\theta^*\gamma_2)$ via the usual action of diffeomorphisms on the tangent bundle.

There are natural maps \[\Bun_c(TM_{a,1},\theta^*\gamma_2) \times \Bun_c(TM_{b,1},\theta^*\gamma_2) \m \Bun_c(TM_{a+b,1},\theta^*\gamma_2)\] which are $\Diff(M_{a,1}) \times \Diff(M_{b,1})$-equivariant. Here $\Diff(M_{a,1}) \times \Diff(M_{b,1}) $ acts on $\Bun_c(TM_{a+b,1},\theta^*\gamma_2)$ via the inclusion \[\Diff(M_{a,1}) \times \Diff(M_{b,1})  \m \Diff(M_{a+b,1}).  \] In particular, there are $\Diff(M_{g,1})$-equivariant maps \[\Bun_c(TM_{g,1},\theta^*\gamma_2) \m \Bun_c(TM_{g',1},\theta^*\gamma_2) \] for $g' \geq g$ which we view as stabilization maps. Since $\pi_0(\Diff(M_{g,1})) \cong \Mod_{g,1}$, we have that $H_i(\Bun_c(TM_{g,1},\theta^*\gamma_2))$ is a $\Mod_{g,1}$-representation. The stabilization maps induce $\Mod_{g,1}$-equivariant maps \[H_i(\Bun_c(TM_{g,1},\theta^*\gamma_2)) \m H_i(\Bun_c(TM_{g',1},\theta^*\gamma_2)). \] Since $\Mod_{g'-g,1}$ acts trivially on the image of this map, the representations $\{H_i(\Bun_c(TM_{g,1},\theta^*\gamma_2)) \}_g$ assemble to form a $U\Mod$-module which will denote by $V(i,\theta)$ (see Randal-Williams-Wahl \cite[Proposition 4.2]{RWW}). We will show this functor has finite polynomial degree. Before we can do this, we need to compare it to a space of sections of a bundle.  Given a bundle $\pi\colon E \m M_g$ with preferred section $\sigma\colon M_g \m E$, let $\Gamma_c(\pi)$ denote the space of sections of $\pi$ that agree with $\sigma$ on a neighborhood of the boundary.

\begin{lemma}\label{homotopicToSectionSpace}
There is a weak homotopy equivalence \[\Gamma_c(( \theta\circ f_g)^* \theta) \m \Bun_c(TM_{g,1} , \theta^*\gamma_2) \] where the preferred section of $(\theta\circ f_g)^* \theta$ is the one corresponding to $f_g$. 
\end{lemma}

\begin{proof}
The setup is summarized in the following diagram:
\[ \xymatrix{
\theta^*V \ar[r]^{\gamma_2^*\theta}\ar[d]_{\theta^*\gamma_2} & V \ar[d]^<<<{\gamma_2}\\
B \ar[r]^{\theta} & BSO(2)&TM_{g,1}\ar[lu]_{G_g}\ar[d] \ar[llu]^>>>>>>>{F_g}\\
&(\theta \circ f_g)^*B \ar[lu]^{\theta^*(\theta\circ f_g)} \ar[r]_>>>>{(\theta \circ f_g)^*\theta} & M_{g,1}\ar[lu]_{\theta \circ f_g} \ar[llu]^{f_g}
}\]

Let $G_g\colon TM_{g,1} \m \gamma_2$ be the composition of $F_g\colon TM_{g,1} \m \theta^*\gamma_2$ and $\gamma_2^*\theta\colon \theta^*\gamma_2 \m \gamma_2$. Use $G_g$ as the preferred bundle map to define $\Bun_c(TM_{g,1} ,\gamma_2)$.  The natural map 
\[\eta\colon \Bun_c(TM_{g,1} , \theta^*\gamma_2) \m \Bun_c(TM_{g,1} ,\gamma_2)\]
 is a fibration. Let $\Phi\colon TM_{g,1} \to \theta^* \gamma_2$ be in the fiber of $\eta$ over $G_g$ and $\phi\colon M_{g,1} \to B$ the map of spaces which $\Phi$ covers. Using the universal property of the pull back $(\theta\circ f)^*B$, $\phi$ defines a section $s_\phi \in\Gamma_c((\theta\circ f_g)^* \theta$. Similarly, a section $s\in \Gamma_c((\theta\circ f_g)^* \theta)$ defines a map $\phi_s = \theta^*(\theta\circ f_g) \circ s$ and a cover $\Phi_s\colon TM_{g,1} \to \theta^*V$ by the universal property of the pull back $\theta^*V$. This proves that the fiber of $\eta$ over $G_g$ is naturally homeomorphic to $\Gamma_c((\theta\circ f_g)^* \theta)$.  By \cite[Lemma 5.1]{GTMW}, $\Bun_c(TM_{g,1} ,\gamma_2)$ is weakly contractible so $\Gamma_c((\theta\circ f_g)^* \theta) \m \Bun_c(TM_{g,1} , \theta^*\gamma_2)$ is a weak homotopy equivalence. 
\end{proof}

\begin{proposition}
\label{sectionIsPoly}
Fix an abelian group $A$ and $n \geq 2$. Assume that the fibers of $\theta$ are $K(A,n)$'s. Then $V(i,\theta)$ is polynomial of degree $\leq i$. 
\end{proposition}

\begin{proof}
Ignoring the $\Mod_{g,1}$-action, $V(i,\theta)_g \cong H_i(\Gamma_c((\theta\circ f_g)^* \theta)$ by \autoref{homotopicToSectionSpace}. Since the tangent bundle of $M_{g,1}$ is trivial, in fact $V(i,\theta)_g \cong H_i(\Map_*(M_{g,1},K(A,n))$ where $\Map_*$ denotes the based mapping space. Thus, the underlying $U1$-module of $V(i,\theta)$ is the same as the module considered in Cohen--Madsen \cite[Example (2) in Section 1.1]{CM}. They prove that  \[g \mapsto H_i(\Map_*(M_{g,1},K(A,n)) \] is polynomial of degree $\le i$. By \autoref{rem:ZXpoly}, this implies the assertion.
%
%This proof is almost identical to that of Cohen--Madsen \cite[Example (2) in Section 1.1]{CM} but with a few modifications. Note that as abelian groups 
%
%We will prove this by induction over $i$. For $i=0$,
%
%First, they consider only the case that the action on $E$ comes exclusively from the action of $\Diff(M_{\infty,1})$ on $M_{\infty,1}$. However, this is never used in their argument. Second, they consider a surface category that allows surfaces of with more than one boundary component. Thus, they show that the cokerenel of stabilizing by pairs of pants and reverse pairs of pants are of lower polynomial degree and show that the kernels are trivial. Since a torus is the composition of a pair of pants and a reverse pair of pants, the claim follows. 
\end{proof}

Let $\sslash$ denote a choice of homotopy quotient functor. Let $\cM_{g,1}(\theta)= \Bun_c(TM_{g,1} , \theta^*\gamma_2)\sslash \Diff(M_{g,1})$. This is the moduli space of surfaces with $\theta$-structure considered by Randal-Williams \cite{RWres}. There is a natural map $\cM_{g,1}(\theta) \m B\Diff(M_{g,1})$ which forgets the $\theta$-structure. Let $\overline B$ denote the fiber of $\theta$.

\begin{lemma} \label{surjec}
If $\theta\colon  B \to BSO(2)$ is $4$-connected, then $H_2(\cM_{g,1}(\theta) )\m  H_2(B\Diff(M_{g,1})) $ is surjective.
\end{lemma}
\begin{proof}
We will show $\cM_{g,1}(\theta) \m B\Diff(M_{g,1})$ is $2$-connected. The fiber of $\cM_{g,1}(\theta) \m B\Diff(M_{g,1})$ is $\Bun_c(TM_{g,1} , \theta^*\gamma_2)$ so it suffices to show this fiber is $1$-connected. As in the proof of \autoref{sectionIsPoly}, we have that $\Bun_c(TM_{g,1} , \theta^*\gamma_2) \simeq \Map_*(M_{g,1}, \overline B)$. Since $M_{g,1}$ is $2$-dimensional and  $\overline B$ is $3$-connected, $\Map_*(M_{g,1}, \overline B)$ is $1$-connected. 
\end{proof}

Let $\lambda \in H_2(B\Diff(M_{3,1})) \cong H_2(\Mod_{3,1})$ be a homology class that pairs with the first kappa class to give $12$. Galatius--Kupers--Randal-Williams \cite[Page 2]{GKRW2} proved that the map inducing secondary stability agrees with the map induced by boundary connect sum with $\lambda/10$ if you work with coefficients where $10$ is invertible. For simplicity, we will prove our secondary stability result with $\Z[\frac{1}{10}]$-coefficients even though an integral result is likely also true. 

From now on, we assume that $\theta$ is $4$-connected. Let $\lambda' \in H_2(\cM_{3,1}(\theta) ;\Z[{\frac1{10}}]) $ be a class that maps to $\lambda/10$. The  $\Diff(M_{a,1}) \times \Diff(M_{b,1})$-equivariant map  \[\Bun_c(TM_{a,1},\theta^*\gamma_2) \times \Bun_c(TM_{b,1},\theta^*\gamma_2) \m \Bun_c(TM_{a+b,1},\theta^*\gamma_2)\] described above induces a map \[ \cM_{a,1}(\theta) \times \cM_{b,1}(\theta) \m \cM_{a+b,1}(\theta).\] In particular, this gives a map of spaces $\cM_{g-1,1}(\theta) \m \cM_{g,1}(\theta)$ which lets us make sense of $H_i(\cM_{g,1}(\theta),\cM_{g-1,1}(\theta))$. Plugging in the class $\lambda'$ gives us a map: \[ t_{\lambda'} \colon  H_{i-2}(\cM_{g-3,1}(\theta),\cM_{g-4,1}(\theta);\Z[\textstyle\frac{1}{10}])  \m H_i(\cM_{g,1}(\theta),\cM_{g-1,1}(\theta);\Z[\textstyle\frac{1}{10}]). \] We will show that this map is an isomorphism in a range.

\begin{lemma} \label{stablilityKpi}
Assume that $\overline B \simeq K(A,n)$ with $n \geq 4$. Then $H_i(\cM_{g,1}(\theta),\cM_{g-1,1}(\theta)) \cong 0$ for $i <
\frac{2}{3}g$ and  \[ t_{\lambda'} \colon  H_{i-2}(\cM_{g-3,1}(\theta),\cM_{g-4,1}(\theta);\Z[\textstyle\frac{1}{10}])  \m H_i(\cM_{g,1}(\theta),\cM_{g-1,1}(\theta);\Z[\textstyle\frac{1}{10}]) \]  is a surjection for $i < \frac{3}{4}g$ and an isomorphism for $i < \frac{3}{4}g-1$.

\end{lemma}

\begin{proof}
For the first claim, consider the relative Serre spectral sequence associated to  the fiber sequence \[ \Bun_c(TM_{g-1,1} , \theta^*\gamma_2)  \m \cM_{g-1,1}(\theta) \m  B\Diff(M_{g-1,1})   \] mapping to
 \[ \Bun_c(TM_{g,1} , \theta^*\gamma_2)  \m \cM_{g,1}(\theta) \m  B\Diff(M_{g,1})  .   \]  This has the form $E^2_{p,q}=$ \[H_p(B \Diff(M_{g,1}),B \Diff(M_{g-1,1});H_q(\Bun_c(TM_{g,1} , \theta^*\gamma_2)),H_q(\Bun_c(TM_{g-1,1} , \theta^*\gamma_2)))\] \[ \implies  H_{p+q}(\cM_{g,1}(\theta),\cM_{g-1,1}(\theta)). \] It follows from the homological stability portion of \autoref{mappingclassgroups} and \autoref{sectionIsPoly} that $E^2_{p,q} \cong 0$ for $p <
\frac{2}{3}(g- q)$ and thus \[H_i(\cM_{g,1}(\theta),\cM_{g-1,1}(\theta)) \cong 0\] for $i < \frac{2}{3}g$. The second claim is proven analogously instead using the secondary stability portion of \autoref{mappingclassgroups}. Here we uses a more sophisticated  version of the relative Serre spectral sequence for the mapping cone of a chain-level lift of $t_{\lambda'}$. This is similar in spirit to the spectral sequence appearing in  \cite[Lemma 2.44]{kupersmillercells}.
\end{proof}

\begin{theorem} \label{stablility4}
Assume that $\theta$ is $4$-connected. Then \[H_i(\cM_{g,1}(\theta),\cM_{g-1,1}(\theta)) \cong 0\] for $i <
\frac{2}{3}g$ and  \[ t_{\lambda'} \colon  H_{i-2}(\cM_{g-3,1}(\theta),\cM_{g-4,1}(\theta);\Z[\textstyle\frac{1}{10}])  \m H_i(\cM_{g,1}(\theta),\cM_{g-1,1}(\theta);\Z[\textstyle\frac{1}{10}]) \]  is a surjection for $i < \frac{3}{4}g$ and an isomorphism for $i < \frac{3}{4}g-1$.

\end{theorem}

\begin{proof}
The claim follows from \autoref{stablilityKpi} and induction up a relative Postnikov tower for the map $\theta$. See Cohen--Madsen \cite[proof of Theorem 2.3]{CM} for a similar argument. 
\end{proof}

The homological stability  portion of the above theorem is a special case of work of Randal-Williams \cite{RWres} but the secondary stability portion is new. We now prove secondary stability for diffeomorphism groups viewed as discrete groups. Extension by the identity gives maps of groups $\Diff(M_{g-1,1}) \m \Diff(M_{g,1}) $ so we can make sense of $H_{i}(B\Diff^\delta(M_{g,1}),B\Diff^\delta(M_{g-1,1}))$.

\begin{theorem} 
We have that \[H_{i}(B\Diff^\delta(M_{g,1}),B\Diff^\delta(M_{g-1,1}) )\cong 0\] for $i <
\frac{2}{3}g$. Moreover, there is a map 
\[H_{i-2}(B\Diff^\delta(M_{g-3,1}),B\Diff^\delta(M_{g-4,1});\Z[\textstyle\frac{1}{10}] ) \m H_{i}(B\Diff^\delta(M_{g,1}),B\Diff^\delta(M_{g-1,1});\Z[\textstyle\frac{1}{10}] )\]
which is a surjection for $i <
\frac{3}{4}g$ and an isomorphism for $i < \frac{3}{4}g-1$.
\end{theorem}

\begin{proof} Let $\theta$ be the natural map from Haefliger’s classifying space of orientable foliations of codimension 2 to $B \GL_2^+(\R) \simeq BSO(2)$. 
Nariman \cite[Lemma 1.13]{Nar} showed there is a natural homology equivalence between $B\Diff^\delta(M_{g,1})$ and $\cM_{g,1}(\theta)$. As explained in \cite[Remark 1.5]{Nar}, it follows from the work of Thurston that $\theta$ is $4$-connected and so the claim follows by \autoref{stablility4}. \end{proof}

%  there is a bundle $\mathcal E_g \m M_{g,1}$ and a homology equivalence. \[B\Diff^{\delta}(M_{g,1})  \m \mathcal E_g\sslash \Diff( M_{g,1}).\] Since this bundle is functorialy built out of tangent bundle and  $M_{g,1}$ is parallelizable, it is in fact a trivial bundle. By the work of Thurston \cite{Thur}, the fiber of this bundle is $3$-connected (see Nariman \cite[Remark 1.5]{Nar}). The claim now follow by \autoref{stablility3}.

The homological stability portion of the above theorem is due to Nariman \cite{Nar} but the secondary stability portion is new and is \autoref{DiffSec}.

%\begin{theorem}[Mather--Thurston] Let $M$ be a $d$-dimensional manifold with boundary. There a bundle $\mathcal E \m M$ with $d$-connected fibers, a section $\sigma \colon  M \m \mathcal E$, and a homology $g\colon B\Diff^\delta(M) \m \Gamma(\mathcal E)\sslash\Diff(M)$. 
%\end{theorem}

\begin{remark}
Sam Nariman suggested to us that it might be possible to use the techniques of this subsection to prove secondary homological stability for discrete symplectomorphism groups $\Symp^\delta(M_{g,1})$. A few difficulties arise. By the work of McDuff \cite{McDuff82}, $\Symp^\delta(M_{g,1})$ is homology equivalent to a connected component of a space of sections of a bundle, not the entire space of sections. Plausibly this can be dealt with as in Nariman \cite{NarSymp}. Furthermore, the bundle map associated to this section space is only $2$-connected, not $4$-connected as is required here. However, it follows from work of Kotschick--Morita \cite{KotMor} that $H_2(B\Symp^\delta(M_{g,1})) \m H_2(B\Diff(M_{g,1}))$ is surjective which may be enough to make the arguments go through. 

There are other natural families of subgroups of $\Diff(M_{g,1})$ whose homology groups do not surject onto $H_2(\Diff(M_{g,1}))$, for example the extended Hamiltonian group of $M_{g,1}$. Likely these groups also exhibit some form of secondary stability but of a flavor different from that of $B\Diff(M_{g,1})$ and $B\Diff^\delta(M_{g,1})$. 
\end{remark}

%This bundle $\mathcal E$ is constructed functorially out of $M$ using Haefliger’s classifying space of foliations \cite{Haef}. Given an open embedinng $f\colon N \m M$, there is a natural isomorphism $f^* \mathcal E_M \cong \mathcal E_N$ and this compatible with the choice of section.

\subsection{Homotopy automorphisms and $\GL_n(\mathbb S)$}

The goal of this subsection is to prove an improved range for homological stability for the monoid of homotopy automorphisms of wedges of $d$-dimensional spheres with coefficients in $\Z[\frac 12]$. Specializing this result to $d= \infty$ will yield homological stability for $\GL_n(\mathbb S)$.

\begin{definition}
Let $X$ and $Y$ be based spaces. Let $\Map_*(X,Y)$ denote the space of based maps, topologized with the compact open topology. Let $\hAut(X) \subset \Map_*(X,Y)$ denote the topological monoid of self homotopy equivalences topologised with the subspace topology. The monoid structure is induced by function composition. %Let $\hAut_{\id}(X)$ denote the submonoid of homotopy automorphisms which are homotopic to the identity.
We denote $\Map_*(S^n,X)$ by $\Omega^n X$. 
%Let $\Omega^n_0 X$ denote the connected component of the constant map in $\Omega^n X$.
Let $B$ denote the bar construction for topological monoids/$E_1$-spaces. 
Let $\bigvee_n X$ denote the $n$-fold wedge product of $X$.  Let $\Sigma^d_+$ denote the $d$-fold suspension functor precomposed with the disjoint basepoint functor.

\end{definition}

 Sending a homotopy automorphism to the induced map on $H_d$ gives a map of monoids $\hAut(\bigvee_n S^d) \m \GL_n(\Z)$. For $d \geq 2$, this map is an isomorphism on $\pi_0$.
 % and has kernel $\hAut_{\id}(\bigvee_n S^d)$.
 The action of the fundamental group on higher homotopy groups gives an action of $\GL_n(\Z)$ on $\pi_i(B \hAut(\bigvee_n S^d) )$ for $i \geq 2$.
 Extending a homotopy automorphism of $\bigvee_n S^d$ to a homotopy automorphism of $\bigvee_{n+1} S^d$ via the identity map on the $(n+1)$st sphere induces a $\GL_n(\Z)$-equivariant map $\pi_i(B \hAut(\bigvee_n S^d) ) \m \pi_i(B \hAut_\id(\bigvee_{n+1} S^d) )$.

\begin{lemma} For $d,i \geq 2$, the sequence $\{\pi_i(B\hAut(\bigvee_n S^d))\}_n$ has the structure of a $U\GL(\Z)$-module which we denote by $\pi_i(B\hAut(\bigvee S^d))$. 
\end{lemma}

\begin{proof}
We say that a homotopy automorphism $f\colon\bigvee_n S^d \m \bigvee_n S^d$ is supported on a collection of spheres $T \subset \{1,\dots,n\}$ if $f$ is the wedge of a homotopy automorphism of $\Sigma^d_+ T$ with the identity on $\Sigma^d_+ ( \{1,\dots,n\} \setminus T)$. Here we view $\bigvee_n S^d$ as \[\left(\Sigma^d_+ T \right) \vee \left( \Sigma^d_+ ( \{1,\dots,n\} \setminus T) \right).\]
By Randal-Williams--Wahl \cite[Proposition 4.2]{RWW}, it suffices to show that conjugation by an automorphism supported on the last $m$ spheres acts trivially on the image of \[\pi_i(B \hAut(\bigvee_n S^d) ) \m \pi_i(B \hAut(\bigvee_{n+m} S^d) ).\] This follows from the fact that homotopy automorphisms with disjoint support commute. 
\end{proof}

We now recall Church--Ellenberg--Farb's definition of $\FI\sharp$-modules \cite[Section 4.1]{CEF}.

\begin{definition}
Let $\FI\sharp$ be the category of finite sets with elements of ${\FI\sharp}(S,T)$ given by injections $f\colon U \m T$ with $U$ a subset of $S$. An $\FI\sharp$-module is a functor from $\FI\sharp$ to the category of abelian groups. We identify $\FI$ with the subcategory of $\FI\sharp$ where we require that $U=S$.
\end{definition}

We can view an $\FI\sharp$-module as an $\FI$-module via restriction. The following follows immediately from \cite[Proposition 3.23]{MPW} and Church--Ellenberg--Farb \cite[Theorem 4.1.5]{CEF}.

\begin{proposition} \label{SharpGen}
Let $A$ be a $\FI\sharp$-module with generation degree $\leq r$. Then the underlying $U\fS$-module of $A$ is polynomial of degree $\le r$ in ranks $>-1$.
\end{proposition}

The proof of Church--Ellenberg--Farb \cite[4.1.7]{CEF} gives the following.

\begin{theorem} \label{polyranksdeg}
Let $A$ be an $\FI\sharp$-module over $\bk$ where the number of generators of $A_n$ as a $\bk$-module is bounded by a polynomial of degree $r$. Then $A$ has generation degree $\leq r$ as a $U\fS$-module. 
\end{theorem}

Combining \autoref{rem:ZXpoly}, \autoref{SharpGen}, and \autoref{polyranksdeg} gives the following.

\begin{corollary} \label{SharpCor}
Let $\G$ be a symmetric stability groupoid with a map $\fS \m \G$. Let $A$ be a $U\G$-module such that the $U\fS$-module extends to an $\FI\sharp$-module. If the number of generators of $A_n$ is bounded by a polynomial of degree $r$, then $A$ is a polynomial $U\G$-module in ranks $>-1$ of degree $\leq r$. 
\end{corollary}

The following is a special case of Hilton--Milnor splitting  \cite{HiltonMilnor}.

\begin{proposition} \label{HM}
For $d\ge 1$, there is a homotopy equivalence: 
\[ \Omega^d \left(\bigvee_n S^d \right) \simeq \prod_{m \geq 1} \left(\Omega^d  S^{(d-1)m +1} \right)^{l_{m,n}}. \] 
where $l_{m,n}$ is the rank of the submodule of the free Lie algebra on $n$ generators spanned by $m$ nested brackets of generators. 
\end{proposition}

\begin{corollary} \label{HMcor}
Let $\cL(m)_n$ denote the submodule of the free Lie algebra on $n$ generators spanned by $m$ nested brackets of generators. For $d,i \geq 2$, we have that 
\[\pi_i(B\hAut(\bigvee_n S^d))  \cong \left( \bigoplus_{1\le m \le \textstyle\frac {i-1}{d-1}+1}  \pi_{i+d-1}(S^{(d-1)m+1}) \otimes \cL(m)_n \right)^n. \] 
\end{corollary}

\begin{proof}

Fix $i \geq 2$. We have that 
\begin{gather*}
\pi_{i}(B \hAut(\bigvee_n S^d))\cong \pi_{i-1}( \hAut(\bigvee_n S^d) )\cong \pi_{i-1}\left( \Map_*(\bigvee_n S^d,\bigvee_n S^d) \right)\\
\cong \pi_{i-1}\left(\left ( \Omega^d \bigvee_n S^d \right)^n \right) \cong \left(\pi_{i-1}\left( \Omega^d \bigvee_n S^d \right)\right)^n \cong  \pi_{i-1} \left(  \prod_{m \geq 1} \left(\Omega^d  S^{(d-1)m +1} \right)^{l_{m,n}} \right)^n\\
\cong \left( \bigoplus_{m \ge 1}  \pi_{i+d-1}(S^{(d-1)m+1}) \otimes \cL(m)_n \right)^n.
\end{gather*}
The assertion follows because 
\[ \pi_{i+d-1}(S^{(d-1)m+1}) \cong 0 \]
for $m>\frac {i-1}{d-1}+1$.
\end{proof}

We now bound the polynomial degree of $\pi_i(B\hAut(\bigvee S^d))$.

\begin{proposition}\label{prop:polydegpi}
For $d\ge 3$ and $i \geq 2$, the $U\GL(\Z)$-module $ \pi_i(B\hAut_{\id}(\bigvee S^d))$ is polynomial of degree $\leq i$ in ranks $>-1$.
\end{proposition}

\begin{proof}

We will first show that the mapping $n \mapsto  \cL(m)_n$ assembles to an $\FI \sharp$-module which we will call $\cL(m)$. It is clear that there is a $U \fS$-module $\cL(m)$ whose value on $n$ is $\cL(m)_n$. This $U \fS$-module factors as the composition of the functor $\Z U\fS(1,-)$ with a functor $\Ab \m \Ab$. By Church--Ellenberg--Farb \cite[Theorem 4.1.5]{CEF}, representable functors $\Z U\fS(m,-)$ are $\FI\sharp$-modules and thus $\cL(m)$ has the structure of an $\FI\sharp$-module.

We have an isomorphism of $U\fS$-modules
\[ \pi_i(B\hAut_{\id}(\bigvee S^d)) \cong  \bigoplus_{1\le m \le \textstyle\frac {i-1}{d-1}+1}  \pi_{i-d+1}(S^{(d-1)m+1}) \otimes \cL(m)\otimes \Z U\fS(1,-) .\] % \otimes \Z U\fS(1,-) 
In particular, $\pi_i(B\hAut_{\id}(\bigvee S^d))$ has the structure of an $\FI\sharp$-module. Note that $l_{m,n} \leq n^m$ so the number of generators of $\pi_i(B\hAut(\bigvee_n S^d ))$ as an abelian group is bounded by polynomial of degree $ \leq \frac {i-1}{d-1}+2$ in $n$.  \autoref{SharpCor} implies that $\pi_i(B\hAut(\bigvee S^d )) $ is a polynomial $U\GL_n(\Z)$-module of degree $\leq \frac {i-1}{d-1}+2$. To prove the assertion, note that the floor of $\frac{i-2}{d-1} +2$ is at most $i$ if $d\ge 3$ and $i\ge 2$.
%
%Observe that $\pi_{i-d+1}(S^{(d-1)m+1}) \cong 0$ for $m \geq \frac{i}{d-1}-1 \geq i$. Combining this with  \autoref{HMcor} shows that \[\pi_i(B\hAut(\bigvee_n S^d ))  \cong \left( \bigoplus_{1 \leq m \leq i}  \pi_{i-d+1}(S^{(d-1)m+1}) \otimes (L_{m}) \right)^n. \] Now note that $l_{m,n} \leq n^m$ so the number of generators of $L_m(n)$ is bounded by a degree $m$ polynomial in $n$. We see that  $\pi_i(B\hAut(\bigvee S^d )) $ is an $\FI\sharp$-module with number of generators of $\pi_i(B\hAut(\bigvee_n S^d) ) $ bounded by $?????$. \autoref{SharpCor} implies that $\pi_i(B\hAut(\bigvee S^d ) $ is a polynomial $U\GL_n(\Z)$-module of degree $\leq ????$. 
%
% \todo{Fix ranges in this proof. Also there are cross terms that we don't deal with correctly.}
\end{proof}

\begin{remark}It is natural to study $\{\pi_i(B\hAut(\bigvee X))\}$ for more general spaces $X$. In general this will only form a $U\fS$-module. Lindell and Saleh \cite{LindellSaleh} have shown that the rational homotopy groups $\{\pi_i^{\Q}(B\hAut(\bigvee X))\}$ define a finitely generated $U\fS$-module, for any simply connected $X$ of finite type. \end{remark}

The following is a direct application of the work of Eilenberg--MacLane \cite[Section 20]{EMII}; also see\textbf{} Dwyer \cite[Lemma 4.3]{DwyerTwisted}. 

\begin{proposition} \label{DwyerCor}
Let $A$ and $B$ be a polynomial module $U \G$-modules of polynomial degree $\leq a$ and $\leq b$ respectively. Then $n \mapsto H_i(K(A_n,j);B_n)$ is polynomial of degree $\le  \frac {a\cdot i}j + b$.
\end{proposition}

%\begin{proof}
%See Dwyer \cite[Section 3]{DwyerTwisted} for a notion of degree of functors of several variables. The functor $U\G \to \Ab^2$ given by $n \mapsto (A_n,B_n)$ has degree $\le \max( a,b)$ and the functor $\Ab^2 \to \Ab$ given by $(A,B) \mapsto H_i(K(A,j);B)$ has degree $\frac ij -1$ by \cite[Lemma 4.3]{DwyerTwisted}. Then \cite[Lemma 3.2]{DwyerTwisted} implies that the composition has degree  $\le \max(a,b) \cdot ( \frac ij -1)$.
%\end{proof}

We now prove \autoref{HAutThm} which is homological stability for $B\hAut(\bigvee_n S^d)$. This theorem is equivalent to showing that \[H_i(B\hAut(\bigvee_{n} S^d),B\hAut(\bigvee_{n-1} S^d));\Z[\textstyle\frac 12]) \cong 0 \text{ for }i \leq \textstyle\frac{2}{3}n\] for $d \geq 3$.

% \autoref{HAutThm}.

%\begin{theorem}
%For $d\ge 2$, 
 %\[H_i(B\hAut(\bigvee_{n} S^d),B\hAut(\bigvee_{n-1} S^d);\Z[\textstyle\frac12]) \cong 0\quad \text{for }i \leq \textstyle\frac{2}{3}n.\] 
%\end{theorem}

\begin{proof}[Proof of \autoref{HAutThm}]
Let $P_r(n)$ denote $r$th stage of the Postnikov tower of $B\hAut(\bigvee_n S^d)$ and let $K_{r}(n)=K(\pi_r(B\hAut(\bigvee_n S^d )),r) $. Let $A$ be a $U\GL_n(\Z)$-module over $\Z[\frac 12]$ of polynomial degree $\leq a$ in ranks $>-1$ . 
We will prove by induction that 
\begin{equation}\label{eq:induction}H_i(P_r(n),P_r(n-1) ;A_n,A_{n-1}     ) \cong 0 \quad\text{for $i<\textstyle\frac23n-a$.}\end{equation}
Observe that
\[ H_i(P_{i+1}(n),P_{i+1}(n-1) ;\Z[\textstyle\frac12]\displaystyle     ) \cong H_i(B\hAut(\bigvee_{n} S^d),B\hAut(\bigvee_{n-1} S^d);\Z[\textstyle\frac12]).\]
Thus establishing \eqref{eq:induction} establishes the theorem.

Since $P_1(n)=B\GL_n(\Z)$, \autoref{improvedRangeGLn}  establishes the induction beginning. Now assume we have proven the claim for all $r<R$ for some $R\ge 2$. Consider the relative Serre spectral sequence for twisted homology associated to the map of fibrations 
\[K_R(n-1) \m P_{R}(n-1) \m P_{R-1}(n-1) \] 
mapping to 
\[K_R(n) \m P_{R}(n) \m P_{R-1}(n)  \] 
where the first fiber sequence has coefficients in $A_{n-1}$ and the second has coefficients in $A_n$. This spectral sequence has its $E^2$-page given by
\[E^2_{p,q} \cong H_p(P_{R-1}(n),P_{R-1}(n-1);H_q(K_R(n),K_R(n-1);A_n,A_{n-1})) \] 
and converges to $H_{p+q}(P_R(n),P_R(n-1);A_n,A_{n-1})$. 
Combining induction hypothesis, \autoref{prop:polydegpi}, and \autoref{DwyerCor} shows that 
\[ E^2_{p,q} \cong 0 \quad\text{for $p< \textstyle\frac23n -R\cdot \frac qR- a$.}\]
In particular, this is true for $p+q < \frac23n -a$, and thus
$H_{i}(P_R(n),P_R(n-1);A_n,A_{n-1}) \cong 0$ for $i<\frac 23n - a$.
\end{proof}

Recall that \[ \underset{d\to \infty}{\colim} \, H_i(B\hAut(\bigvee_{n}S^d)) \cong H_i(B \GL_n(\mathbb S)).\] Hence homological stability for these homotopy automorphism monoids implies homological stability for $B \GL_n(\mathbb S)$ (\autoref{GLS cor}).

\begin{remark}
It is an interesting question if a version of \autoref{HAutThm} holds for $d=1$ or $2$. For $d=1$, $B\hAut(\bigvee_n S^d )$ has the same homotopy type as the classifying space of the automorphism group of the free group on $n$ letters and can  be thought of as a moduli space of graphs. Hatcher--Vogtmann \cite[Proposition 1.2]{HV45} proved rational homological stability for $B\hAut(\bigvee_n S^1 )$ with a slope $\frac{4}{5}$ stable range. 
\end{remark}

\begin{remark}
In work in progress, the first two authors and Alexander Kupers have established a slope $1$ homological stable range for $\GL_n(\Z)$ with coefficients in $\Z[\frac 12]$.  This will allow one to improve the stable range for $B\hAut(\bigvee_{n}S^d)$ and $\GL_n(\mathbb S)$ to slope $1$ as well.
\end{remark}

\appendix
\section{A review of stability arguments and a heuristic overview of the paper}

\label{appendix}

In this appendix, we give a review of past stability arguments and explain how to generalize them to prove the theorems of the paper. This appendix will not involve rigorous proofs and can be freely ignored by any reader who does not find it helpful or enjoy this style of informal discussion. Many technical hypotheses will be omitted. We will first sketch the standard proof of homological stability for a family of groups. Then we will talk about how to prove stability with polynomial coefficients and how to prove representation stability with untwisted coefficients. Finally, we will sketch an approach to representation stability with polynomial coefficients. This will involve showing the polynomial coefficients themselves satisfy a form of representation stability. Apart from the final subsection, most of this is a summary of arguments appearing in \cite{RWW,PS,CE,Pa2} and others.

\subsection{Homological stability with untwisted coefficients}

Consider a family of groups 
\[G_0 \hookrightarrow G_1 \hookrightarrow G_2 \hookrightarrow\ldots \] such as the braid groups, general linear groups, etc.\ To prove homological stability for these groups, one considers a chain complex 
\[C^\G_n:= \Z \leftarrow \Z[ G_n /G_{n-1}]  \leftarrow \Z[ G_n /G_{n-2}] \leftarrow \ldots\]
This complex $C^\G_n$ is the reduced cellular chains on a certain $CW$-complex associated to the groups $G_n$. In many examples, one can use combinatorial or topological techniques to prove that $C^\G_n$ is highly connected, in a range increasing with $n$.

The filtration of $C^\G_n$ by homological degree induces a filtration of the homotopy orbits \[(C^\G_n)_{hG_n}:=C^\G_n \otimes^{\mathbb{L}}_{\Z G_n} \Z.\] The resulting spectral sequence has the form of \autoref{E1PageHU}. 

\begin{figure}[h!]    \centering \begin{tikzpicture} {\footnotesize
  \matrix (m) [matrix of math nodes,
    nodes in empty cells,nodes={minimum width=3ex,
    minimum height=5ex,outer sep=2pt},
    column sep=9ex,row sep=5ex, text height=1.5ex, text depth=0.25ex]{ 
     &[-4ex]   H_3(G_{n}) &  H_3(G_{n-1})   &   H_3(G_{n-2}) &[4ex]  H_3(G_{n-3})  &  \; \\  
     &[-4ex]   H_2(G_{n}) &  H_2(G_{n-1})   &   H_2(G_{n-2}) &[4ex]  H_2(G_{n-3})  & \;  \\          
     &[-4ex]  H_1(G_{n}) &  H_1(G_{n-1})   &   H_1(G_{n-2}) &[4ex]  H_1(G_{n-3}) &  \; \\             
      &[-4ex]  H_0(G_{n}) &  H_0(G_{n-1})   &   H_0(G_{n-2}) &[4ex]  H_0(G_{n-3}) & \; \\
&      &     &     & &\\       }; 

 \draw[-stealth, red] (m-1-3.west) -- (m-1-2.east) node [midway,above] {$d^1=\iota$} node [midway,below]{} ;
 \draw[-stealth, red] (m-2-3.west) -- (m-2-2.east) node [midway,above] {$d^1=\iota$} node [midway,below] {};
 \draw[-stealth, red] (m-3-3.west) -- (m-3-2.east) node [midway,above] {$d^1=\iota$} node [midway,below] {};
 \draw[-stealth, red] (m-4-3.west) -- (m-4-2.east) node [midway,above] {$d^1=\iota$} node [midway,below] {};

 \draw[-stealth, red] (m-1-4.west) -- (m-1-3.east) node [midway,above] {$d^1=0$} node [midway,below] {};
 \draw[-stealth, red] (m-2-4.west) -- (m-2-3.east) node [midway,above] {$d^1=0$} node [midway,below] {};
 \draw[-stealth, red] (m-3-4.west) -- (m-3-3.east) node [midway,above] {$d^1=0$} node [midway,below] {};
 \draw[-stealth, red] (m-4-4.west) -- (m-4-3.east) node [midway,above] {$d^1=0$} node [midway,below] {}; 

 \draw[-stealth, red] (m-1-5.west) -- (m-1-4.east) node [midway,above] {$d^1=\iota$} node [midway,below] {};
 \draw[-stealth, red] (m-2-5.west) -- (m-2-4.east) node [midway,above] {$d^1=\iota$} node [midway,below] {};
 \draw[-stealth, red] (m-3-5.west) -- (m-3-4.east) node [midway,above] {$d^1=\iota$} node [midway,below] {};
 \draw[-stealth, red] (m-4-5.west) -- (m-4-4.east) node [midway,above] {$d^1=\iota$} node [midway,below] {};

 \draw[-stealth, red] (m-1-6.west) -- (m-1-5.east);
 \draw[-stealth, red] (m-2-6.west) -- (m-2-5.east);
 \draw[-stealth, red] (m-3-6.west) -- (m-3-5.east);
 \draw[-stealth, red] (m-4-6.west) -- (m-4-5.east);

\draw[thick] (m-1-1.north east) -- (m-5-1.east) ;
\draw[thick] (m-5-1.north) -- (m-5-6.north east) ;

}
\end{tikzpicture}
\caption{$E^1$-page for untwisted homological stability} \label{E1PageHU}
\end{figure}  
 
The differentials on the $E^1$-page alternate between being equal to the stabilization map \[\iota :H_i(G_m) \m H_{i}(G_{m+1})\] or the zero map. This implies that the $E^2$-page has the form of \autoref{E2PageHU}. 

If we want to prove stability in homological degree $k$, we should study the $k$th row of the spectral sequence; more specifically, we should prove that the $k$th row of the $E^2$-page vanishes in a range increasing with $n$. By induction on $k$ we assume that all rows below the $k$th vanish in a range on the $E^2$-page. For $n$ large enough, this rules out differentials into the groups $\coker(H_k(G_{n-1}) \m H_k(G_{n}))$ and $  \ker(H_k(G_{n-1}) \m H_k(G_{n}))$, and all differentials out of these two groups leave the first quadrant. Thus, these groups agree with the corresponding values on the $E^\infty$-page. But by connectivity estimates of $C^\G_n$ we know that the $E^\infty$-page vanishes in a range, and we deduce homological stability.

%Suppose we know by induction that $H_0(G_n)$ and $H_1(G_n)$ stabilize. Then the spectral sequence has the form of \autoref{E2PageHUzero}. In other words, the bottom two rows vanishing in a range since their entries are the kernel or cokernel of the stabilization map. If we assume the chain complex $C^\G_n$ has vanishing homology in a range increasing with $n$, then the abutment of the spectral sequence, $H_*((C^\G_n)_{hG_n})$, will be zero in a range increasing with $n$. This implies $coker(H_2(G_{n-1}) \m H_2(G_{n}))$ and $  ker(H_2(G_{n-1}) \m H_2(G_{n}))$ both vanish and we obtain homological stability for $H_2(G_n)$. This is an instance of the inductive step in the standard argument proving homological stability.

\begin{figure}[h!]    \centering \begin{tikzpicture} {\footnotesize
  \matrix (m) [matrix of math nodes,
    nodes in empty cells,nodes={minimum width=3ex,
    minimum height=5ex,outer sep=2pt},
    column sep=9ex,row sep=5ex, text height=1.5ex, text depth=0.25ex]{ 
     &[-4ex]  \coker(H_3(G_{n-1}) \m H_3(G_{n})) &  \ker(H_3(G_{n-1}) \m H_3(G_{n})) &  \coker(H_3(G_{n-3}) \m H_3(G_{n-2})) &  &  \; \\  
     &[-4ex]  \coker(H_2(G_{n-1}) \m H_2(G_{n})) &  \ker(H_2(G_{n-1}) \m H_2(G_{n})) &  \coker(H_2(G_{n-3}) \m H_2(G_{n-2})) &   & \;  \\          
     &[-4ex]  \coker(H_1(G_{n-1}) \m H_1(G_{n})) &  \ker(H_1(G_{n-1}) \m H_1(G_{n})) &  \coker(H_1(G_{n-3}) \m H_1(G_{n-2})) &  &  \; \\             
      &[-4ex] \coker(H_0(G_{n-1}) \m H_0(G_{n})) &  \ker(H_0(G_{n-1}) \m H_0(G_{n})) &  \coker(H_0(G_{n-3}) \m H_0(G_{n-2})) & & \; \\
&      &     &     & &\\       };

\draw[thick] (m-1-1.north east) -- (m-5-1.east) ;
\draw[thick] (m-5-1.north) -- (m-5-6.north east) ;

}
\end{tikzpicture}
\caption{$E^2$-page for untwisted homological stability} \label{E2PageHU}
\end{figure}

%
%
%\begin{figure}[h!]    \centering \begin{tikzpicture} {\footnotesize
  %\matrix (m) [matrix of math nodes,
    %nodes in empty cells,nodes={minimum width=3ex,
    %minimum height=5ex,outer sep=2pt},
    %column sep=9ex,row sep=5ex, text height=1.5ex, text depth=0.25ex]{ 
  %%
     %\\  
     %&[-4ex]  coker(H_2(G_{n-1}) \m H_2(G_{n})) &  ker(H_2(G_{n-1}) \m H_2(G_{n})) &  &   & \;  \\          
     %&[-4ex]  0 &  0&  0 & 0 &  \; \\             
      %&[-4ex] 0 &  0 &  0 & 0 & 0 \; \\
%&      &     &     & &\\       }; 
%
 %
%
%
 %\draw[-stealth, red] (m-4-5.west) -- (m-2-2.east) node [midway,above] {$d^3$} node [midway,below] {};
%
%\draw[-stealth, red] (m-3-4.west) -- (m-2-2.east) node [midway,above] {$d^2$} node [midway,below] {};
%
%
%\draw[thick] (m-1-1.north east) -- (m-5-1.east) ;
%\draw[thick] (m-5-1.north) -- (m-5-6.north east) ;
%
%
%}
%\end{tikzpicture}
%\caption{$E^2$-page for untwisted homological stability with vanishing range shown} \label{E2PageHUzero}
%\end{figure}  

\subsection{Homological stability with twisted coefficients}

Consider a family of representations $A_n$ of the groups $G_n$ equipped with $G_n$-equivariant maps $A_n \m A_{n+1}$. Now suppose we want to prove stability for $H_i(G_n;A_n)$. Consider the chain complex:  
\[C^\G_n(A):= A_n \leftarrow  \Ind_{G_{n-1}}^{G_n} A_{n-1}  \leftarrow  \Ind_{G_{n-2}}^{G_n} A_{n-2} \leftarrow \ldots \]
We call this the \emph{central stability chains} of $A=\{A_n\}$. For $A_n=\Z$, this is the chain complex $C^\G_n$ considered when proving untwisted homological stability in the previous subsection. Considering the homotopy orbits $C^\G_n(A)_{hG_n}$ one obtains a spectral sequence with $E^1$-page of the form of \autoref{E1PageHTdrs}.

\begin{figure}[h!]    \centering \begin{tikzpicture} {\footnotesize
  \matrix (m) [matrix of math nodes,
    nodes in empty cells,nodes={minimum width=3ex,
    minimum height=5ex,outer sep=2pt},
    column sep=9ex,row sep=5ex, text height=1.5ex, text depth=0.25ex]{ 
     &[-4ex]   H_3(G_{n};A_n) &  H_3(G_{n-1};A_{n-1})   &   H_3(G_{n-2};A_{n-2}) &[4ex]  H_3(G_{n-3};A_{n-3})  &  \; \\  
     &[-4ex]   H_2(G_{n};A_n) &  H_2(G_{n-1};A_{n-1})   &   H_2(G_{n-2};A_{n-2}) &[4ex]  H_2(G_{n-3};A_{n-3})  & \;  \\          
     &[-4ex]  H_1(G_{n};A_n) &  H_1(G_{n-1};A_{n-1})   &   H_1(G_{n-2};A_{n-2}) &[4ex]  H_1(G_{n-3};A_{n-3}) &  \; \\             
      &[-4ex]  H_0(G_{n};A_n) &  H_0(G_{n-1};A_{n-1})   &   H_0(G_{n-2};A_{n-2}) &[4ex]  H_0(G_{n-3};A_{n-3}) & \; \\
&      &     &     & &\\       }; 

 \draw[-stealth, red] (m-1-3.west) -- (m-1-2.east) node [midway,above] {$d^1=\iota$} node [midway,below]{} ;
 \draw[-stealth, red] (m-2-3.west) -- (m-2-2.east) node [midway,above] {$d^1=\iota$} node [midway,below] {};
 \draw[-stealth, red] (m-3-3.west) -- (m-3-2.east) node [midway,above] {$d^1=\iota$} node [midway,below] {};
 \draw[-stealth, red] (m-4-3.west) -- (m-4-2.east) node [midway,above] {$d^1=\iota$} node [midway,below] {};

 \draw[-stealth, red] (m-1-4.west) -- (m-1-3.east) node [midway,above] {$d^1=0$} node [midway,below] {};
 \draw[-stealth, red] (m-2-4.west) -- (m-2-3.east) node [midway,above] {$d^1=0$} node [midway,below] {};
 \draw[-stealth, red] (m-3-4.west) -- (m-3-3.east) node [midway,above] {$d^1=0$} node [midway,below] {};
 \draw[-stealth, red] (m-4-4.west) -- (m-4-3.east) node [midway,above] {$d^1=0$} node [midway,below] {}; 

 \draw[-stealth, red] (m-1-5.west) -- (m-1-4.east) node [midway,above] {$d^1=\iota$} node [midway,below] {};
 \draw[-stealth, red] (m-2-5.west) -- (m-2-4.east) node [midway,above] {$d^1=\iota$} node [midway,below] {};
 \draw[-stealth, red] (m-3-5.west) -- (m-3-4.east) node [midway,above] {$d^1=\iota$} node [midway,below] {};
 \draw[-stealth, red] (m-4-5.west) -- (m-4-4.east) node [midway,above] {$d^1=\iota$} node [midway,below] {};

 \draw[-stealth, red] (m-1-6.west) -- (m-1-5.east);
 \draw[-stealth, red] (m-2-6.west) -- (m-2-5.east);
 \draw[-stealth, red] (m-3-6.west) -- (m-3-5.east);
 \draw[-stealth, red] (m-4-6.west) -- (m-4-5.east);

\draw[thick] (m-1-1.north east) -- (m-5-1.east) ;
\draw[thick] (m-5-1.north) -- (m-5-6.north east) ;

}
\end{tikzpicture}
\caption{$E^1$-page for twisted homological stability (using derived representation stability)} \label{E1PageHTdrs}
\end{figure}
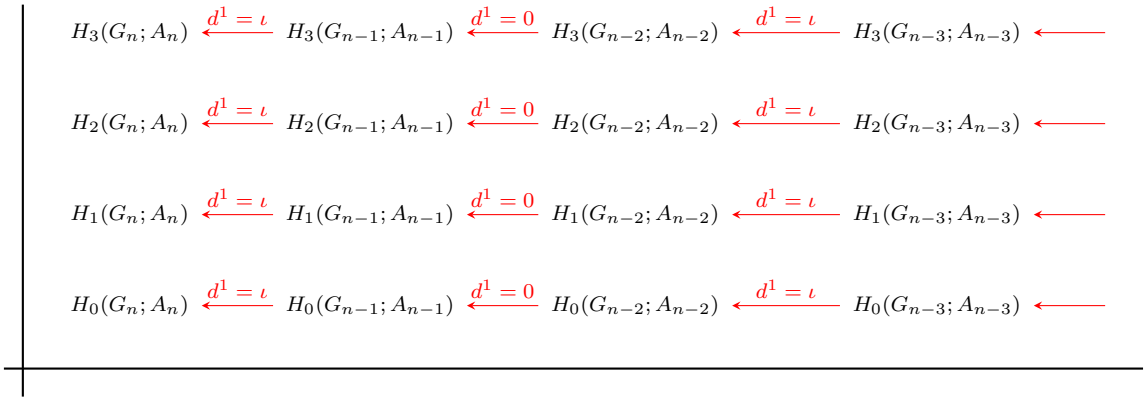  

This spectral sequence behaves in exactly the same way as the spectral sequence in the previous subsection. If one knows that $C^\G_n(A)$ has vanishing homology in a range increasing with $n$, then one can use the same argument to prove homological stability with coefficients in $A_n$. Unfortunately this chain complex does not come from a combinatorially defined $CW$-complex. In fact, vanishing of the homology of this chain complex is roughly equivalent to the condition we call \emph{derived representation stability}. Since it is often hard to check this condition, this approach is not commonly used; however, we will return to this approach soon.

One could instead just consider the chain complex $C^\G_n \otimes A_n$, which has vanishing homology in the same range as $C^\G_n$. Taking homotopy orbits we obtain a spectral sequence with $E^1$-page of the form of \autoref{E1PageHTnotRel}. The drawback of this spectral sequence is that it seems to be designed for comparing $H_i(G_n;A_n)$ with $H_i(G_{n-1};A_n)$, and we are instead interested in comparing $H_i(G_n;A_n)$ with $H_i(G_{n-1};A_{n-1})$.

\begin{figure}[h!]    \centering \begin{tikzpicture} {\footnotesize
  \matrix (m) [matrix of math nodes,
    nodes in empty cells,nodes={minimum width=3ex,
    minimum height=5ex,outer sep=2pt},
    column sep=9ex,row sep=5ex, text height=1.5ex, text depth=0.25ex]{ 
     &[-4ex]   H_3(G_{n};A_n) &  H_3(G_{n-1};A_{n})   &   H_3(G_{n-2};A_{n}) &[4ex]  H_3(G_{n-3};A_{n})  &  \; \\  
     &[-4ex]   H_2(G_{n};A_n) &  H_2(G_{n-1};A_{n})   &   H_2(G_{n-2};A_{n}) &[4ex]  H_2(G_{n-3};A_{n})  & \;  \\          
     &[-4ex]  H_1(G_{n};A_n) &  H_1(G_{n-1};A_{n})   &   H_1(G_{n-2};A_{n}) &[4ex]  H_1(G_{n-3};A_{n}) &  \; \\             
      &[-4ex]  H_0(G_{n};A_n) &  H_0(G_{n-1};A_{n})   &   H_0(G_{n-2};A_{n}) &[4ex]  H_0(G_{n-3};A_{n}) & \; \\
&      &     &     & &\\       }; 

 \draw[-stealth, red] (m-1-3.west) -- (m-1-2.east) node [midway,above] {} node [midway,below]{} ;
 \draw[-stealth, red] (m-2-3.west) -- (m-2-2.east) node [midway,above] {} node [midway,below] {};
 \draw[-stealth, red] (m-3-3.west) -- (m-3-2.east) node [midway,above] {} node [midway,below] {};
 \draw[-stealth, red] (m-4-3.west) -- (m-4-2.east) node [midway,above] {} node [midway,below] {};

 \draw[-stealth, red] (m-1-4.west) -- (m-1-3.east) node [midway,above] {} node [midway,below] {};
 \draw[-stealth, red] (m-2-4.west) -- (m-2-3.east) node [midway,above] {} node [midway,below] {};
 \draw[-stealth, red] (m-3-4.west) -- (m-3-3.east) node [midway,above] {} node [midway,below] {};
 \draw[-stealth, red] (m-4-4.west) -- (m-4-3.east) node [midway,above] {} node [midway,below] {}; 

 \draw[-stealth, red] (m-1-5.west) -- (m-1-4.east) node [midway,above] {} node [midway,below] {};
 \draw[-stealth, red] (m-2-5.west) -- (m-2-4.east) node [midway,above] {} node [midway,below] {};
 \draw[-stealth, red] (m-3-5.west) -- (m-3-4.east) node [midway,above] {} node [midway,below] {};
 \draw[-stealth, red] (m-4-5.west) -- (m-4-4.east) node [midway,above] {} node [midway,below] {};

 \draw[-stealth, red] (m-1-6.west) -- (m-1-5.east);
 \draw[-stealth, red] (m-2-6.west) -- (m-2-5.east);
 \draw[-stealth, red] (m-3-6.west) -- (m-3-5.east);
 \draw[-stealth, red] (m-4-6.west) -- (m-4-5.east);

\draw[thick] (m-1-1.north east) -- (m-5-1.east) ;
\draw[thick] (m-5-1.north) -- (m-5-6.north east) ;

}
\end{tikzpicture}
\caption{$E^1$-page of a not so useful spectral sequence} \label{E1PageHTnotRel}
\end{figure}
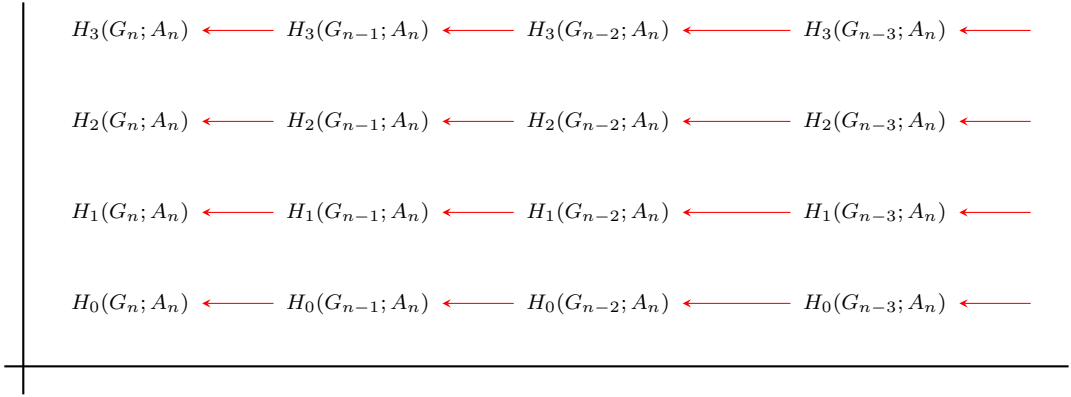

The spectral sequence of \autoref{E1PageHTnotRel} is nice since we can check that it converges to zero in a range. The spectral sequence of \autoref{E1PageHTdrs} is nice because it has the desired $E^1$-page. If these two spectral sequences were equal in a range, then we could profit off of the desirable properties of each spectral sequence. This is where the polynomial condition comes in. Unlike derived representation stability, polynomiality is often straightforward to verify.

Polynomiality implies that for large $n$, if  $n<m$, then the map $A_n \m A_m$ is injective and the sequence of cokernels $\{\coker(A_{n} \m A_m)\}$ is ``simpler'' than the sequence $\{A_n\}$. If ``simpler'' implied $H_i(G_m;\coker(A_m \m A_n  ))=0  $, then the two spectral sequences would agree and this would prove homological stability with polynomial coefficients. It is not true that we can assume that $H_i(G_n;\coker(A_n \m A_m  ))$ vanishes, only that it stabilizes. In particular, the spectral sequences of \autoref{E1PageHTdrs} and \autoref{E1PageHTnotRel} do not in fact agree in a range in general. However, if one considers \emph{relative} versions, then they do agree in a range. One is led to considering the spectral sequence associated to \[\cone \Big((C^\G_{n-1} \otimes A_{n-1})_{hG_{n-1}} \m  (C^\G_n \otimes A_n)_{hG_n}   \Big).\]
This has $E^1$-page depicted in \autoref{E1PageHTcone}. If the sequence of representations $A=\{A_n\}$ satisfy the polynomial condition, then one can show in a range that the $E^1$-page can be simplified to that depicted in \autoref{E1PageHTconeSimp}. The $E^1$- and $E^2$-pages of \autoref{E1PageHTconeSimp} are easy to understand, and it is easy to see that the $E^\infty$-page of \autoref{E1PageHTcone} vanishes in a range. At this point we can run the same argument as in the case of constant coefficients, to prove homological stability with polynomial coefficients.

\begin{figure}[h!]    \centering \begin{tikzpicture} {\footnotesize
  \matrix (m) [matrix of math nodes,
    nodes in empty cells,nodes={minimum width=3ex,
    minimum height=5ex,outer sep=2pt},
    column sep=9ex,row sep=5ex, text height=1.5ex, text depth=0.25ex]{ 
     &[-4ex] H_3(G_{n},G_{n-1};A_n,A_{n-1}) &  H_3(G_{n-1},G_{n-2};A_n,A_{n-1})  &  H_3(G_{n-2},G_{n-3};A_n,A_{n-1}) & &  \; \\  
     &[-4ex]  H_2(G_{n},G_{n-1};A_n,A_{n-1}) &  H_2(G_{n-1},G_{n-2};A_n,A_{n-1})  &  H_2(G_{n-2},G_{n-3};A_n,A_{n-1}) & &   \\          
     &[-4ex]  H_1(G_{n},G_{n-1};A_n,A_{n-1}) &  H_1(G_{n-1},G_{n-2};A_n,A_{n-1})  &  H_1(G_{n-2},G_{n-3};A_n,A_{n-1}) & &   \; \\             
      &[-4ex]  H_0(G_{n},G_{n-1};A_n,A_{n-1}) &  H_0(G_{n-1},G_{n-2};A_n,A_{n-1})  &  H_0(G_{n-2},G_{n-3};A_n,A_{n-1}) & & \; \\
&      &     &     & &\\       }; 

 \draw[-stealth, red] (m-1-3.west) -- (m-1-2.east) node [midway,above] {} node [midway,below]{} ;
 \draw[-stealth, red] (m-2-3.west) -- (m-2-2.east) node [midway,above] {} node [midway,below] {};
 \draw[-stealth, red] (m-3-3.west) -- (m-3-2.east) node [midway,above] {} node [midway,below] {};
 \draw[-stealth, red] (m-4-3.west) -- (m-4-2.east) node [midway,above] {} node [midway,below] {};

 \draw[-stealth, red] (m-1-4.west) -- (m-1-3.east) node [midway,above] {} node [midway,below] {};
 \draw[-stealth, red] (m-2-4.west) -- (m-2-3.east) node [midway,above] {} node [midway,below] {};
 \draw[-stealth, red] (m-3-4.west) -- (m-3-3.east) node [midway,above] {} node [midway,below] {};
 \draw[-stealth, red] (m-4-4.west) -- (m-4-3.east) node [midway,above] {} node [midway,below] {}; 

 \draw[-stealth, red] (m-1-5.west) -- (m-1-4.east) node [midway,above] {} node [midway,below] {};
 \draw[-stealth, red] (m-2-5.west) -- (m-2-4.east) node [midway,above] {} node [midway,below] {};
 \draw[-stealth, red] (m-3-5.west) -- (m-3-4.east) node [midway,above] {} node [midway,below] {};
 \draw[-stealth, red] (m-4-5.west) -- (m-4-4.east) node [midway,above] {} node [midway,below] {};

\draw[thick] (m-1-1.north east) -- (m-5-1.east) ;
\draw[thick] (m-5-1.north) -- (m-5-6.north east) ;

}
\end{tikzpicture}
\caption{$E^1$-page associated to $\cone \Big((C^\G_{n-1} \otimes A_{n-1})_{hG_{n-1}} \m  (C^\G_n \otimes A_n)_{hG_n} \Big )$  } \label{E1PageHTcone}
\end{figure}

\begin{figure}[h!]    \centering \begin{tikzpicture} {\footnotesize
  \matrix (m) [matrix of math nodes,
    nodes in empty cells,nodes={minimum width=3ex,
    minimum height=5ex,outer sep=2pt},
    column sep=9ex,row sep=5ex, text height=1.5ex, text depth=0.25ex]{ 
     &[-4ex] H_3(G_{n},G_{n-1};A_n,A_{n-1}) &  H_3(G_{n-1},G_{n-2};A_{n-1},A_{n-2})  &  H_3(G_{n-2},G_{n-3};A_{n-2},A_{n-3}) & &  \; \\  
     &[-4ex]  H_2(G_{n},G_{n-1};A_n,A_{n-1}) &  H_2(G_{n-1},G_{n-2};A_{n-1},A_{n-2})  &  H_2(G_{n-2},G_{n-3};A_{n-2},A_{n-3}) & &   \\          
     &[-4ex]  H_1(G_{n},G_{n-1};A_n,A_{n-1}) &  H_1(G_{n-1},G_{n-2};A_{n-1},A_{n-2})  &  H_1(G_{n-2},G_{n-3};A_{n-2},A_{n-3}) & &   \; \\             
      &[-4ex]  H_0(G_{n},G_{n-1};A_n,A_{n-1}) &  H_0(G_{n-1},G_{n-2};A_{n-1},A_{n-2})  &  H_0(G_{n-2},G_{n-3};A_{n-2},A_{n-3}) & & \; \\
&      &     &     & &\\       }; 

 \draw[-stealth, red] (m-1-3.west) -- (m-1-2.east) node [midway,above] {} node [midway,below]{} ;
 \draw[-stealth, red] (m-2-3.west) -- (m-2-2.east) node [midway,above] {} node [midway,below] {};
 \draw[-stealth, red] (m-3-3.west) -- (m-3-2.east) node [midway,above] {} node [midway,below] {};
 \draw[-stealth, red] (m-4-3.west) -- (m-4-2.east) node [midway,above] {} node [midway,below] {};

 \draw[-stealth, red] (m-1-4.west) -- (m-1-3.east) node [midway,above] {} node [midway,below] {};
 \draw[-stealth, red] (m-2-4.west) -- (m-2-3.east) node [midway,above] {} node [midway,below] {};
 \draw[-stealth, red] (m-3-4.west) -- (m-3-3.east) node [midway,above] {} node [midway,below] {};
 \draw[-stealth, red] (m-4-4.west) -- (m-4-3.east) node [midway,above] {} node [midway,below] {}; 

 \draw[-stealth, red] (m-1-5.west) -- (m-1-4.east) node [midway,above] {} node [midway,below] {};
 \draw[-stealth, red] (m-2-5.west) -- (m-2-4.east) node [midway,above] {} node [midway,below] {};
 \draw[-stealth, red] (m-3-5.west) -- (m-3-4.east) node [midway,above] {} node [midway,below] {};
 \draw[-stealth, red] (m-4-5.west) -- (m-4-4.east) node [midway,above] {} node [midway,below] {};

\draw[thick] (m-1-1.north east) -- (m-5-1.east) ;
\draw[thick] (m-5-1.north) -- (m-5-6.north east) ;

}
\end{tikzpicture}
\caption{Simplified $E^1$-page associated to $\cone \Big((C^\G_{n-1} \otimes A_{n-1})_{hG_{n-1}} \m  (C^\G_n \otimes A_n)_{hG_n} \Big )$  } \label{E1PageHTconeSimp}
\end{figure}

\subsection{Representation stability with untwisted coefficients}

We now change gears and discuss representation stability with untwisted coefficients. 	Consider a family of short exact sequences of groups 
	\[1 \m N_n \m G_n \m Q_n  \m 1.\]   For example, this could be the pure braid groups mapping to the braid groups with symmetric groups as quotients. A way to try to prove representation stability for the $H_i(N_n)$ as $Q_n$-representations is as follows. The chain complex $C^\G_n$ has an action of $N_n$, so one can consider the spectral sequence associated to $(C^\G_n)_{hN_n}$. The $E^1$-page is depicted in \autoref{E1PageRU}.

\begin{figure}[h!]    \centering \begin{tikzpicture} {\footnotesize
  \matrix (m) [matrix of math nodes,
    nodes in empty cells,nodes={minimum width=3ex,
    minimum height=5ex,outer sep=2pt},
    column sep=9ex,row sep=5ex, text height=1.5ex, text depth=0.25ex]{ 
     &[-4ex]   H_3(N_{n}) & \Ind_{Q_{n-1}}^{Q_n}  H_3(N_{n-1})   &   \Ind_{Q_{n-2}}^{Q_n} H_3(N_{n-2}) & \Ind_{Q_{n-3}}^{Q_n}  H_3(N_{n-3})  &  \; \\  
     &[-4ex]   H_2(N_{n}) & \Ind_{Q_{n-1}}^{Q_n}  H_2(N_{n-1})   &   \Ind_{Q_{n-2}}^{Q_n} H_2(N_{n-2}) & \Ind_{Q_{n-3}}^{Q_n}  H_2(N_{n-3})  & \;  \\          
     &[-4ex] H_1(N_{n}) & \Ind_{Q_{n-1}}^{Q_n}  H_1(N_{n-1})   &   \Ind_{Q_{n-2}}^{Q_n} H_1(N_{n-2}) & \Ind_{Q_{n-3}}^{Q_n}  H_1(N_{n-3})  &  \; \\             
      &[-4ex]  H_0(N_{n}) & \Ind_{Q_{n-1}}^{Q_n}  H_0(N_{n-1})   &   \Ind_{Q_{n-2}}^{Q_n} H_0(N_{n-2}) & \Ind_{Q_{n-3}}^{Q_n}  H_0(N_{n-3})  & \; \\
&      &     &     & &\\       }; 

 \draw[-stealth, red] (m-1-3.west) -- (m-1-2.east) node [midway,above] {} node [midway,below]{} ;
 \draw[-stealth, red] (m-2-3.west) -- (m-2-2.east) node [midway,above] {} node [midway,below] {};
 \draw[-stealth, red] (m-3-3.west) -- (m-3-2.east) node [midway,above] {} node [midway,below] {};
 \draw[-stealth, red] (m-4-3.west) -- (m-4-2.east) node [midway,above] {} node [midway,below] {};

 \draw[-stealth, red] (m-1-4.west) -- (m-1-3.east) node [midway,above] {} node [midway,below] {};
 \draw[-stealth, red] (m-2-4.west) -- (m-2-3.east) node [midway,above] {} node [midway,below] {};
 \draw[-stealth, red] (m-3-4.west) -- (m-3-3.east) node [midway,above] {} node [midway,below] {};
 \draw[-stealth, red] (m-4-4.west) -- (m-4-3.east) node [midway,above] {} node [midway,below] {}; 

 \draw[-stealth, red] (m-1-5.west) -- (m-1-4.east) node [midway,above] {} node [midway,below] {};
 \draw[-stealth, red] (m-2-5.west) -- (m-2-4.east) node [midway,above] {} node [midway,below] {};
 \draw[-stealth, red] (m-3-5.west) -- (m-3-4.east) node [midway,above] {} node [midway,below] {};
 \draw[-stealth, red] (m-4-5.west) -- (m-4-4.east) node [midway,above] {} node [midway,below] {};

 \draw[-stealth, red] (m-1-6.west) -- (m-1-5.east);
 \draw[-stealth, red] (m-2-6.west) -- (m-2-5.east);
 \draw[-stealth, red] (m-3-6.west) -- (m-3-5.east);
 \draw[-stealth, red] (m-4-6.west) -- (m-4-5.east);

\draw[thick] (m-1-1.north east) -- (m-5-1.east) ;
\draw[thick] (m-5-1.north) -- (m-5-6.north east) ;

}
\end{tikzpicture}
\caption{$E^1$-page for untwisted representation stability} \label{E1PageRU}
\end{figure}  

%Using that the spectral sequence converges to zero in a range, %it suffices to produce zeros on the $E^2$-page as indicated in %\autoref{E2PageRUzero}.

On the $E^2$-page, the leftmost column is given by \[\coker( \Ind_{Q_{n-1}}^{Q_n} H_i(N_{n-1}) \m H_i(N_{n})). \] Thus, the leftmost column measures the \emph{generators} of $H_i(N):=\{H_i(N_{n})\}$ in the sense of representation stability. Similarly, the second column from the left measures \emph{relations}. 

Let us try to run a similar argument as in the case of homological stability. Thus we assume inductively that all rows below the $k$th row are zero in a range on the $E^2$-page, and we analyze the $k$th row. By the same argument as in the constant coefficient case, we deduce that the entries in the first two columns of the $E^2$-page on the $k$th row are zero for $n$ large enough. Note however that we can not immediately say anything about the third column or beyond, since a priori a nontrivial $E^2$-differential may emanate from these entries. Since the first two columns measure generators and relations, what we can deduce is precisely that $H_k(N)$ is presented in finite degree as a sequence of $Q_n$-representations.  To make this argument work inductively, we need that zeros on the left two columns on the $E^2$-page propagate to the right as $n$ increases. In other words, we need that if a sequence of $Q_n$-representations is presented in finite degree, then it has derived representation stability. This condition on the groups $Q_n$ is what we call \emph{degreewise coherence}. It is currently only known for a few families of groups, such as the symmetric groups.

\subsection{Representation stability with twisted coefficients}

As before, consider a collection of short exact sequences of groups \[1 \m N_n \m G_n \m Q_n  \m 1\] and let $A_n$ be a $G_n$-representation. Now suppose we want to prove $H_i(N_n;A_n)$ has representation stability with respect to the $Q_n$-action. Taking the homotopy orbits $C_n^\G(A)_{hN_n}$ with respect to the $N_n$-action gives a spectral sequence with $E^1$-page as in \autoref{E1PageRTdr}. The $E^1$ and $E^2$-pages can be analyzed exactly as was done in the previous section. The problem is that in order to know that the spectral sequences converges to zero in a range, we need to know derived representation stability for $A=\{A_n\}$, which is hard to check.

\begin{figure}[h!]    \centering \begin{tikzpicture} {\footnotesize
  \matrix (m) [matrix of math nodes,
    nodes in empty cells,nodes={minimum width=3ex,
    minimum height=5ex,outer sep=2pt},
    column sep=9ex,row sep=5ex, text height=1.5ex, text depth=0.25ex]{ 
     &[-4ex]   H_3(N_{n};A_n) & \Ind_{Q_{n-1}}^{Q_n}  H_3(N_{n-1};A_{n-1})   &   \Ind_{Q_{n-2}}^{Q_n} H_3(N_{n-2};A_{n-2}) & \Ind_{Q_{n-3}}^{Q_n}  H_3(N_{n-3};A_{n-3})  &  \; \\  
     &[-4ex]   H_2(N_{n};A_n) & \Ind_{Q_{n-1}}^{Q_n}  H_2(N_{n-1};A_{n-1})   &   \Ind_{Q_{n-2}}^{Q_n} H_2(N_{n-2};A_{n-2}) & \Ind_{Q_{n-3}}^{Q_n}  H_2(N_{n-3};A_{n-3})  & \;  \\          
     &[-4ex] H_1(N_{n};A_n) & \Ind_{Q_{n-1}}^{Q_n}  H_1(N_{n-1};A_{n-1})   &   \Ind_{Q_{n-2}}^{Q_n} H_1(N_{n-2};A_{n-2}) & \Ind_{Q_{n-3}}^{Q_n}  H_1(N_{n-3};A_{n-3})  &  \; \\             
      &[-4ex]  H_0(N_{n};A_n) & \Ind_{Q_{n-1}}^{Q_n}  H_0(N_{n-1};A_{n-1})   &   \Ind_{Q_{n-2}}^{Q_n} H_0(N_{n-2};A_{n-2}) & \Ind_{Q_{n-3}}^{Q_n}  H_0(N_{n-3};A_{n-3})  & \; \\
&      &     &     & &\\       }; 

 \draw[-stealth, red] (m-1-3.west) -- (m-1-2.east) node [midway,above] {} node [midway,below]{} ;
 \draw[-stealth, red] (m-2-3.west) -- (m-2-2.east) node [midway,above] {} node [midway,below] {};
 \draw[-stealth, red] (m-3-3.west) -- (m-3-2.east) node [midway,above] {} node [midway,below] {};
 \draw[-stealth, red] (m-4-3.west) -- (m-4-2.east) node [midway,above] {} node [midway,below] {};

 \draw[-stealth, red] (m-1-4.west) -- (m-1-3.east) node [midway,above] {} node [midway,below] {};
 \draw[-stealth, red] (m-2-4.west) -- (m-2-3.east) node [midway,above] {} node [midway,below] {};
 \draw[-stealth, red] (m-3-4.west) -- (m-3-3.east) node [midway,above] {} node [midway,below] {};
 \draw[-stealth, red] (m-4-4.west) -- (m-4-3.east) node [midway,above] {} node [midway,below] {}; 

 \draw[-stealth, red] (m-1-5.west) -- (m-1-4.east) node [midway,above] {} node [midway,below] {};
 \draw[-stealth, red] (m-2-5.west) -- (m-2-4.east) node [midway,above] {} node [midway,below] {};
 \draw[-stealth, red] (m-3-5.west) -- (m-3-4.east) node [midway,above] {} node [midway,below] {};
 \draw[-stealth, red] (m-4-5.west) -- (m-4-4.east) node [midway,above] {} node [midway,below] {};

 \draw[-stealth, red] (m-1-6.west) -- (m-1-5.east);
 \draw[-stealth, red] (m-2-6.west) -- (m-2-5.east);
 \draw[-stealth, red] (m-3-6.west) -- (m-3-5.east);
 \draw[-stealth, red] (m-4-6.west) -- (m-4-5.east);

\draw[thick] (m-1-1.north east) -- (m-5-1.east) ;
\draw[thick] (m-5-1.north) -- (m-5-6.north east) ;

}
\end{tikzpicture}
\caption{$E^1$-page for twisted homological stability (assuming derived representation stability)} \label{E1PageRTdr}
\end{figure}

As before, it is reasonable to expect that polynomiality should be a sufficient replacement for derived representation stability. If we had considered the action of $N_n$ on $C^\G_n \otimes A_n$ instead of on $C^\G_n(A_n)$, then we would have obtained a spectral sequence similar to \autoref{E1PageRTdr} except with $H_q(N_p;A_p)$ replaced with $H_q(N_p;A_n)$ on the $E^1$-page. This spectral sequence does converge to zero, but has an undesirable $E^1$-page. As before, we cannot immediately use the polynomial condition to replace $H_i(N_p;A_n)$ with $H_i(N_p;A_p)$. In the case of untwisted coefficients, the spectral sequence we used came from considering \[\cone \Big((C^\G_{n-1} \otimes A_{n-1})_{hG_{n-1}} \m  (C^\G_n \otimes A_n)_{hG_n}   \Big).\] One might guess that we just need to replace $G$'s with $N$'s and instead consider \[\cone \Big((C^\G_{n-1} \otimes A_{n-1})_{hN_{n-1}} \m  (C^\G_n \otimes A_n)_{hN_n}   \Big).\] But as far as we can tell, this is not useful.

The mapping cone works very well to measure homological stability, and this was useful when we were trying to prove homological stability. However, when we are trying to prove representation stability, we should replace ``$\cone$'' with something that measures representation stability, like central stability chains.  

We remark, however, that \[\cone \Big((C^\G_{n-1} \otimes A_{n-1})_{hN_{n-1}} \m  (C^\G_n \otimes A_n)_{hN_n}   \Big)\] is naturally a triple complex with one of the three directions concentrated in only two degrees. When we replace ``$\cone$'' with central stability chains, we obtain something that is a triple complex in a more essential way. Triple complexes are hard to study. To simplify things, let us assume each $N_n=1$ (and hence $Q_n=G_n$).  With untwisted coefficients, there is nothing interesting about $H_*(1)$. However, $H_0(1;A_n)=A_n$, so representation stability for $1$ with coefficients in $A$ is the same as representation stability for $A$. Once we set $N_n=1$, the triple complex become a double complex. This spectral sequence converges to zero as before. The polynomial condition lets us simplify the $E^1$-page in a range. This is the reasoning that led us to study the spectral sequences used in \autoref{SecPolyn} of this paper. The upshot is that this spectral sequence lets us prove that polynomial coefficient systems have derived representation stability. Once we have establish this we can go back and reconsider the spectral sequence of \autoref{E1PageRTdr}. The only reason why that spectral sequence did not seem useful was we did not know if it converged to zero. With the polynomiality implies derived representation stability result, we do know that it converges to zero and thus we can profit off of its simpler $E^1$-page. This is what we do in \autoref{Sec4}.

\bibliographystyle{amsalpha}
\bibliography{Polynomial}

\end{document}